\documentclass[a4paper,oneside,reqno]{amsart}

\calclayout

\usepackage{etoolbox}

\newtoggle{default}\toggletrue{default}
\newtoggle{birk}\togglefalse{birk}
\newtoggle{birk_t2}\togglefalse{birk_t2}
\newtoggle{siam}\togglefalse{siam}
\newtoggle{springer}\togglefalse{springer}

\newtoggle{endwithsymbol}\togglefalse{endwithsymbol}

\toggletrue{default}
\toggletrue{endwithsymbol}

\usepackage[utf8]{inputenc}
\usepackage[english]{babel}
\usepackage{microtype} 
\usepackage{lmodern}
\usepackage{amsmath,amssymb,amsthm}
\usepackage{thmtools} 

\usepackage{hyperref} 

\hypersetup{%
  colorlinks=true,
  linkcolor=blue,
  citecolor=green!60!black,
  hypertexnames=false,
}

\usepackage[sort]{cleveref} 
\usepackage{orcidlink}

\usepackage{mathtools} 
\usepackage{upgreek}
\usepackage{accents} 
\usepackage{stmaryrd} 

\usepackage{tikz} 
\usetikzlibrary{cd}
\usetikzlibrary{calc}

\usepackage{dsfont} 
\usepackage{pifont} 

\usepackage{ifthen}
\usepackage{enumitem}
\setlist[itemize]{labelindent=\parindent,leftmargin=*,itemsep=3pt,topsep=5pt}
\setlist[enumerate]{label=\textup{(\roman{*})},itemsep=3pt,labelindent=0pt,leftmargin=*}

\usepackage{nicematrix}

\usepackage{graphicx} 
\usepackage{subcaption}
\graphicspath{{../../figs/}{../../pics/}{./}}

\usepackage{booktabs} 
\usepackage{esint} 

\usepackage{bbm}

\usepackage{color}

\usepackage{pgfpages} 
\usepackage{pgfplots} 
\pgfplotsset{compat=1.13} 

\definecolor{NathanaelColor}{rgb}{0.2,0.0,0.6}

\newcommand{\argdot}{\boldsymbol{\cdot}}
\newcommand{\restrict}{\mkern-6mu\upharpoonright}

\newcommand{\R}{\mathbb{R}} 
\newcommand{\N}{\mathbb{N}} 
\newcommand{\Z}{\mathbb{Z}} 
\newcommand{\C}{\mathbb{C}} 

\newcommand{\Cc}[1][\infty]{\mathrm{C}_{\mathrm{c}}\ifthenelse{\equal{#1}{}}{}{^{#1}}}
\newcommand{\Lp}[2][]{\mathrm{L}_{#2\ifthenelse{\equal{#1}{}}{}{,#1}}} 
\newcommand{\Lb}{\mathcal{L}_{\mathrm{b}}} 
\newcommand{\sobH}{\mathrm{H}}
\newcommand{\cH}{\accentset{\circ}{\sobH}}

\newcommand{\indicator}{\mathds{1}} 
\newcommand{\iu}{\mathrm{i}} 
\newcommand{\euler}{\mathrm{e}} 
\newcommand{\dd}{\mathrm{d}} 
\newcommand{\dx}[1][x]{\,\dd#1}

\newcommand{\transposed}{^{\mathsf{T}}}

\newcommand{\m}{\mathrm{m}} 
\newcommand{\mbound}{K} 

\DeclareMathOperator{\ran}{ran}

\DeclareMathOperator{\dom}{dom}

\renewcommand{\div}{\operatorname{div}}
\DeclareMathOperator{\grad}{grad}
\DeclareMathOperator{\curl}{curl}
\newcommand{\divcon}{\mathop{\div_{-1}}}

\newcommand{\cgrad}{\mathop{\accentset{\circ}{\grad}}}
\newcommand{\ccurl}{\mathop{\accentset{\circ}{\curl}}}
\newcommand{\cdiv}{\mathop{\accentset{\circ}{\div}}}

\DeclarePairedDelimiter{\set}{\lbrace}{\rbrace}
\DeclarePairedDelimiter{\norm}{\lVert}{\rVert}
\DeclarePairedDelimiter{\abs}{\vert}{\vert}

\DeclarePairedDelimiterX{\dset}[2]{\{}{\}}{#1\,\delimsize\vert\,\mathopen{} #2}
\DeclarePairedDelimiterX{\scprod}[2]{\langle}{\rangle}{#1,#2}

\DeclarePairedDelimiter{\jump}{\llbracket}{\rrbracket}

\renewcommand{\Re}{\operatorname{Re}}
\renewcommand{\Im}{\operatorname{Im}}

\newcommand{\weakto}{\rightharpoonup}

\newcommand{\Htopo}{\mathrm{H}}
\newcommand{\Gtopo}{\mathrm{G}}

\newcommand{\Qm}[1]{Q_m\left[#1\right]}
\newcommand{\Qmr}[1]{Q_m\left[#1\right]_\rho}

\newcommand{\ord}[1]{\mathcal{O}\left(#1\right)}
\newcommand{\eps}{\varepsilon}

\newcommand{\pmtrx}[1]{\ensuremath{\begin{pmatrix}#1 \end{pmatrix}}}

\newcommand{\vecsymb}[1]{\boldsymbol{#1}}

\theoremstyle{plain}
\newtheorem{theorem}{Theorem}[section]
\newtheorem{lemma}[theorem]{Lemma}

\newcommand{\myqedhere}{\qedhere}

    \declaretheorem[style=definition,sibling=theorem,qed=\ding{169}]{definition}
    \declaretheorem[style=definition,sibling=theorem,qed=\ding{169}]{example}

    \declaretheorem[style=remark,sibling=theorem,qed=\ding{169}]{remark}
\begin{document}

\title{Homogenisation for Maxwell and Friends} 

\author[A.~Buchinger]{Andreas Buchinger\,\orcidlink{0009-0004-4203-5874}}

\address{TU Bergakademie Freiberg \\
  Institute of Applied Analysis \\
  Akademiestrasse 6 \\
  D-09596 Freiberg \\
  Germany}

\email{andreas.buchinger@math.tu-freiberg.de}

\author[S.~Franz]{Sebastian Franz\,\orcidlink{0000-0002-2458-1597}}

\address{Technische Universit\"at Dresden \\
  Institute of Scientific Computing \\
  Zellescher Weg 25 \\
  D-01217 Dresden \\
  Germany}
\email{sebastian.franz@tu-dresden.de}

\author[N.~Skrepek]{Nathanael Skrepek\,\orcidlink{0000-0002-3096-4818}}

\address{University of Twente \\
  Department of Applied Mathematics \\
  University of Twente P.O.\ Box 217 \\
  7500 AE Enschede \\
  The Netherlands}
\email{n.skrepek@utwente.nl}

\author[M.~Waurick]{Marcus Waurick\,\orcidlink{0000-0003-4498-3574}}

\address{TU Bergakademie Freiberg \\
  Institute of Applied Analysis \\
  Akademiestrasse 6 \\
  D-09596 Freiberg \\
  Germany}

\email{marcus.waurick@math.tu-freiberg.de}

\date{\today}



\ifboolexpr{togl{birk_t2} or togl{birk}}{%
}{%
}%


\begin{abstract}
  We refine the understanding of continuous dependence on coefficients of solution operators under the nonlocal $\Htopo$-topology viz Schur topology in the setting of evolutionary equations in the sense of Picard. We show that certain components of the solution operators converge strongly. The weak convergence behaviour known from homogenisation problems for ordinary differential equations is recovered on the other solution operator components. The results are underpinned by a rich class of examples that, in turn, are also treated numerically, suggesting a certain sharpness of the theoretical findings. Analytic treatment of an example that proves this sharpness is provided too. Even though all the considered examples contain local coefficients, the main theorems and structural insights are of operator-theoretic nature and, thus, also applicable to nonlocal coefficients. The main advantage of the problem class considered is that they contain mixtures of type, potentially highly oscillating between different types of PDEs; a prototype can be found in Maxwell's equations highly oscillating between the classical equations and corresponding eddy current approximations.
\end{abstract}

\ifboolexpr{togl{default} or togl{birk_t2} or togl{birk}}{\maketitle}{}%

\section{Introduction}\label{sec:intro}

The theory of homogenisation addresses the effective behaviour of solutions of certain differential equations with highly oscillatory coefficients. The mathematical theory roots in the late 60s of the 20th century, and a standard account of the theory can be found in the seminal monographs \cite{BLP78,ZKO94,Ta09}; with a more elementary introduction in \cite{CiDo99}. With a focus on elliptic differential equations in variational form, a general viewpoint has led to the development of the notions of G- and $\Htopo$-convergence, see, e.g., \cite{MuTa97,Ta09,CiDo99}. In these notions, the main object of study were sequences of certain matrix-valued multiplication operators. In order to understand homogenisation problems for evolutionary equations in the sense of Picard, a class of abstract operator equations providing a unified set-up for many time-dependent (partial) differential equations of mathematical physics, see \cite{Pi09}, a more operator-theoretic perspective needed to be advanced. We refer to \cite[Chapters~13 and~14]{SeTrWa22} and the references therein to get a picture of the first ten years of research concerning homogenisation theory for evolutionary equations. In the course of understanding this realm of questions in \cite{Wa18}, the notion of nonlocal $\Htopo$-convergence has been developed, which is a nonlocal generalisation of classical $\Htopo$-convergence, allowing for general operator coefficients. In turn, this led to the development of the Schur topology, see \cite{NiWa22,Wa22,BuSkWa24}, with the main result in \cite{BuSkWa24} as culmination point, establishing continuous dependence results for evolutionary equations if the operator coefficients are endowed with a holomorphic variant of the Schur topology. Note that \cite{BuSkWa24} together with the compactness statement in \cite{BuErWa24} (which in turn is a particular perspective given by the more general framework for Friedrichs systems in \cite{BuErWa23}) provides a good understanding of homogenisation problems for evolutionary equations.

Even though the theoretical understanding is well-developed, a decent list of rather involved examples illustrating the theory is still missing. Thus, the first aim of the present article is to fill this gap. Moreover, the fundamental difference of assumptions in \cite{BuSkWa24} compared to the $\mathrm{G}$-compactness statement in \cite{BuErWa24} is a certain compactness condition, which is needed in the former and rather irrelevant for the latter. Hence, a second insight gathered here is that the compactness condition actually improves weak convergence to strong convergence on parts of the solution operator. Analytic treatment of one example will yield a decomposition into two infinite-dimensional subspaces such that one part converges strongly and the other is known to only converge weakly. Thirdly, we shall underpin our theoretical findings by treating all the examples also numerically. These numerical experiments additionally highlight our theoretical findings. The question, whether the part with strong convergence that we obtain is maximal, is an avenue open for future research.

We emphasise that the list of examples range from ordinary differential equations to higher dimensional partial differential equations, where the latter class also contains mixed type equations with highly oscillatory coefficients. In particular, an example for Maxwell's equations is treated, which, after the homogenisation process, triggers a memory effect in the limit equation, which is classical and can also be found in \cite{We01,Wa16a}. The list of examples stresses the versatility of the concept of evolutionary equations and the applicability of the main convergence result in \cite{BuSkWa24}. The Schur topology provides the precise setting enabling us to compute the effective evolutionary equations in a systematic manner. The numerics are based on \cite{FrTrWa19} developed precisely for mixed type problems written in the form of evolutionary equations.

We quickly summarise the organisation of the paper. In \Cref{sec:EvolEq}, we recall the general setting of evolutionary equations together with its corner stone Picard's \Cref{thm:wpee}. \Cref{sec:htt} serves to present a round up of results of homogenisation theory needed here. In particular, we recall the notion of nonlocal $\Htopo$-convergence and provide the main convergence result of \cite{BuSkWa24} together with the refinement concerning the quality of convergence. In \Cref{sec:APHC} we state and prove some additional results concerning homogenisation theory that either have been overlooked so far or have only been announced in the literature. The list of examples can be found in \Cref{sec:examples}; the corresponding numerical study is provided in \Cref{sec:numerics}. \Cref{sec:con} contains a small conclusion. Some supplementary material can be found in \Cref{sec:app}.

The Hilbert spaces that we will consider are anti-linear in the first and linear in the second argument.
For a Hilbert space $\mathcal{H}$, a bounded linear operator $A\in\Lb(\mathcal{H})$, and $c>0$, we will write $\Re A\geq c$ instead of
\begin{equation*}
\forall h\in\mathcal{H}:\Re \scprod{h}{Ah}_{\mathcal{H}} \geq c \scprod{h}{h}_{\mathcal{H}}=c\norm{h}^{2}_{\mathcal{H}}\text{,}
\end{equation*}
and $\Re A >0$ if we deem the exact knowledge of $c>0$ unimportant. If we write $A\colon\dom(A)\subseteq\mathcal{H}_1\to\mathcal{H}_2$ for Hilbert spaces $\mathcal{H}_1$, $\mathcal{H}_2$, then
$A$ stands for a possibly unbounded operator with domain $\dom(A)$ and the adjoint $A^{\ast}$ is considered with respect to $\mathcal{H}_1$ and $\mathcal{H}_2$.

\section{The Setting: Evolutionary Equations}\label{sec:EvolEq}

In this section, we recall the basic setting of evolutionary equations in the sense of Picard, \cite{Pi09}. All results presented in the current section can be found with complete proofs in~\cite{SeTrWa22}. Let $\mathcal{H}$ be a Hilbert space, $\nu>0$ and we set
\begin{equation*}
  \Lp[\nu]{2}(\R;\mathcal{H}) \coloneqq \dset*{f\in \Lp[\mathrm{loc}]{1}(\R;\mathcal{H})}{\int_\R \norm{f(t)}_{\mathcal{H}}^{2} \euler^{-2\nu t} \dx[t] <\infty}.
\end{equation*}
This is a Hilbert space endowed with the obvious scalar product. The Sobolev space of weakly differentiable functions $f\in \Lp[\nu]{2}(\R;\mathcal{H})$ with distributional derivative $f' \in \Lp[\nu]{2}(\R;\mathcal{H})$ is denoted by $\sobH^1_\nu(\R;\mathcal{H})$. Next, we define
\begin{align*}
  \partial_{t,\nu} \colon\left\{
  \begin{array}{rcl}
    \sobH^{1}_{\nu}(\R;\mathcal{H}) \subseteq \Lp[\nu]{2}(\R;\mathcal{H}) &\to& \Lp[\nu]{2}(\R;\mathcal{H}) \\
    f &\mapsto& f'.
  \end{array}
  \right.
\end{align*}
The Fourier--Laplace transformation $\mathcal{L}_{\nu} \in \Lb(\Lp[\nu]{2}(\R;\mathcal{H}),\Lp{2}(\R;\mathcal{H}))$ is the unitary extension of the mapping satisfying
\begin{equation*}
  (\mathcal{L}_{\nu} f)(\xi) \coloneqq \frac{1}{\sqrt{2\uppi}} \int_{\R} \euler^{-\iu \xi t-\nu t} f(t) \dx[t]
  \quad(\xi \in \R, f\in \Cc[](\R;\mathcal{H})).
\end{equation*}
This unitary transformation provides the spectral representation for $\partial_{t,\nu}$. Indeed, defining the multiplication-by-argument operator
\begin{align*}
  \m\colon\left\{
  \begin{array}{rcl}
    \dset[\big]{f\in \Lp{2}(\R;H)}
    {(\xi \mapsto \xi f(\xi))\in \Lp{2}(\R;\mathcal{H})} \subseteq \Lp{2}(\R;\mathcal{H}) &\to& \Lp{2}(\R;\mathcal{H}), \\
    f&\mapsto& (\xi\mapsto \xi f(\xi)),
  \end{array}
  \right.
\end{align*}
we obtain the equality
\begin{equation*}
  \partial_{t,\nu} = \mathcal{L}_\nu^* (\iu \m +\nu)\mathcal{L}_\nu.
\end{equation*}
This equality serves as  a means  to define holomorphic functions of $\partial_{t,\nu}$. For this we introduce the following notion.

\begin{definition}
  Let $M\colon \dom(M)\subseteq \C \to \Lb(\mathcal{H})$. We call $M$ a \emph{material law}, if
  \begin{enumerate}
    \item $\dom(M)$ is open, $M$ is holomorphic and
    \item\label{item:material-law-bounded} there exists $\nu\in \R$ such that $\C_{\Re>\nu}\coloneqq \dset{z\in\C}{\Re z>\nu} \subseteq \dom(M)$ and
          \begin{equation*}
          \norm{M}_{\infty,\nu} \coloneqq \sup_{z\in \C_{\Re>\nu}} \norm{M(z)} < \infty.
          \end{equation*}
  \end{enumerate}
  We set $s_b(M)\coloneqq \inf \dset{\nu\in \R}{\ref{item:material-law-bounded} \ \text{holds}}$ and call it the \emph{abscissa of boundedness of} $M$. The set of all material laws with abscissa of boundedness lower than some $\mu\in \R$ is denoted by $\mathcal{M}(\mathcal{H},\mu)$.

  For a material law $M\in \mathcal{M}(\mathcal{H},\nu)$, we furthermore define the corresponding \emph{material law operator}  $M(\partial_{t,\nu}) \in \Lb(\Lp[\nu]{2}(\R;\mathcal{H}))$ by
  \begin{equation*}
    M(\partial_{t,\nu}) f \coloneqq \mathcal{L}_\nu^{\ast} M(\iu \m+\nu) \mathcal{L}_\nu f
    \quad(f\in \Lp[\nu]{2}(\R;\mathcal{H})),
  \end{equation*}
  where
  \begin{equation*}
    (M(\iu \m + \nu)\phi )(t) \coloneqq M(\iu t + \nu)\phi(t)
    \quad(\phi\in \Lp{2}(\R;\mathcal{H}), \text{a.e.}\ t\in \R).\myqedhere
  \end{equation*}
\end{definition}

Next, we present the fundamental theorem for evolutionary equations, Picard's well-posedness theorem. For this, we do not use a different notation for a skew-selfadjoint operator acting on $\mathcal{H}$ and its (canonical) skew-selfadjoint extension to $\Lp[\nu]{2}(\R;\mathcal{H})$. It will always be clear from the context, which operator is considered.

\begin{theorem}[{{Picard's Theorem, \cite[Theorem 6.2.1]{SeTrWa22}}}]\label{thm:wpee}
  Let $A\colon \dom(A)\subseteq \mathcal{H}\to \mathcal{H}$ be skew-selfadjoint, $\nu\in \R$, $M\in \mathcal{M}(\mathcal{H},\nu)$. Assume there exists $c>0$ such that
  \begin{equation*}
    \Re zM(z)\geq c
    \quad \text{for all }z\in \C_{\Re\geq\nu}.
  \end{equation*}
  Then, the operator
  \begin{align*}
    \mathcal{B}_\nu \colon \left\{
    \begin{array}{rcl}
      \sobH^{1}_{\nu} (\R;\mathcal{H}) \cap \Lp[\nu]{2}(\R;\dom(A)) \subseteq \Lp[\nu]{2}(\R;\mathcal{H}) &\to& \Lp[\nu]{2}(\R;\mathcal{H}), \\
      U &\mapsto&  [\partial_{t,\nu}M(\partial_{t,\nu})+A] U,
    \end{array}
    \right.
  \end{align*}
  is closable. The closure is continuously invertible, $\mathcal{S}_\nu\coloneqq \overline{\mathcal{B}_\nu}^{-1}$. Moreover, $\norm{\mathcal{S}_{\nu}} \leq 1/c$.
\end{theorem}

\begin{remark}
  \begin{enumerate}
    \item\label{item:consequences-of-picard-thm} There are additional consequences of the assumptions in Picard's Theorem that we will not need but are worth mentioning. For instance, $\mathcal{S}_\nu$ is independent of the particular choice of $\nu$. If $\mu\geq\nu$ and $f\in \Lp[\nu]{2}(\R;\mathcal{H})\cap \Lp[\mu]{2}(\R;\mathcal{H})$ then $\mathcal{S}_\nu f = \mathcal{S}_\mu f$. Also, the solution operator $\mathcal{S}_\nu$ is \emph{causal}; that is, for all $f,g\in \Lp[\nu]{2}(\R;\mathcal{H})$ and $a\in \R$, $f=g$ on $\lparen -\infty,a \rbrack$ implies $\mathcal{S}_\nu f = \mathcal{S}_\nu g$ $\lparen -\infty,a \rbrack$. Furthermore, we have the following regularity statement: if $f\in \sobH^1_\nu(\R;\mathcal{H})$, then $\mathcal{S}_\nu f\in \sobH^{1}_{\nu}(\R;H)\cap \Lp[\nu]{2}(\R;\dom(A))$.

    \item\label{item:picard-cont-inv-mat-law} The positive-definiteness condition on the material law implies that, for all $z\in \C_{\Re\geq\nu}$, the operator $(zM(z)+A)$ is continuously invertible on $\mathcal{H}$ with the norm of the inverse bounded by $1/c$.
          In fact, it is only this uniformly bounded invertibility that is used to obtain the numerous conclusions. Thus, instead of the positive definiteness condition, it suffices to assume that there exists $c>0$ such that
          \begin{equation*}
            \sup_{z\in \C_{\Re\geq\nu}} \norm{(zM(z)+A)^{-1}} \leq 1/c.
          \end{equation*}
          Then, the same conclusions as in Picard's Theorem (also the ones mentioned in \cref{item:consequences-of-picard-thm}) hold.
    \item As mentioned in \cref{item:picard-cont-inv-mat-law}, the proof of Picard's Theorem hinges upon having $(zM(z)+A)^{-1}\in\Lb(\mathcal{H})$
          for $z\in \C_{\Re\geq\nu}$. In fact, one shows that this defines a material law with corresponding operator $\mathcal{S}_{\nu}$.\myqedhere
  \end{enumerate}
\end{remark}

For $d\in\N$ and an open $\Omega\subseteq\R^d$, we define $\grad\restrict_{\Cc}\colon\Cc (\Omega)\subseteq\Lp{2} (\Omega)
\to\Lp{2}(\Omega)^d$ and
$\div\restrict_{\Cc}\colon\Cc(\Omega)^d\subseteq\Lp{2}(\Omega)^d\to\Lp{2}(\Omega)$
the usual way. We also get $\curl\restrict_{\Cc}\colon\Cc(\Omega)^3\subseteq\Lp{2}(\Omega)^3\to\Lp{2}(\Omega)^3$. As usual, we can weakly extend these
unbounded operators to their maximal domains, i.e., $\grad\coloneqq-(\div\restrict_{\Cc})^\ast$, $\div\coloneqq-(\grad\restrict_{\Cc})^\ast$ and
$\curl\coloneqq(\curl\restrict_{\Cc})^\ast$. Adjoining again, we get
the operators with zero boundary conditions $\cgrad\coloneqq-\div^\ast=\overline{\grad\restrict_{\Cc}}$,
$\cdiv\coloneqq-\grad^\ast=\overline{\div\restrict_{\Cc}}$ and $\ccurl\coloneqq\curl^\ast=\overline{\curl\restrict_{\Cc}}$. Since all these extended
operators are closed, their domains are Hilbert spaces if endowed with the respective graph product. Most prominently, we get the Sobolev spaces
$\sobH^1(\Omega)$ and $\cH^1(\Omega)$ stemming from the domains of $\grad$ and $\cgrad$ respectively, and similarly, $\sobH(\div,\Omega)$, $\cH(\div,\Omega)$,
$\sobH(\curl,\Omega)$ and $\cH(\curl,\Omega)$.

\begin{example}[Maxwell's Equations]\label{ex:MaxwellAsEvolEq}
For $\Omega\subseteq\R^3$ open, $\mathcal{H}\coloneqq \Lp{2}(\Omega)^3\times \Lp{2}(\Omega)^3$,
$c>0$ and $\nu_0>0$, let the bounded and measurable dielectric permittivity,
magnetic permeability and electric conductivity $\epsilon,\mu,\sigma\colon\Omega \to \R^{3\times 3}$
satisfy
\begin{equation*}
\forall x\in\Omega: \mu(x)=\mu(x)^\ast\geq c\text{ and } \nu\epsilon(x)+\Re\sigma(x)=
\nu\epsilon(x)^\ast+\Re\sigma(x)\geq c
\end{equation*}
for all $\nu\geq \nu_0$. Maxwell's equations for the electric and magnetic field, $E$ and $H$, with current $j_0\in\dom(\partial_{t,\nu})$ and perfect conductor
boundary conditions read
\begin{equation*}
\left[\partial_{t,\nu}\begin{pmatrix} \epsilon&0\\ 0&\mu\end{pmatrix}+\begin{pmatrix} \sigma&0\\ 0&0\end{pmatrix}
+\begin{pmatrix} 0&-\curl \\ \ccurl &0\end{pmatrix}\right]\begin{pmatrix} E\\ H\end{pmatrix}
=\begin{pmatrix} j_0\\ 0\end{pmatrix}\text{.}
\end{equation*}
Note that all conditions of \Cref{thm:wpee} are met. We emphasise that due to positive parameter $\nu>0$, we ask for an implicit homogeneous initial condition at $-\infty$. Classical initial value problems can also be treated,
see \cite[Chapter 9]{SeTrWa22}.
\end{example}

\section{Homogenisation Theory and Theorems}\label{sec:htt}

In this section, we shall add some more structural and less computational points to this subject matter. In particular, we will recall the classical (local) homogenisation theory
from~\cite{MuTa97} and~\cite{Ta09}, we will introduce the notion of nonlocal $\Htopo$-convergence (and, thus, the Schur topology)
from \cite{Wa18} (generalised in~\cite{Wa22}), and we will refine the (nonlocal) homogenisation result for evolutionary equations from \cite{BuSkWa24}.

In \Cref{sec:EvolEq}, we have introduced $\div$ weakly on its maximal domain in $\Lp{2}(\Omega)^{d}$, where $\Omega\subseteq\R^{d}$ open and $d\in\N$.
We will extend this operator canonically to $\divcon\colon \Lp{2}(\Omega)^d\to \sobH^{-1}(\Omega)$, where $\sobH^{-1}(\Omega)$ is the dual space of $\cH^1(\Omega)$, via
$\divcon(\varphi)(u)\coloneqq -\scprod{\varphi}{\cgrad u}_{\Lp{2}(\Omega)^{d}}$ for all $\varphi\in \Lp{2}(\Omega)^{d}$ and $u\in\cH^1(\Omega)$. Clearly, $\divcon$ is bounded and anti-linear.

\subsection{Local $\Htopo$-Convergence}\label{subsec:localHtopo}

Classical PDE-homogenisation theory is tailored to the sequence of unique solutions $(u_n)_n$ in $\cH^1(\Omega),n\in\N,$ of
\begin{equation}\label{PDElocalHconvergSequence}
  -\divcon a_{n} \cgrad u_{n} = f
\end{equation}
where, for some $d\in\N$, $\Omega\subseteq\R^d$ is bounded and open, $f\in\sobH^{-1}(\Omega)$, and $(a_{n})_{n\in\N}$ is a given sequence of linear bounded operators
on $\Lp{2}(\Omega)^d$ with $\Re a_n\geq c$ for all $n\in\N$ and some $c>0$.
In case that $(u_n)_{n\in\N}$ at least weakly converges to some $u\in \cH^1(\Omega)$, one would like to obtain an $a\in\Lb(\Lp{2}(\Omega)^d)$ with $\Re a\geq \tilde{c}$ for some $\tilde{c}>0$ such that $u$ solves
\begin{equation}\label{PDElocalHconvergSolution}
  -\divcon a \cgrad u = f\text{,}
\end{equation}
and one would like this $a$ to be independent of $f$.
At this point, we have to note two things. Firstly, this convergence is apparently more a condition on the coefficient sequence $(a_{n})_{n\in\N}$ than on the solutions (which
 clearly depend on $f$). Secondly,
it turns out that neither existence nor uniqueness of a limit in terms of $(a_{n})_{n\in\N}$
is guaranteed in this general setting.\footnote{Actually, existence follows by \Cref{thm:HConvergenceHausdorffCompactMetrisable} if we restrict the coefficient sequence
  to the set~\labelcref{eq:definSetMAlphaBetaOmega} of multiplication operators. This leads to the older concept of $\Gtopo$-convergence ($\Htopo$-convergence implies $\Gtopo$-convergence)
  introduced by Sergio Spagnolo in the 1960s. Alternatively, \Cref{thm:nonlocal-H-convergence-comes-from-a-topology} below can be viewed as replacing the restriction to~\labelcref{eq:definSetMAlphaBetaOmega} by requiring mere boundedness of the sequence $(a_n)_{n\in\N}$, i.e., a nonlocal generalisation. Uniqueness is only achievable in special cases since the solution $u$ of~\labelcref{PDElocalHconvergSolution} disregards constant real skew-selfadjoint additive perturbations of $a$ (cf.~\Cref{le:dirich-grad-vanishes-for-skew-selfad}).
For an in-depth discussion see, e.g.,~\cite[Chapter~6]{Ta09}.}
In the special case of multiplication (and thus local) operators
$a_n,a\in \Lp{\infty}(\Omega)^{d\times d}, n\in\N$, \cite{MuTa97} and~\cite{Ta09} now provide the following theory.

\begin{definition}[{\cite[Definition~6.4]{Ta09}}]
Consider the set
\begin{equation}\label{eq:definSetMAlphaBetaOmega}
  M(\alpha,\beta,\Omega)\coloneqq \dset[\big]{a \in \Lp{\infty}(\Omega)^{d\times d}}{\Re a(x)\geq\alpha,\Re a(x)^{-1}\geq \tfrac{1}{\beta} \text{ for }x\in\Omega\text{ a.e.}}
\end{equation}
for some $d\in\N$, $\Omega\subseteq\R^d$ bounded and open, and for some $0<\alpha<\beta$.

We call a  sequence $(a_n)_{n\in\N}$ from
 $M(\alpha,\beta,\Omega)$ \emph{locally $\Htopo$-convergent} to an $a\in M(\alpha,\beta,\Omega)$ if and only if for all $f\in\sobH^{-1}(\Omega)$
 the sequence $(u_n)_{n\in\N}$ of $\cH^1$-solutions of~\labelcref{PDElocalHconvergSequence} weakly converges to the
 solution $u\in\cH^1(\Omega)$ of~\labelcref{PDElocalHconvergSolution} in $\cH^{1}(\Omega)$, and $a_n\cgrad u_n$ weakly converges to $a\cgrad u$ in $\Lp{2}(\Omega)^d$, i.e.,
 \begin{equation*}
   u_{n} \weakto u \mspace{10mu} \text{in} \ \cH^{1}(\Omega)
   \quad\text{and}\quad
   a_{n} \cgrad u_{n} \weakto a \cgrad u \mspace{10mu} \text{in} \ \Lp{2}(\Omega)^{d}. \myqedhere
 \end{equation*}
\end{definition}

\begin{theorem}[{\cite[p.~82 (metric topology) and Theorem~6.5 (compact)]{Ta09}}]\label{thm:HConvergenceHausdorffCompactMetrisable}
For $d\in\N$, $\Omega\subseteq\R^d$ bounded and open, and $0<\alpha<\beta$, there exists a topology $\tau$ on $M(\alpha,\beta,\Omega)$ that induces local $\Htopo$-convergence and renders the space
$(M(\alpha,\beta,\Omega),\tau)$ Hausdorff, compact, and even metrisable.
\end{theorem}
In certain cases, we can compute the homogenisation limit explicitly.
\begin{theorem}[{\cite[Theorem~5.12]{CiDo99}}]\label{thm:ExplicitLocalHLimitForStratMedia}
  Let $\Omega \subseteq \R^{d}$ be open and bounded for some $d\in\N$ and $a = (a_{ij})_{i,j=1}^{d} \in M(\alpha,\beta,\Omega)$ given by bounded and measurable
  $\ell_{1}$-periodic functions
  \begin{equation*}
    \hat{a}_{ij} \colon \R \to \R,
  \end{equation*}
  where $a_{ij} \colon \Omega \to \R$ and $a_{ij}(x) = \hat{a}_{ij}(x_{1})$ for $x=(x_{1},\dots,x_{d}) \in \Omega$.
  For $n\in\N$, we define $a_{n}(x) \coloneqq a(nx)$ and the corresponding sequence $(a_{n})_{n\in\N}$. Then for every $f \in \sobH^{-1}(\Omega)$, the solution sequence $(u_{n})_{n\in\N}$ to the problems
  \begin{equation}\label{eq:div-grad-system-in-CiDo99-theorem}
      -\divcon a_{n} \cgrad u_{n} = f
  \end{equation}
  converges weakly to $u_{\hom}$ in $\cH^{1}(\Omega)$, where $u_{\hom}$ is the solution of
  \begin{equation*}
    -\divcon a_{\hom} \cgrad u_{\hom}= f\text{,}
  \end{equation*}
  and the matrix-valued function $a_{\hom} \in M(\alpha,\beta,\Omega)$ is given by
  \begin{alignat*}{3}
    {(a_{\hom})}_{11} &= \frac{1}{\mathrm{m}(\frac{1}{\hat{a}_{11}})}, &&\\
    {(a_{\hom})}_{1j} &= {(a_{\hom})}_{11} \mathrm{m}\Bigl(\frac{\hat{a}_{1j}}{\hat{a}_{11}}\Bigr), && 2 \leq j \leq d \\
    {(a_{\hom})}_{i1} &= {(a_{\hom})}_{11} \mathrm{m}\Bigl(\frac{\hat{a}_{i1}}{\hat{a}_{11}}\Bigr), && 2 \leq i \leq d \\
    {(a_{\hom})}_{ij} &= {(a_{\hom})}_{11} \mathrm{m}\Bigl(\frac{\hat{a}_{i1}}{\hat{a}_{11}}\Bigr) \mathrm{m}\Bigl(\frac{\hat{a}_{1j}}{\hat{a}_{11}}\Bigr) + \mathrm{m}\Bigl(\hat{a}_{ij} - \frac{\hat{a}_{i1}\hat{a}_{1j}}{\hat{a}_{11}}\Bigr), &\quad& 2 \leq i,j \leq d,
  \end{alignat*}
  where $\mathrm{m}(b)$ is the integral mean over $b\in\Lp{1}(0,\ell_{1})$, i.e., $\mathrm{m}(b) \coloneqq \frac{1}{\ell_{1}}\int_{0}^{\ell_{1}} b(z) \dx[z]$.
  Additionally, the sequence $(a_{n} \cgrad u_{n})_{n\in\N}$ converges weakly to $a_{\hom} \cgrad u_{\hom}$ in $\Lp{2}(\Omega)^{d}$, i.e.,
  \begin{equation}\label{eq:hidden-extra-assertion}
    a_{n} \cgrad u_{n} \weakto a_{\hom} \cgrad u_{\hom} \quad\text{in}\quad \Lp{2}(\Omega)^{d}.
  \end{equation}
  All in all, this means $(a_{n})_{n\in\N}$ locally $\Htopo$-converges to $a_{\hom}$.
\end{theorem}

\begin{remark}
  Note the following:
  \begin{itemize}
    \item The last claim~\eqref{eq:hidden-extra-assertion} is not part of the original statement, but it is shown in the course of proof, see \cite[Proof of Thm.~5.10]{CiDo99}.

    \item Furthermore, \cite{CiDo99} defines $M(\alpha,\beta,\Omega)$ without the real part operator, since only real vector spaces are treated. However, if $a$ is an $\R^{d\times d}$-valued function, then both approaches to $M(\alpha,\beta,\Omega)$ coincide (theirs and ours), see \Cref{le:equivalent-M-for-real-matrices}.

    \item The previous result is the periodic case of the more general \cite[Lemma 5.3]{Ta09},
     where the sequence $(a_{n})_{n\in\N}$ depends on the first variable only, and the entries $(a_{n})_{ij}$ satisfy certain $\Lp{\infty}$-$\Lp{1}$-weak-$\ast$ convergence conditions (see \cite[(5.6)--(5.9)]{Ta09}) that, due to \Cref{thm:WeakStarLimitPerMultOp}, exactly yield the above limit in the periodic case.\myqedhere%
  \end{itemize}
\end{remark}

\subsection{Nonlocal $\Htopo$-Convergence}\label{subsectionNonlocalHConv}

In~\cite{Wa18}, the desire to treat general (possibly nonlocal) coefficients instead of only multiplication operators, and to treat a whole class of (possibly time-dependent) systems at
once led to the following operator-theoretic approach to homogenisation.

For the rest of this section, let $\mathcal{H}$ be a separable Hilbert space that can be decomposed orthogonally into two closed subspaces
$\mathcal{H}=\mathcal{H}_{0}\oplus \mathcal{H}_{1}$. Hence, any $a\in\Lb(\mathcal{H})$ can equivalently be written as $\begin{pmatrix}a_{00}&a_{01}\\
a_{10}&a_{11}\end{pmatrix}$ where $a_{ij}\in\Lb(\mathcal{H}_j,\mathcal{H}_i)$.

\begin{definition}[{cf.~\cite[Theorem~4.1]{Wa18}}]\label{def:nonlocal-H-convergence}
  Consider
  \begin{equation*}
    \mathcal{M}(\mathcal{H}_0,\mathcal{H}_1)\coloneqq \dset{a\in\Lb(\mathcal{H})}{a_{00}^{-1}\in\Lb(\mathcal{H}_0)\text{ and }a^{-1}\in\Lb(\mathcal{H})}\text{.}
  \end{equation*}
  A sequence $(a_n)_{n\in\N}$ from $\mathcal{M}(\mathcal{H}_0,\mathcal{H}_1)$ is said to \emph{nonlocally $\Htopo$-converge} to an $a\in\mathcal{M}(\mathcal{H}_0,\mathcal{H}_1)$
  if and only if $a_{n,00}^{-1}$, $a_{n,10}a_{n,00}^{-1}$, $a_{n,00}^{-1}a_{n,01}$, and $a_{n,11}-a_{n,10}a_{n,00}^{-1}a_{n,01}$ converge pointwise weakly
  (i.e., with respect to the weak operator topology) to $a_{00}^{-1}$, $a_{10}a_{00}^{-1}$, $a_{00}^{-1}a_{01}$, and $a_{11}-a_{10}a_{00}^{-1}a_{01}$ respectively.
\end{definition}

\begin{theorem}[{cf.~\cite[Chapter~5]{Wa18}}]\label{thm:nonlocal-H-convergence-comes-from-a-topology}
  There exists a topology $\tau(\mathcal{H}_0,\mathcal{H}_1)$ on $\mathcal{M}(\mathcal{H}_0,\mathcal{H}_1)$, called \emph{nonlocal $\Htopo$-topology}
  or \emph{Schur topology}, that induces nonlocal $\Htopo$-convergence and renders the space $(\mathcal{M}(\mathcal{H}_0,\mathcal{H}_1),\tau(\mathcal{H}_0,\mathcal{H}_1))$
  Hausdorff. Furthermore, regard
  \begin{equation*}
    \mathcal{M}(\gamma, \mathcal{H}_0,\mathcal{H}_1)
    \coloneqq
    \begin{aligned}[t]
      \left\{\vphantom{\frac{1}{\gamma_{1}}}a \in \mathcal{M}(\mathcal{H}_{0},\mathcal{H}_{1})
      \,\right|\,&
                   \Re a_{00} \geq \gamma_{00}, \Re a_{00}^{-1} \geq \frac{1}{\gamma_{11}}, \\
                 &\quad\norm{a_{10} a_{00}^{-1}} \leq \gamma_{10}, \norm{a_{00}^{-1}a_{01}} \leq \gamma_{01}, \\
                 &\qquad\Re (a_{11} - a_{10}a_{00}^{-1}a_{01})^{-1} \geq \frac{1}{\gamma_{11}}, \\
                 &{\left.\vphantom{\frac{1}{\gamma_{1}}}\quad\qquad\mspace{5mu}\Re (a_{11} - a_{10}a_{00}^{-1}a_{01}) \geq \gamma_{00}
                   \right\}}\text{,}
    \end{aligned}
  \end{equation*}
  where $\gamma=\begin{pmatrix}\gamma_{00}&\gamma_{01}\\\gamma_{10}&\gamma_{11}\end{pmatrix}\in (0,\infty)^{2\times 2}$.
  Then, $(\mathcal{M}(\gamma,\mathcal{H}_0,\mathcal{H}_1),\tau(\mathcal{H}_0,\mathcal{H}_1))$ is compact and metrisable.
\end{theorem}

In the following, if the decomposition of $\mathcal{H}$ into $\mathcal{H}_0$ and $\mathcal{H}_1$ is clear from the context, we will simply write $  \mathcal{M}(\gamma)$ instead of $  \mathcal{M}(\gamma, \mathcal{H}_0,\mathcal{H}_1)$.
As one would expect, nonlocal $\Htopo$-convergence reasonably generalises local $\Htopo$-convergence. For the matter of simplicity, we will only treat topologically trivial
domains (in 3D: simply connected with connected complement). For the general case, see~\cite{Wa22}.

\begin{remark}[{Helmholtz decomposition, \cite[Examples~2.3 and~2.4]{Wa18}}]\label{re:HelmhDecompTopoTrivial}
For a bounded simply connected weak Lipschitz domain $\Omega\subseteq\R^3$ with connected complement and continuous boundary, we have the following orthogonal decompositions into closed subspaces
\begin{equation*}
\Lp{2}(\Omega)^3=\ran(\cgrad)\oplus \ran(\curl)= \ran(\grad)\oplus \ran(\ccurl)\text{,}
\end{equation*}
where $\Omega^{\complement}$ being connected implies the first decomposition, and $\Omega$ being simply connected the second one.
In other words, we have $\ran(\grad)=\ker(\curl)$, $\ran(\cgrad)=\ker(\ccurl)$, $\ker(\div)=\ran(\curl)$ and $\ker(\cdiv)=\ran(\ccurl)$.
Furthermore, the Picard-Weber-Weck selection theorem holds (\cite{Pi84}): $\dom(\div)\cap \dom(\ccurl)$ and $\dom(\cdiv)\cap\dom(\curl)$,
both endowed with the inner product $\scprod{\argdot}{\argdot}_{\Lp{2}(\Omega)^3}+\scprod{\div\argdot}{\div\argdot}_{\Lp{2}(\Omega)}+\scprod{\curl\argdot}{\curl\argdot}_{\Lp{2}(\Omega)^{3}}$,
are each compactly embedded into $\Lp{2}(\Omega)^3$.
\end{remark}

\begin{theorem}[{\cite[Theorem~5.11]{Wa18}}]\label{thm:H-converges-equivalent-to-Schur-convergence}
Let $\Omega\subseteq\R^3$ be a bounded weak Lipschitz domain with connected complement. For $0<\alpha<\beta$, a  sequence $(a_n)_{n\in\N}$ from
 $M(\alpha,\beta,\Omega)$ locally $\Htopo$-converges to some $a\in M(\alpha,\beta,\Omega)$ if and only if it
 $\tau(\ran(\cgrad),\ran(\curl))$-converges to $a$.
\end{theorem}

Note that the previous result can be extended to $\Omega \subseteq \R^{d}$ for $d > 3$. The key is to find the corresponding Helmholtz/Hodge decomposition, see, e.g., \cite[Theorem 3.6 \& Theorem 2.2]{Wa18b}.

\subsection{Homogenisation Theory for Evolutionary Equations}

In the remaining section, assume that $\mathcal{H}$ is a separable Hilbert space and $A\colon\dom(A)\subseteq\mathcal{H}\to\mathcal{H}$ skew-selfadjoint. Additionally,
let the Hilbert space $\dom(A)\cap (\ker A)^{\perp}$ (with respect to $\scprod{\argdot}{\argdot}_{\mathcal{H}}+\scprod{A\argdot}{A\argdot}_{\mathcal{H}}$) be compactly embedded into $\mathcal{H}$.
This implies $\ker(A)^{\perp}=\ran A$ (see, e.g.,~\cite[Lemma~4.1]{ElGoWa19} or the FA-Toolbox in \cite{PaZu23}), i.e.,
$\mathcal{H}=\mathcal{H}_{0}\oplus \mathcal{H}_{1}$ in the sense of \Cref{subsectionNonlocalHConv} with the Hilbert spaces $\mathcal{H}_{0}\coloneqq\ker(A)$ and $\mathcal{H}_{1}\coloneqq\ran(A)$.
Looking at the proof of~\cite[Lemma~6.4]{BuSkWa24}, we see that actually the following stronger statement was shown.
\begin{lemma}\label{lemma-WOTSOT-Theorie-Paper}
  Let $\gamma\in (0,\infty)^{2\times 2}$ and consider a sequence $(T_{n})_{n\in\N}$ from $\mathcal{M}(\gamma)$ that $\tau(\ker(A),\ran(A))$-converges to some $T\in\mathcal{M}(\gamma)$. Then,
  $(T_{n}+A)^{-1},(T+A)^{-1}\in\Lb (\mathcal{H})$, for all $n\in \N$, and for $\varphi\in\mathcal{H}$,
  \begin{itemize}
          \item the $\ker(A)$-component of $(T_{n}+A)^{-1}\varphi$ weakly converges to $(T+A)^{-1}\varphi$,
          \item the $\ran(A)$-component of $(T_{n}+A)^{-1}\varphi$ converges to $(T+A)^{-1}\varphi$.
  \end{itemize}
\end{lemma}
Apparently, by (canonically) extending the orthogonal projections from $\mathcal{H}$ to $\Lp[\nu]{2}(\R;\mathcal{H})$, the decomposition of $\mathcal{H}$ can be lifted to
\begin{equation*}
  \Lp[\nu]{2}(\R;\mathcal{H})=\Lp[\nu]{2}(\R;\ker(A))\oplus\Lp[\nu]{2}(\R;\ran(A))\text{,}
\end{equation*}
where $\nu\in\R$ is arbitrary.

Employing this and \Cref{lemma-WOTSOT-Theorie-Paper}, we immediately obtain the following refinement of the homogenisation theorem~\cite[Theorem~6.5]{BuSkWa24}:
\begin{theorem}\label{thm:main-nonlocal-homogenisation-result}
  Consider $\nu_{0}>0$ and material laws $M_{n}$ with $\C_{\Re>\nu_{0}}\subseteq\dom(M_{n})$ for $n\in\N$. Let there be some $c,d>0$ such that for all $z\in\C_{\Re>\nu_{0}}$ we have $\Re zM_{n}(z)\geq c$ and $\norm{M_{n}(z)}\leq d$.
  Then, $M_{n}\in\mathcal{M}(\mathcal{H},\nu)$ for all $n\in\N$ and $\nu>\nu_0$. Assume there exists  $M\colon\C_{\Re>\nu_{0}}\to\mathcal{M}(\ker(A),\ran(A))$
  with $\norm{M(z)}\leq d$ on $\C_{\Re>\nu_{0}}$. If $(M_{n})_{n\in\N}$ pointwise $\tau(\ker(A),\ran(A))$-converges to $M$ on $\C_{\Re>\nu_{0}}$,
  then  $M\in \mathcal{M}(\mathcal{H},\nu)$ for all $\nu>\nu_{0}$ as well as $\Re zM(z)\geq c$ on $\C_{\Re>\nu_{0}}$. Most importantly, for
  $f\in \Lp[\nu]{2}(\R,\mathcal{H})$, $\nu>\nu_{0}$, $n\in\N$, $\mathcal{S}_{n,\nu}\coloneqq\overline{[\partial_{t,\nu}M_{n}(\partial_{t,\nu})+A]}^{-1}$
  and $\mathcal{S}_{\nu}\coloneqq\overline{[\partial_{t,\nu}M(\partial_{t,\nu})+A]}^{-1}$,
    \begin{itemize}
          \item the $\Lp[\nu]{2}(\R;\ker(A))$-component of $\mathcal{S}_{n,\nu}f$ weakly converges to $\mathcal{S}_{\nu}f$,
          \item the $\Lp[\nu]{2}(\R;\ran(A))$-component of $\mathcal{S}_{n,\nu}f$ converges to $\mathcal{S}_{\nu}f$.
  \end{itemize}
\end{theorem}

The previous theorem will be a crucial tool to justify the convergence of the examples in \Cref{sec:examples}. In fact, by means of one analytic and several numerical examples we shall see that the convergence statements seem to be optimal.
\begin{remark}
Consider the setting of \Cref{thm:main-nonlocal-homogenisation-result}.
Whether $\Lp[\nu]{2}(\R;\ran(A))$ is a maximal (disregarding the trivial finite-dimensional extensions) subspace, on which $(\mathcal{S}_{n,\nu}f)_{n\in\N}$ converges to $\mathcal{S}_{\nu}f$ in norm for each
$f\in \Lp[\nu]{2}(\R,\mathcal{H})$, remains an open question.
\end{remark}

\section{Additional Properties of (Non-)Local $\Htopo$-Convergence}\label{sec:APHC}

In this section, we will prove a few, mostly straightforward, properties of local and nonlocal  $\Htopo$-convergence that we will need in \Cref{sec:examples}, and that only implicitly appear in the literature.

\subsection{Nonlocal $\Htopo$-Convergence of the Inverse Sequence}
A frequent scenario one is confronted with (and that also appears in \Cref{sec:examples}) is knowing how to compute
 the $\tau(\mathcal{H}_1,\mathcal{H}_0)$-limit but rather needing the $\tau(\mathcal{H}_0,\mathcal{H}_1)$-limit. The following theorem offers a solution: invert the sequence, compute the limit, and then invert back. This observation originated in \cite{Wa22}.
\begin{theorem}\label{thm:nonlocal-H-convergence-symm-wrt-inv}
For a sequence $(a_n)_{n\in\N}$ in $\mathcal{M}(\mathcal{H}_0,\mathcal{H}_1)$, the sequence of inverses $((a_n)^{-1})_{n\in\N}$ lies in $\mathcal{M}(\mathcal{H}_1,\mathcal{H}_0)$.
Furthermore, $\tau(\mathcal{H}_0,\mathcal{H}_1)$-convergence of $(a_n)_{n\in\N}$ to an $a\in \mathcal{M}(\mathcal{H}_0,\mathcal{H}_1)$ is equivalent to $\tau(\mathcal{H}_1,\mathcal{H}_0)$-convergence
of $((a_n)^{-1})_{n\in\N}$ to $a^{-1}$.
\end{theorem}
\begin{proof}
By~\cite[Lemma~4.8]{Wa18}, we obtain $(a_n)^{-1}\in\mathcal{M}(\mathcal{H}_1,\mathcal{H}_0)$ with the decomposition
\begin{equation*}
(a_n)^{-1}=\begin{pmatrix}a_{n,00}^{-1}+a_{n,00}^{-1}a_{n,01}\tilde{a}_na_{n,10}a_{n,00}^{-1}&-a_{n,00}^{-1}a_{n,01}\tilde{a}_n\\-\tilde{a}_n a_{n,10}a_{n,00}^{-1}&\tilde{a}_n\end{pmatrix}
\end{equation*}
for $n\in\N$, where $\tilde{a}_{n}\coloneqq (a_{n,11}-a_{n,10}a_{n,00}^{-1}a_{n,01})^{-1}$. The definition of the respective Schur topologies now directly implies the statement.
\end{proof}

\subsection{A Pasting Theorem}

In \Cref{sec:examples}, we will also need a kind of a pasting property of local $\Htopo$-convergence. That means we want to conclude local $\Htopo$-convergence
on a domain from local $\Htopo$-convergence on each of two disjoint subdomains that together span the domain in a certain topological and measure theoretical sense.
This is basically an immediate corollary of~\cite[Lemma~10.5]{Ta09}.

\begin{theorem}\label{thm:localsums}
Consider $d\in\N$, $\Omega\subseteq\R^d$ bounded and open, $0<\alpha<\beta$, a sequence $(a_n)_{n\in\N}$ in $M(\alpha,\beta,\Omega)$, and some $a\in M(\alpha,\beta,\Omega)$.
If for almost every $x\in \Omega$ there exists an open neighborhood $\omega_x\subseteq\Omega$ of $x$ such that $(a_n\restrict_{\omega_x})_{n\in\N}$ (locally) $\Htopo$-converges to $a\restrict_{\omega_x}$ in $M(\alpha,\beta,\omega_x)$, then
$(a_n)_{n\in\N}$ (locally) $\Htopo$-converges to $a$ in $M(\alpha,\beta,\Omega)$.
\end{theorem}

\begin{proof}
  By \Cref{thm:HConvergenceHausdorffCompactMetrisable}, every subsequence of $(a_n)_{n\in\N}$ has another subsequence that (locally) $\Htopo$-converges to some $b\in M(\alpha,\beta,\Omega)$.
  From \cite[Lemma~10.5]{Ta09} and from the assumptions, we infer that $a$ and $b$ coincide. As the initial subsequence was chosen arbitrarily, this concludes the proof.
\end{proof}

\subsection{Independence of Boundary Conditions}
Another property that we will need in \Cref{sec:examples} is a variant of \Cref{thm:H-converges-equivalent-to-Schur-convergence}
with different boundary conditions that was already proposed in~\cite[Remark~5.12]{Wa18}, see also~\cite[Theorem~2.4]{BuSkWa24}. It hinges upon
the independence of the homogeneous Dirichlet boundary conditions that
were imposed in the definition of local $\Htopo$-convergence (cf.\ \Cref{subsec:localHtopo}). To be precise, the following concrete realisation of~\cite[Lemma~10.3]{Ta09} holds.
\begin{lemma}
Consider $d\in\N$, $\Omega\subseteq\R^d$ bounded, open and connected, and $0<\alpha<\beta$. Then, for a sequence $(a_n)_{n\in\N}$ in
 $M(\alpha,\beta,\Omega)$ and $a\in M(\alpha,\beta,\Omega)$, the following statements are equivalent:
\begin{enumerate}
 \item\label{it:localHconvDBC} $(a_n)_{n\in\N}$ locally $\Htopo$-converges to $a$.
 \item\label{it:localHconvNeuBC} For all $f \in \sobH^{1}_{\perp}(\Omega)'$, where  $\sobH^{1}_{\perp}(\Omega)\coloneqq\dset{v \in \sobH^{1}(\Omega)}{\int_{\Omega}v = 0}$ (endowed with the $\sobH^{1}(\Omega)$-scalar product),
 the unique solutions $u_{n}\in \sobH^{1}_{\perp}(\Omega),n\in\N$, of
\begin{equation}\label{eq:localHomogenisationNeumannBDSeq}
\forall v\in \sobH^{1}_{\perp}(\Omega) \colon\scprod*{a_{n}\grad u_{n}}{\grad v}_{\Lp{2}(\Omega)^d}=f(v)
\end{equation}
weakly $\sobH^{1}_{\perp}(\Omega)$-converge to the unique solution $u\in \sobH^{1}_{\perp}(\Omega)$ of
\begin{equation}\label{eq:localHomogenisationNeumannBDLimit}
\forall v\in \sobH^{1}_{\perp}(\Omega)\colon\scprod*{a\grad u}{\grad v}_{\Lp{2}(\Omega)^d}=f(v)\text{,}
\end{equation}
and $a_n\grad u_n$ weakly $\Lp{2}(\Omega)^d$-converges to $a\grad u$.
\end{enumerate}
\end{lemma}
\begin{proof}
  First, the same construction that eventually proves metrisability in \Cref{thm:HConvergenceHausdorffCompactMetrisable} also shows that the convergence in~\labelcref{it:localHconvNeuBC} is induced by a Hausdorff topology
on $M(\alpha,\beta,\Omega)$.

  Next, assume~\labelcref{it:localHconvDBC}, and consider the solution sequence $(u_{n})_{n\in\N}$
  from~\labelcref{eq:localHomogenisationNeumannBDSeq} for some fixed $f$. Extract any subsequence, and use the same notation.
  By classical elliptic theory, this sequence is bounded in $\sobH^{1}(\Omega)$
  by the norm of $f$ times a constant that only depends on $\alpha$ and the shape of $\Omega$. Thus, the sequence $(a_{n}\grad u_{n})_{n\in\N}$ is bounded
  in $\Lp{2}(\Omega)^d$ by the same bound times $1/\beta$, and we obtain yet another subsequence, $u\in\sobH^{1}(\Omega)$, as well as $q\in \Lp{2}(\Omega)^d$
  such that $u_n\weakto u$ in $\sobH^{1}(\Omega)$, which also implies $u\in\sobH^{1}_{\perp}(\Omega)$, as well as $a_{n}\grad u_{n}\weakto q$ in $\Lp{2}(\Omega)^d$. For $w\in \cH^{1}(\Omega)$, we can calculate
\begin{equation*}
\scprod*{a_{n}\grad u_{n}}{\grad w}_{\Lp{2}(\Omega)^d}=f(w-w_{\Omega})\text{,}
\end{equation*}
where $w_{\Omega}\coloneqq \abs{\Omega}^{-1}\int_{\Omega} w \dx$ is the integral mean. Straightforward computations show $(w\mapsto f(w-w_{\Omega}))\in \sobH^{-1}(\Omega)$.
Therefore, we can apply \cite[Lemma~10.3]{Ta09} which yields $q=a\grad u$. It remains to verify~\labelcref{eq:localHomogenisationNeumannBDLimit}, but this is an immediate consequence of
$a_{n}\grad u_{n}\weakto a\grad u$ in $\Lp{2}(\Omega)^d$.

So, the identity mapping is continuous from the compact local $\Htopo$-topology to the Hausdorff topology that induces the convergence in~\labelcref{it:localHconvNeuBC}. Hence, these two topologies on
$M(\alpha,\beta,\Omega)$ coincide.
\end{proof}

With this,~\cite[Remark~5.12]{Wa18} yields:

\begin{theorem}\label{thm:H-converges-equivalent-to-Schur-convergence-diff-BD}
Let $\Omega\subseteq\R^3$ be a simply connected bounded weak Lipschitz domain. For $0<\alpha<\beta$, a  sequence $(a_n)_{n\in\N}$ from
 $M(\alpha,\beta,\Omega)$ locally $\Htopo$-converges to an $a\in M(\alpha,\beta,\Omega)$ if and only if it
 $\tau(\ran(\grad),\ran(\ccurl))$-converges to $a$.
\end{theorem}

\subsection{$\Htopo$-convergence in 2D}
Finally, we will apply \Cref{thm:H-converges-equivalent-to-Schur-convergence} in two dimensions in \Cref{sec:examples}. Fortunately, this easily turns out to work exactly as expected and is in fact very similar to the 3D case.

\begin{remark}[Helmholtz Decomposition in 2D]\label{re:HelmholtzDecom2D}
Let $\Omega\subseteq \R^2$ be open and simply connected with the segment property. Define $J\coloneqq\begin{psmallmatrix}0&-1\\ 1&0\end{psmallmatrix}$. Then, it is known that\footnote{Recall that the ranges are closed: Indeed, boundedness and the segment property of $\Omega$ yield that the Rellich--Kondrachov selection theorem holds (both $\sobH^{1}(\Omega)\hookrightarrow \Lp{2}(\Omega)$ and $\cH^{1}(\Omega)\hookrightarrow \Lp{2}(\Omega)$ are compact embeddings). Thus, the claim follows by a standard argument for $\grad$ and $\cgrad$,  see, e.g., again the FA-Toolbox in \cite{PaZu23}, and the fact that $J$ is a topological isomorphism.}
\begin{equation}\label{eq:HD2D}
\Lp{2}(\Omega)^2=\ran(\cgrad)\oplus \ran(J\grad)= \ran(\grad)\oplus \ran(J\cgrad).
\end{equation}
Since this statement seems to be difficult to find (or follows from more involved statements involving Betti numbers and differential forms, see, e.g., \cite{DS52,Pi79}), we provide a short argument here.

Clearly, we have (see also  \cite[Appendix]{Pa15})
\begin{equation*}
\Lp{2}(\Omega)^2=\ran(\cgrad)\oplus \ran(J\grad)\oplus \mathcal{H}_{\mathrm{D}}= \ran(\grad)\oplus \ran(J\cgrad)\oplus  \mathcal{H}_{\mathrm{N}},
  \end{equation*}
  where
  \begin{equation*}
      \mathcal{H}_{\mathrm{D}} \coloneqq \ker(\div)\cap \ker(\cdiv (-J)),\text{ and }      \mathcal{H}_{\mathrm{N}} \coloneqq \ker(\cdiv)\cap \ker(\div (-J))
  \end{equation*}
  are the harmonic Dirichlet and Neumann fields in two spatial dimensions. Note that
  \begin{equation*}
     \dim \mathcal{H}_{\mathrm{D}} = \dim \mathcal{H}_{\mathrm{N}}:
  \end{equation*}
 Indeed, for $q \in \Lp{2}(\Omega)^2$, we have $q\in  \mathcal{H}_{\mathrm{D}}$ if and only if $Jq \in  \mathcal{H}_{\mathrm{N}}$. Thus, in order to prove \eqref{eq:HD2D} it remains to show $\mathcal{H}_{\mathrm{N}}=\{0\}$. For this, let $f\in\mathcal{H}_{\mathrm{N}}$. Note that without loss of generality, we may assume that $f$ attains values in $\R$ only. Then $g (x+\iu y)\coloneqq f_1(x,y)-\iu f_2(x,y)$ for all $(x,y)\in \Omega \subseteq \R^2 \cong \C$ defines a complex-valued function with $\Re g=f_1$ and $\Im g = -f_2$. It is not difficult to see that, locally as $f\in \mathcal{H}_{\mathrm{N}}$, $g$ satisfies the Cauchy--Riemann equations in a distributional sense. In particular, $\Re g=f_1$ and $\Im g=-f_2$ are harmonic distributions and, by Weyl's lemma, $f_1,f_2\in \mathrm{C}^\infty(\Omega)$. In consequence, $g$ is holomorphic. Since $\Omega$ is simply connected, by  Cauchy's integral theorem, there exists a potential $G$ such that $G'=g$. Then, consider $F(x,y)\coloneqq \Re G(x+\iu y)$ and compute for $(x,y)\in \Omega$
 \begin{align*}
     \partial_x F(x,y) &= \Re G'(x+\iu y) = \Re g(x+\iu y) = f_1(x,y) \text{ and }\\
     \partial_y F(x,y) &= \Re \iu G'(x+\iu y) = -\Im g(x+\iu y) = f_2(x,y).
 \end{align*} Next, as $\Omega$ has the segment property, \cite[Theorem 3.2(2)]{Ag10} applies to an approximating sequence using the shift technique similar to the one considered in \cite[p.~170]{Pi82}, and since $f\in L_2(\Omega)^2$, we infer that $F\in \dom(\grad)$. By assumption, $\grad F = f\in \ker(\cdiv)$. In particular,
 \[
    \langle \grad F, \grad F \rangle = -    \langle  F,\cdiv \grad F \rangle = 0,
 \]which implies that $F$ is constant and, hence, $f=0$ as desired.

 As a consequence of the above decomposition, we also get $\ran(\grad)=\ker(\div J)$, $\ran(\cgrad)=\ker(\cdiv J)$, $\ker(\div)=\ran(J\grad)$ and $\ker(\cdiv)=\ran(J\cgrad)$.
\end{remark}

With this, we get yet another variant of \Cref{thm:H-converges-equivalent-to-Schur-convergence}.

\begin{theorem}\label{thm:H-converges-equivalent-to-Schur-convergence-2D}
Let $\Omega\subseteq\R^2$ be open, bounded and simply connected satisfying the segment property. For $0<\alpha<\beta$, a  sequence $(a_n)_{n\in\N}$ from
 $M(\alpha,\beta,\Omega)$ locally $\Htopo$-converges to an $a\in M(\alpha,\beta,\Omega)$ if and only if it
 $\tau(\ran(\cgrad),\ran(J\grad))$-converges to $a$.
\end{theorem}

\begin{proof}
The short proof of \cite[Theorem~5.11]{Wa18} only needs $\R^3$ to use the classical Helmholtz decomposition. The underlying \cite[Theorem~4.10]{Wa18} works on an abstract Hilbert space level.
Hence, using \Cref{re:HelmholtzDecom2D}, we can readily rewrite \cite[Theorem~5.11]{Wa18} for $\R^2$.
\end{proof}

\section{Examples}\label{sec:examples}
In this section, we will present a range of homogenisation examples (and their respective limits) that fall into the regime of \Cref{thm:main-nonlocal-homogenisation-result}. We focus on the pre-asymptotic homogenisation problems first and provide the respective limits after that.
All of them can be written in the form
\begin{equation}\label{eq:modelSeqEvolEqForHomogen}
  \bigl(\partial_{t,\nu}M_{0,n} + M_{1,n} + A\bigr) U = F\text{,}
\end{equation}
for $M_{0,n},M_{1,n}\in \Lb(\mathcal{H}),n\in\mathbb{N},$ and suitable $\nu>0$.
In other words, the material laws are of the form $M_n(z)\coloneqq M_{0,n} + z^{-1}M_{1,n}$ for $n\in\N$ and $z\in\C\setminus \set{0}$
which readily implies $M_n\in \mathcal{M}(\mathcal{H},\nu)$ for $n\in\N$ and for all $\nu>0$.
These examples will illustrate different convergence phenomena corresponding to the respective different decompositions $\mathcal{H} = \mathcal{H}_{0} \oplus \mathcal{H}_{1}=\ker(A)\oplus\ran(A)$.

We state the examples next.

\subsection{The Examples}
One of the standard situations of \Cref{thm:main-nonlocal-homogenisation-result} is the case when $A=0$, i.e., an ODE, in which case one can choose $\mathcal{H}_{1} = \set{0}$ and $\mathcal{H}_{0} = \mathcal{H}=\ker(A)$.
This means only weak convergence of the solutions can be expected.

\begin{example}[An ordinary differential equation]\label{ex:ode}
Consider
  \begin{equation}\label{eq:SineParODEtobeHomogenized}
    \partial_{t,\nu} u_n(t,x) + \sin\left(2\uppi n x\right)u_n(t,x) = f(t,x),
  \end{equation}
  where $\Omega\coloneqq (0,1)$, $\mathcal{H}\coloneqq\Lp{2}(\Omega)$, and $\nu>2$.

  Clearly, we have
  \begin{equation*}
  \forall h\in\Lp{2}(\Omega): \Re \scprod{h}{zh}_{\Lp{2}(\Omega)}+\Re \scprod{h}{\sin(2\uppi n \argdot)h}_{\Lp{2}(\Omega)}\geq \norm{h}^{2}_{\Lp{2}(\Omega)}\text{,}
  \end{equation*}
  and
  \begin{equation*}
  \forall h\in\Lp{2}(\Omega): \norm{h+z^{-1}\sin(2\uppi n \argdot)h}_{\Lp{2}(\Omega)}\leq 2\norm{h}_{\Lp{2}(\Omega)}
  \end{equation*}
  for $\Re z> 2$ and $n\in\mathbb{N}$.

  Thus, we are in the setting of \Cref{thm:main-nonlocal-homogenisation-result} with $\nu_0\coloneqq 2$. In particular,~\labelcref{eq:SineParODEtobeHomogenized}
  is well-posed (in the sense of \Cref{thm:wpee} with $A\coloneqq 0$) for each $n\in\mathbb{N}$, $\nu>\nu_0$, and $f\in\Lp[\nu]{2}(\R;\mathcal{H})$, and by~\cite{Wa16a}
  \begin{equation}\label{eq:explicitSolutionSeqODEexample}
    u_n(t,x)=\int_{-\infty}^t \euler^{-(t-s)\sin(2\uppi nx)}f(s,x) \dx[s]
  \end{equation}
  gives the solution sequence explicitly.
\end{example}

The other extreme case is $\mathcal{H}_{0} = \set{0}$, meaning $A$ is one-to-one and the Hilbert space $\dom(A)$ compactly embeds into $\mathcal{H}_{1} =\mathcal{H}=\ran(A)$. Hence, we obtain strong convergence
of the solutions. In fact, a finite-dimensional kernel is already sufficient for strong convergence on the whole space since the weak and strong operator topology coincide on finite-dimensional spaces.

In the following we model oscillations with the indicator function $\indicator_{O_{n}}\colon \R \to \set{0,1}$ of the set
\begin{equation*}
  O_{n} \coloneqq \bigcup_{k \in \Z} \bigl\lparen \tfrac{2k}{2n}, \tfrac{2k+1}{2n}\bigr\rparen,n\in\N,
\end{equation*}
see \Cref{fig:indicator}. Note that $\indicator_{O_{n}}(x) = \indicator_{O_{1}}(nx)$ for $x\in\R$ and $n\in\N$.

\begin{example}[Finite-dimensional kernel]\label{ex:1d-pde-with-compactness}
  This example appears in \cite{FrWa18}. Let $n\in\N$ and set
  \begin{equation*}
    \epsilon_{n}(x) \coloneqq \indicator_{O_{n}}(x),\quad
    \sigma_{n}(x) \coloneqq 1 - \indicator_{O_{n}}(x)\text{.}
  \end{equation*}
  We consider the following rough-coefficient PDE\footnote{For $\mathrm{C}^{\infty}_{\#}(0,1)\coloneqq\{\varphi\restrict_{(0,1)}:\varphi\in\mathrm{C}^{\infty}(\R),\varphi(\argdot+1)=\varphi(\argdot)\}$, we define
  $\partial_{x}\restrict_{\mathrm{C}^{\infty}_{\#}}\colon \mathrm{C}^{\infty}_{\#}(0,1)\subseteq\Lp{2}(0,1)\to \Lp{2}(0,1)$ as the usual derivative. Weakly extending, we obtain
  $\partial_{\#}\coloneqq-(\partial_{x}\restrict_{\mathrm{C}^{\infty}_{\#}})^{\ast}$ and $\partial_{\#}=-(\partial_{\#})^{\ast}=\overline{\partial_{x}\restrict_{\mathrm{C}^{\infty}_{\#}}}$ with the domain
  as Hilbert space $\sobH_{\#}^1(0,1)$.}
  \begin{equation}\label{eq:RoughCoeff1DPDE}
    \left[
      \partial_{t,\nu}
      \begin{pmatrix}
        \epsilon_{n} & 0\\
        0 & 1
      \end{pmatrix}
      +
      \begin{pmatrix}
        \sigma_{n} & 0\\
        0 & 0
      \end{pmatrix}
      +
      \begin{pmatrix}
        0 & \partial_{\#} \\
        \partial_{\#} & 0
      \end{pmatrix}
    \right] U_{n}
    = F,
  \end{equation}
  where $\Omega\coloneqq (0,1)$, $\mathcal{H}\coloneqq\Lp{2}(\Omega)^2$, $\nu>\varepsilon$ for an arbitrary $\varepsilon>0$, $A\coloneqq\begin{psmallmatrix}
        0 & \partial_{\#} \\
        \partial_{\#} & 0
      \end{psmallmatrix}$, and $M_n(z)\coloneqq\begin{psmallmatrix}
        \epsilon_{n} & 0\\
        0 & 1
      \end{psmallmatrix}
      +
      z^{-1}
      \begin{psmallmatrix}
        \sigma_{n} & 0\\
        0 & 0
      \end{psmallmatrix}$
      for $n\in\mathbb{N}$ and $\Re z >\varepsilon$.

      Obviously, $\Re z M_n(z)\geq\min\{1,\varepsilon\}$ and $\norm{M_n(z)}\leq\max\{1,1/\varepsilon\}$ hold for $\Re z >\varepsilon$ and $n\in\mathbb{N}$.
      Moreover, $\sobH_{\#}^1(0,1)$ compactly embeds into $\Lp{2}(\Omega)$, by the Rellich--Kondrachov theorem.

      Hence, we are in the setting of \Cref{thm:main-nonlocal-homogenisation-result} with
      $\nu_0\coloneqq \varepsilon$, and $\mathcal{H}_0=\ker(A)$ is the one-dimensional space of constant functions in $\Lp{2}(\Omega)$ times itself.
      In particular,~\labelcref{eq:RoughCoeff1DPDE} is well-posed (in the sense of \Cref{thm:wpee}) for each $n\in\mathbb{N}$, $\nu>\nu_0$, and $F\in\Lp[\nu]{2}(\R;\mathcal{H})$.
\end{example}

\begin{figure}[h]
  \centering
  \begin{tikzpicture}[xscale=2, yscale=1.5]
  \coordinate (x) at (2.2,0);
  \coordinate (-x) at (-2.2,0);

  \coordinate (y) at (0,1.5);
  \coordinate (one) at (0,1);

  \coordinate (v) at ($1/1.5*(0,0.1)$); 
  \coordinate (h) at ($1/2*(0.1,0)$); 

  \draw[<->] (-x) -- (x);
  \draw[->] (0,-0.5) -- (y);

  \draw ($(one) + (h)$) -- ($(one) - (h)$);

  \node[below] at (x) {$x$};
  \node[left] at (one) {$1$};

  \foreach \x in {-2,...,1}{
    \draw[very thick] (\x,1) -- (\x + 0.5,1);
    \draw[very thick] (\x + 0.5,0) -- (\x + 1,0);
  }

  \foreach \x in {-2,-1.5,...,2}{
    \draw ($(\x,0) + (v)$) -- ($(\x,0) - (v)$);
    \node[below] at ($(\x,0) - (v)$) {$\x$};
  }

\end{tikzpicture}

  \caption{\label{fig:indicator}Graph of $\indicator_{O_{n}}$ for $n=1$}
\end{figure}

The intermediate case, meaning both $\ker(A)$ and $\ran(A)$ are infinite-dimensional, has Maxwell's equations as a prominent example.
We will additionally provide some other PDEs with such spatial derivative operators $A$.

\begin{example}[Infinite-dimensional kernel and range]\label{ex:pde-without-compactness}
  We provide several examples in different spatial dimensions and with different spatial derivatives.
  \begin{enumerate}
    \item\label{item:1D-glueing-example}
          The first example in one spatial dimension reads ($\indicator_{(-1,0)}$ is the indicator function of $(-1,0)$, $\partial_x$ stands for the weak derivative on $\Lp{2}(-1,1)$ and $^{\ast}$ denotes the $\Lp{2}$-adjoint
          operator)
          \begin{equation}\label{eq:1D-glueing-example-sequence}
            \left[
              \partial_{t,\nu}
              \smash[b]{\underbrace{%
                  \begin{pmatrix}
                    1 & 0\\
                    0 & 1
                  \end{pmatrix}
                }_{\eqqcolon\mathrlap{M_{0,n}}}
                +
                \underbrace{
                  \begin{pmatrix}
                    \sin\left(2\uppi n \argdot\right) & 0\\
                    0 & \sin\left(2\uppi n \argdot\right)
                  \end{pmatrix}%
                }_{\eqqcolon\mathrlap{M_{1,n}}}
              }%
              +
              \smash[b]{\underbrace{%
                \begin{pmatrix}
                0 & \partial_x\indicator_{(-1,0)}\\
                -(\partial_x\indicator_{(-1,0)})^{\ast} & 0
                \end{pmatrix}%
              }_{\eqqcolon\mathrlap{A}}}
            \right]
            U_n
            = F,
            \vphantom{%
              \underbrace{%
                  \begin{pmatrix}
                    0 & \partial_x\indicator_{[-1,0]}\\
                    -(\partial_x\indicator_{[-1,0]})^{\ast} & 0
                  \end{pmatrix}%
                }_{=\mathrlap{A}}
            }
          \end{equation}
          where $\Omega\coloneqq (-1,1)$, $\mathcal{H}\coloneqq\Lp{2}(\Omega)^{2}$ and $\nu>2$.

          Note that the operator $\partial_{x} \indicator_{(-1,0)}$ on $\Lp{2}(-1,1)$ can be decomposed in a block operator by decomposing $\Lp{2}(-1,1)$ into $\Lp{2}(-1,0) \oplus \Lp{2}(0,1)$. Because of the Sobolev embedding theorem, the corresponding block operator is then
          \begin{equation*}
            \partial_{x} \indicator_{(-1,0)} =
            \begin{pmatrix}
              \mathring{\partial}_{x,\{0\}} & 0 \\
              0 & 0
            \end{pmatrix},
          \end{equation*}
          where $\mathring{\partial}_{x,\{0\}}$ is the derivative on $\Lp{2}(-1,0)$ with domain $\dset{f \in \sobH^{1}(-1,0)}{f(0) = 0}$, i.e., the product of $\partial_{x}$ and $\indicator_{(-1,0)}$ gives a ``boundary'' condition in the middle of the interval.
          Using a direct computation and integration by parts, the $\Lp{2}$-adjoint is then
          \begin{equation*}
            -(\partial_{x} \indicator_{(-1,0)})^{\ast} =
            \begin{pmatrix}
              -(\mathring{\partial}_{x,\{0\}})^{\ast} & 0 \\
              0 & 0
            \end{pmatrix}
            =
            \begin{pmatrix}
              \mathring{\partial}_{x,\{-1\}} & 0 \\
              0 & 0
            \end{pmatrix},
          \end{equation*}
          where $\mathring{\partial}_{x,\{-1\}}$ is the derivative on $\Lp{2}(-1,0)$ with domain $\dset{f \in \sobH^{1}(-1,0)}{f(-1) = 0}$.
          This naturally induces a decomposition of $\Lp{2}(-1,1)^{2}$ into
          \begin{equation*}
            \underbrace{\Lp{2}(-1,0)^{2}}_{=\mathrlap{\mathcal{H}_{1}}} \mathclose{} \oplus \mathopen{}\underbrace{\Lp{2}(0,1)^{2}}_{=\mathrlap{\mathcal{H}_{0}}},
          \end{equation*}
          since we easily obtain $\Lp{2}(-1,0)^{2}=\ran A$ and $\Lp{2}(0,1)^{2}=\ker A$. Moreover, the Rellich--Kondrachov theorem implies that $\dom(A)\cap\mathcal{H}_{1}$
          compactly embeds into $\mathcal{H}$.

          Thus, similarly to \Cref{eq:SineParODEtobeHomogenized}, we are in the setting of \Cref{thm:main-nonlocal-homogenisation-result} with $\nu_0\coloneqq 2$. In particular,~\labelcref{eq:1D-glueing-example-sequence}
          is well-posed (in the sense of \Cref{thm:wpee}) for each $n\in\mathbb{N}$, $\nu>\nu_0$, and $F\in\Lp[\nu]{2}(\R;\mathcal{H})$.
          \end{enumerate}
Going for two spatial dimensions we consider stratified coefficients, i.e., they oscillate only in one direction. For that purpose, we extend $\indicator_{O_{n}}$ (without changing the notation) to
$\indicator_{O_{n}}\colon \R^2 \to \set{0,1}$ via $\indicator_{O_{n}}(x,y)\coloneqq \indicator_{O_{n}}(x)$ for all $x,y\in\R$.
        \begin{enumerate}[resume]
    \item\label{item:ex3} 
          For $\Omega \coloneqq (-2,2)^{2}$, $\Omega_{1} \coloneqq (-1,1)^{2} $ and its indicator function $\indicator_{\Omega_{1}}$ and arbitrary constants $\epsilon_0,\mu_0>0$, consider the problem given by
          $\nu>\varepsilon$ for an arbitrary $\varepsilon>0$ and the following operators on $\mathcal{H}\coloneqq\Lp{2}(\Omega) \oplus \Lp{2}(\Omega)^{2}$:
          \begin{equation}\label{eq:first2DexampleMatLawSpatOp}
          \begin{aligned}
            M_{0,n}
            &\coloneqq\indicator_{\Omega_{1}}
              \begin{pmatrix}
                1-\indicator_{O_{n}} & 0\\
                0 & 1+\indicator_{O_{n}}
              \end{pmatrix}
              + (1-\indicator_{\Omega_{1}})
              \begin{pmatrix}
                \epsilon_0 & 0\\
                0 & \mu_0
              \end{pmatrix},\\
            M_{1,n}
            &\coloneqq\indicator_{\Omega_{1}}
              \begin{pmatrix}
                \indicator_{O_{n}}& 0\\
                0 & 0
              \end{pmatrix},
              \quad
              A \coloneqq
              \begin{pmatrix}
                0 & \div\\
                \cgrad& 0
              \end{pmatrix}.
          \end{aligned}
          \end{equation}
          Note that the oscillating coefficient is only active on the subdomain $\Omega_{1}$.
          Thus, we have a non-periodic anisotropic homogenisation problem.

	Obviously, $\Re z M_n(z)\geq\min \{\varepsilon\epsilon_0,\varepsilon \mu_0,1,\varepsilon\}$ and $\norm{M_n(z)}\leq\max\{\epsilon_0,\mu_0,2,1/\varepsilon\}$ hold for $n\in\mathbb{N}$ and $\Re z >\varepsilon$.
          By \Cref{re:HelmholtzDecom2D}, $\mathcal{H}$ decomposes into
          \begin{equation}\label{eq:kerranAdecomp2Dfirstex}
            \mathcal{H}_{0}=\ker A = \begin{pmatrix} \{0\} \\ \ran(J\grad) \end{pmatrix}
            \quad\text{and}\quad
            \mathcal{H}_{1}=\ran A = \begin{pmatrix} \Lp{2}(\Omega) \\ \ran (\cgrad) \end{pmatrix}.
          \end{equation}
          By the Rellich--Kondrachov theorem, $\cH^1(\Omega)=\cH^1(\Omega)\cap\ran (\div)=\cH^1(\Omega)\cap\ker(\cgrad)^{\perp}$
          compactly embeds into $\Lp{2}(\Omega)$. Going to the adjoints, we easily infer that $\sobH(\div,\Omega)\cap\ran (\cgrad)=\sobH(\div,\Omega)\cap\ker(\div)^{\perp}$
          also compactly embeds into $\Lp{2}(\Omega)^2$, so $\dom(A)\cap\mathcal{H}_{1}$ compactly embeds into $\mathcal{H}$.

          Alltogether, we are in the setting of \Cref{thm:main-nonlocal-homogenisation-result} with $\nu_0\coloneqq \varepsilon$. In particular, the PDE~\labelcref{eq:modelSeqEvolEqForHomogen}
           arising from~\labelcref{eq:first2DexampleMatLawSpatOp}
          is well-posed (in the sense of \Cref{thm:wpee}) for each $n\in\mathbb{N}$, $\nu>\nu_0$, and $F\in\Lp[\nu]{2}(\R;\mathcal{H})$.

    \item\label{item:ex4}
    The same as~\labelcref{item:ex3} with slightly different $M_{0,n}$ and $M_{1,n}$:
          \begin{align*}
            M_{0,n}
            &\coloneqq\indicator_{\Omega_{1}}
              \begin{pmatrix}
                1+\indicator_{O_{n}} & 0\\
                0 & 1-\indicator_{O_{n}}
              \end{pmatrix}
              + (1-\indicator_{\Omega_{1}})
              \begin{pmatrix}
                \epsilon_0 & 0\\
                0 & \mu_0
              \end{pmatrix},
            \\
            M_{1,n}
            &\coloneqq\indicator_{\Omega_{1}}
              \begin{pmatrix}
                0 & 0\\
                0 & \indicator_{O_{n}}
              \end{pmatrix},
              \quad
              A \coloneqq
              \begin{pmatrix}
                0 & \div\\
                \cgrad & 0
              \end{pmatrix}
              .
          \end{align*}
\end{enumerate}
In three dimensions, we also consider the stratified case, and we extend $\indicator_{O_{n}}$ once more (without changing the notation) to
$\indicator_{O_{n}}\colon \R^3 \to \set{0,1}$ via $\indicator_{O_{n}}(x,y,z)\coloneqq \indicator_{O_{n}}(x)$ for all $x,y,z\in\R$.
\begin{enumerate}[resume]
    \item\label{item:ex5} The fourth and final example are Maxwell's equations (cf.\ \Cref{ex:MaxwellAsEvolEq}). For $\Omega \coloneqq (-2,2)^{3}$, $\Omega_{1} \coloneqq (-1,1)^{3}$ and its indicator function $\indicator_{\Omega_{1}}$
    and arbitrary constants $\epsilon,\mu,\sigma,\epsilon_0,\mu_0>0$, consider the problem given by
          $\nu>\varepsilon$ for an arbitrary $\varepsilon>0$ and the following operators on $\mathcal{H}\coloneqq\Lp{2}(\Omega)^{3} \oplus \Lp{2}(\Omega)^{3}$:
          \begin{equation}\label{eq:MaxwellExampleMatLawSpatOp}
          \begin{aligned}
            M_{0,n}
            &\coloneqq\indicator_{\Omega_{1}}
              \begin{pmatrix}
                \epsilon(1-\indicator_{O_{n}}) & 0\\
                0 & \mu(1+\indicator_{O_{n}})
              \end{pmatrix}
              +(1-\indicator_{\Omega_{1}})
              \begin{pmatrix}
                \epsilon_0 & 0\\
                0 & \mu_0
              \end{pmatrix},
            \\
            M_{1,n}
            &\coloneqq\indicator_{\Omega_{1}}
              \begin{pmatrix}
                \sigma\indicator_{O_{n}} & 0\\
                0 & 0
              \end{pmatrix},
              \quad
              A \coloneqq
              \begin{pmatrix}
                0 & -\curl\\
                \ccurl & 0
              \end{pmatrix}.
          \end{aligned}
          \end{equation}

	For $n\in\mathbb{N}$ and $\Re z >\varepsilon$,
          we have $\Re z M_n(z)\geq\min(\varepsilon\epsilon_0,\varepsilon\mu_0,\sigma,\varepsilon\epsilon,\varepsilon\mu)$ and $\norm{M_n(z)}\leq\max(\epsilon_0,\mu_0,2\mu,\sigma/\varepsilon,\epsilon)$.
          By \Cref{re:HelmhDecompTopoTrivial}, the kernel and range of $A$ yield the following decomposition of $\mathcal{H}$:
          \begin{equation}\label{eq:kerranAdecomp3DMaxwellex}
            \mathcal{H}_{0} = \begin{pmatrix} \ker (\ccurl) \\ \ker (\curl) \end{pmatrix}
            = \begin{pmatrix} \ran (\cgrad) \\ \ran (\grad) \end{pmatrix}
            \quad\text{and}\quad
            \mathcal{H}_{1} = \begin{pmatrix} \ran (\curl) \\ \ran (\ccurl) \end{pmatrix}\text{.}
          \end{equation}
          Moreover, we obtain that $\cH(\curl,\Omega)\cap\ran(\curl)=\cH(\curl,\Omega)\cap\ker(\div)$ and $\sobH(\curl,\Omega)\cap\ran(\ccurl)=\sobH(\curl,\Omega)\cap\ker(\cdiv)$ are each compactly embedded into
          $\Lp{2}(\Omega)^{3}$ (Picard--Weber--Weck selection theorem; see \cite{Pi84}).

          Therefore, we are in the setting of \Cref{thm:main-nonlocal-homogenisation-result} with $\nu_0\coloneqq \varepsilon$. In particular, the PDE~\labelcref{eq:modelSeqEvolEqForHomogen}
          arising from~\labelcref{eq:MaxwellExampleMatLawSpatOp}
          is well-posed (in the sense of \Cref{thm:wpee}) for each $n\in\mathbb{N}$, $\nu>\nu_0$, and $F\in\Lp[\nu]{2}(\R;\mathcal{H})$.\qedhere
  \end{enumerate}
\end{example}

\subsection{Homogenisation limits of the examples}\label{sec:examples-limits}
In this section we calculate the homogenisation limits of the previous examples.

\begin{example}[An ordinary differential equation]\label{ex:ode-limit}
  For \Cref{ex:ode}, the convergence of material laws necessary for \Cref{thm:main-nonlocal-homogenisation-result} boils down to finding an $M\colon\C_{\Re>\nu_{0}}\to\Lb(\mathcal{H})$
  with $\norm{M(z)}\leq 2$ and $M(z)^{-1}\in\Lb(\mathcal{H})$ on $\C_{\Re>\nu_{0}}$ such that $M_n(z)^{-1}$ converges in the weak operator topology to $M(z)^{-1}$ for each $z\in\C_{\Re>\nu_{0}}$, where,
  for $n\in\mathbb{N}$ and $z\in\C_{\Re>\nu_{0}}$, the operator
  $M_n(z)\in \Lb(\mathcal{H})$ stands for the multiplication-with-$(1+z^{-1}\sin(2\uppi n \argdot))$ operator.
  Then, the well-posed (in the sense of \Cref{thm:wpee} with $A\coloneqq 0$) limit problem is given by
  \begin{equation*}
  \partial_{t,\nu}M(\partial_{t,\nu})u^{\hom}(t,x)=f(t,x)\text{,}
  \end{equation*}
  for each $\nu>\nu_0$ and $f\in\Lp[\nu]{2}(\R;\mathcal{H})$.

  By~\cite{Wa16a}, we have\footnote{After some tedious but basic calculations, one gets that the operator norm of the double series is strictly smaller than $1$. Thus, one
   indeed obtains $M(z),M(z)^{-1}\in\Lb(\mathcal{H})$ with $\norm{M(z)}\leq 2$ on $\C_{\Re>\nu_{0}}$.\par
   Strictly speaking,~\cite{Wa16a} directly shows convergence of the solution
   operators. Convergence of the corresponding material laws follows by compactness, see, e.g., \cite[Theorem~5.6 and Lemma~5.7]{BuSkWa24}, and the unique correspondence between material laws and their respective operators.}
  \begin{equation}\label{eq:limitMatLawODEDoubleSeries}
  M(z)=1 + \sum_{j=1}^{\infty}
    \left(-\sum_{m=1}^{\infty} \frac{(2m)!}{(2^{m} m!)^{2}}z^{-2m}\right)^{j}
  \end{equation}
  for $z\in\C_{\Re>\nu_{0}}$.

  Therefore, the homogenised problem reads
  \begin{equation*}
    \partial_{t,\nu} u^{\hom}(t,x) + \sum_{j=1}^{\infty} \partial_{t,\nu}
    \left(-\sum_{m=1}^{\infty} \frac{(2m)!}{(2^{m} m!)^{2}}\partial_{t,\nu}^{-2m}\right)^{j} u^{\hom}(t,x)
    = f(t,x)\text{,}
  \end{equation*}
  for each $\nu>\nu_0$ and $f\in\Lp[\nu]{2}(\R;\mathcal{H})$.
  Moreover by~\cite{Wa16a}, the homogenised solution explicitly reads
  \begin{equation*}
    u^{\hom}(t,x)=\int_{-\infty}^t I_0(t-s)f(s,x) \dx[s]\text{,}
  \end{equation*}
  where $I_0$ is the modified Bessel function of first kind. Comparing this to~\labelcref{eq:explicitSolutionSeqODEexample} and consulting \Cref{thm:main-nonlocal-homogenisation-result}, we conclude
  that $u_n$ converges to $u^{\hom}$ weakly but (due to the oscillation) in general not strongly for each $\nu>\nu_0$ and $f\in\Lp[\nu]{2}(\R;\mathcal{H})$.
\end{example}

\begin{remark}
  For the treatment of homogenisation problems for ODEs we refer to the classical \cite{Ta89}. In this article nonlocal effects have been noticed after a homogenisation process of a sequence of non-periodic ODEs. We refer to \cite{Wa14, Wa12b} for a thorough positioning of the present operator-theoretic approach to the classical viewpoint by Tartar and related works. We particulary refer to \cite[Rem.\ 3.8]{Wa14} for the relation to the concept of Young-measures.
\end{remark}

\begin{example}[Finite-dimensional kernel]\label{ex:1d-pde-with-compactness-limit}
  For \Cref{ex:1d-pde-with-compactness}, we will manually calculate the limit of the material laws that \Cref{thm:main-nonlocal-homogenisation-result} asks for.
  The orthogonal projection from $\Lp{2}(0,1)$ onto the closed subspace of constant functions is given by
  the integral mean. Recalling the definition of the nonlocal $\Htopo$-topology and considering $(\epsilon_n)_{n\in\mathbb{N}}$ first, we need to find the limits of
  $(\int_{0}^1 \epsilon_n(x) \dx[x])^{-1}$, $\int_{0}^1 \varphi(x)\epsilon_n(x) \dx[x]$, and $\int_{0}^1 \varphi(x)\epsilon_n(x)\psi(x) \dx[x]$ for all $\varphi,\psi\in \Lp{2}(0,1)$ with integral mean $0$.
  By \Cref{thm:WeakStarLimitPerMultOp} and since $\int_{0}^1 \epsilon_1(x)\dx[x]=1/2$, the respective limits are $2$, $0$, and $\int_{0}^1 \varphi(x)(1/2)\psi(x) \dx[x]$ respectively. With that, one easily proves $\epsilon_n\to 1/2$
  in the nonlocal $\Htopo$-topology. Arguing similarly for $(\sigma_n)_{n\in\N}$, we obtain the homogenised material law $M(z)\coloneqq\begin{psmallmatrix}
        1/2 & 0\\
        0 & 1
      \end{psmallmatrix}
      +
      z^{-1}
      \begin{psmallmatrix}
        1/2 & 0\\
        0 & 0
      \end{psmallmatrix}\in \mathcal{M}(\ker(A),\ran(A))$ for $z\in\C_{\Re>\nu_{0}}$.

       Therefore, \Cref{thm:main-nonlocal-homogenisation-result} yields the well-posed homogenised problem
  \begin{equation*}
    \left[
      \partial_{t,\nu}
      \begin{pmatrix}
        \frac{1}{2} & 0\\
        0 & 1
      \end{pmatrix}
      +
      \begin{pmatrix}
        \frac{1}{2} & 0\\
        0 & 0
      \end{pmatrix}
      +
      \begin{pmatrix}
        0 & \partial_\# \\
        \partial_\# & 0
      \end{pmatrix}
    \right] U^{\hom}
    =F,
  \end{equation*}
  and the solutions $U_n$ converge strongly to $U^{\hom}$ for each $\nu>\nu_0$ and $F\in\Lp[\nu]{2}(\R;\mathcal{H})$.
\end{example}

\begin{example}[Infinite-dimensional kernel and range]\label{ex:pde-without-compactness-limits}
  Let us look at the PDE-sequences in \Cref{ex:pde-without-compactness} and compute their limits.
  \begin{enumerate}
    \item\label{item:1D-glueing-example-limit}
          \Cref{ex:pde-without-compactness}~\ref{item:1D-glueing-example} is in some sense a concatenation of two problems.
          We have already discussed that the decomposition of $\mathcal{H}$ into the range $\mathcal{H}_1$ and kernel $\mathcal{H}_0$ of $A$
          is given by $\Lp{2}(-1,1)^{2} = \Lp{2}(-1,0)^{2} \oplus \Lp{2}(0,1)^{2}$.
          Hence, we can decompose $U$, $F$, $A$, etc.\ accordingly and obtain the following representation of the PDE for $n\in\mathbb{N}$, $\nu>\nu_0$ and $F\in\Lp[\nu]{2}(\R;\mathcal{H})$
          \begin{align*}
            \MoveEqLeft[8]
            \left[\partial_{t,\nu}
              \begin{pmatrix}
                \begin{pmatrix}1 & 0\\0 &1 \end{pmatrix} & 0\\ 0&\begin{pmatrix}1 & 0\\0 &1 \end{pmatrix}
              \end{pmatrix}
              + \sin(2\uppi n \argdot)
              \begin{pmatrix}
                \begin{pmatrix}1 & 0\\ 0 &1 \end{pmatrix} & 0\\ 0 &\begin{pmatrix}1 & 0 \\ 0 &1 \end{pmatrix}
              \end{pmatrix}
            \right.\\
            &\left. +
              \begin{pmatrix}
                \begin{pmatrix} 0 & \mathring{\partial}_{x,\{0\}}\\ \mathring{\partial}_{x,\{-1\}} & 0\end{pmatrix} & 0 \\
                0 &\begin{pmatrix}0 & 0 \\ 0 & 0 \end{pmatrix}
              \end{pmatrix}
            \right]
            \begin{pmatrix}
              U_{n,1} \\ U_{n,2}
            \end{pmatrix}
            =
            \begin{pmatrix}
              F_{1} \\ F_{2}
            \end{pmatrix}.
          \end{align*}
          We now manually calculate the homogenised limit of the material laws.

          Because of the diagonal shape of the material law the nonlocal $\Htopo$-convergence is decoupled. To be precise, we only have to deal with $a_{n,00}^{-1}$ and $a_{n,11}$ (Note that we have swapped the order of $\ran A$ and $\ker A$ in our decomposition, hence $a_{n,00}$ is in the lower right corner and $a_{n,11}$ in the upper left).

          The limiting process on $\mathcal{H}_1$ is a direct consequence of \Cref{thm:WeakStarLimitPerMultOp} implying $a_{n,11} = 1 + \frac{1}{z} \sin(2\uppi n \argdot) \to 1$
          in the weak operator topology for every $z \in \C_{\Re > \nu_0}$.
          Thus, on $(-1,0)$, we obtain the well-posed homogenised problem
          \begin{equation*}
            \left[\partial_{t,\nu} \begin{pmatrix}1 & 0 \\ 0 &1 \end{pmatrix}
              + \begin{pmatrix}0 & \mathring{\partial}_{x,\{0\}}\\ \mathring{\partial}_{x,\{-1\}} & 0 \end{pmatrix}\right]
            U_{1}^{\hom}
            = F_{1},
          \end{equation*}
          for each $\nu>\nu_0$ and $F_1\in\Lp[\nu]{2}(\R;\mathcal{H}_1)$.

          On $\mathcal{H}_0$, we are in the setting of \Cref{ex:ode}.
          Therefore, on $(0,1)$, we obtain from \Cref{ex:ode-limit} (omitting the system) the homogenised solution
          \begin{equation*}
            U_{2}^{\hom}(t,x) = \int_{0}^{t} I_{0}(t-s) F_{2}(s,x) \dx[s],
          \end{equation*}
          for each $\nu>\nu_0$ and $F_2\in\Lp[\nu]{2}(\R;\mathcal{H}_0)$.

          To sum up, the homogenised material law $M(z)\in\mathcal{M}(\mathcal{H}_0,\mathcal{H}_1)$  for $z\in\C_{\Re>\nu_{0}}$ is given by
          $\begin{psmallmatrix}M_{11}&0\\0&M_{00}(z)\end{psmallmatrix}$,
          where $M_{11}(z)\coloneqq \begin{psmallmatrix}1&0\\0&1\end{psmallmatrix}$ and $M_{00}(z)$ is~\labelcref{eq:limitMatLawODEDoubleSeries} times $\begin{psmallmatrix}1&0\\0&1\end{psmallmatrix}$.
           \Cref{thm:main-nonlocal-homogenisation-result} yields well-posedness of the corresponding homogenised problem
           \begin{equation*}
  \left[\partial_{t,\nu}\begin{pmatrix}M_{11}&0\\0&M_{00}(\partial_{t,\nu})\end{pmatrix} + A\right] \begin{pmatrix}
              U_{1}^{\hom} \\ U_{2}^{\hom}
            \end{pmatrix} = F\text{,}
\end{equation*}
           strong convergence of $U_{n,1}$ to $U_{1}^{\hom}$ and weak convergence of $U_{n,2}$ to $U_{2}^{\hom}$ for each $\nu>\nu_{0}$ and $F\in\Lp[\nu]{2}(\R;\mathcal{H})$. Recall from \Cref{ex:ode-limit}
           that in general $U_{2}^{\hom}$ is indeed only reached in a week sense.

    \item\label{item:ex3-limit}
          In \Cref{ex:pde-without-compactness}~\ref{item:ex3} the block operators were given with respect to the decomposition $\Lp{2}(\Omega) \oplus \Lp{2}(\Omega)^{2}$. On $\Omega\setminus \overline{\Omega_1}$,
          the material law is constantly $\begin{psmallmatrix}
                \epsilon_0 & 0\\
                0 & \mu_0
              \end{psmallmatrix}$.

           The $\ker A \oplus \ran A$ decomposition~\labelcref{eq:kerranAdecomp2Dfirstex} lets us compute the nonlocal $\Htopo$-limit in the first row manually.
          With \Cref{thm:WeakStarLimitPerMultOp}, we infer
          \begin{equation*}
            1 - \indicator_{O_{n}} + \frac{1}{z} \indicator_{O_{n}} \to \frac{1}{2} + \frac{1}{z} \frac{1}{2}
          \end{equation*}
          for $z\in\C_{\Re>\nu_{0}}$ in $\Lb(\Lp{2}(\Omega_1))$ in the weak operator topology. That implies
          \begin{equation*}
            \indicator_{\Omega_{1}}\left(1 - \indicator_{O_{n}} + \frac{1}{z} \indicator_{O_{n}}\right)+(1-\indicator_{\Omega_{1}})\epsilon_0 \to \indicator_{\Omega_{1}}\left(\frac{1}{2} + \frac{1}{z} \frac{1}{2}\right)
            +(1-\indicator_{\Omega_{1}})\epsilon_0
          \end{equation*}
          for $z\in\C_{\Re>\nu_{0}}$ in $\Lb(\Lp{2}(\Omega))$ in the weak operator topology.

          The second row of~\labelcref{eq:kerranAdecomp2Dfirstex}  asks for the $\tau(\ran(J\grad), \ran(\cgrad))$-limit of
          \begin{equation*}
          \indicator_{\Omega_{1}}\begin{pmatrix} 1 + \indicator_{O_{n}} & 0 \\ 0 & 1 + \indicator_{O_{n}}\end{pmatrix}+(1-\indicator_{\Omega_{1}})\mu_0\text{.}
          \end{equation*}
           We compute the
          $\tau(\ran(\cgrad),\ran(J\grad))$-limit
          \begin{multline*}
            \indicator_{\Omega_{1}}\begin{pmatrix} (1 + \indicator_{O_{n}})^{-1} & 0 \\ 0 & (1 + \indicator_{O_{n}})^{-1}\end{pmatrix}+(1-\indicator_{\Omega_{1}})\mu_0^{-1}
             \\ %
            \to \indicator_{\Omega_{1}}\begin{pmatrix} \frac{2}{3} & 0 \\ 0 & \frac{3}{4} \end{pmatrix}+(1-\indicator_{\Omega_{1}})\mu_0^{-1},
          \end{multline*}
          and via inverting the $\tau(\ran(J\grad), \ran(\cgrad))$-limit
          \begin{equation*}
            \indicator_{\Omega_{1}}\begin{pmatrix} 1 + \indicator_{O_{n}} & 0 \\ 0 & 1 + \indicator_{O_{n}}\end{pmatrix}+(1-\indicator_{\Omega_{1}})\mu_0
            \to\indicator_{\Omega_{1}}\begin{pmatrix} \frac{3}{2} & 0 \\ 0 & \frac{4}{3} \end{pmatrix}+(1-\indicator_{\Omega_{1}})\mu_0,
          \end{equation*}
          using \Cref{thm:nonlocal-H-convergence-symm-wrt-inv}, \Cref{thm:H-converges-equivalent-to-Schur-convergence-2D}, \Cref{thm:localsums} and \Cref{thm:ExplicitLocalHLimitForStratMedia}.

          All in all, the homogenised material law $M(z)\in\mathcal{M}(\mathcal{H}_0,\mathcal{H}_1)$ for $z\in\C_{\Re>\nu_{0}}$ is given by
          $M(z)\coloneqq M_{0} + z^{-1}M_{1}$, where
          \begin{equation*}
            M_0 \coloneqq
            \indicator_{\Omega_{1}}
            \begin{pmatrix}
              \frac{1}{2} & 0\\
              0 & \begin{pmatrix}
                \frac{3}{2} & 0\\
                0 & \frac{4}{3}
              \end{pmatrix}
            \end{pmatrix}
            +
            (1-\indicator_{\Omega_{1}})
            \begin{pmatrix}
              \epsilon_0 & 0\\
              0 & \mu_0
            \end{pmatrix},\quad
            M_1 \coloneqq \indicator_{\Omega_{1}}
            \begin{pmatrix}
              \frac{1}{2} & 0\\
              0 & 0
            \end{pmatrix}.
          \end{equation*}
          \Cref{thm:main-nonlocal-homogenisation-result} yields well-posedness of the homogenised problem
          \begin{equation*}
  \bigl(\partial_{t,\nu}M_{0} + M_{1} + A\bigr) U^{\hom} = F
\end{equation*}
           with weak convergence of the $\Lp[\nu]{2}(\R;\ker(A))$-component of the solution sequence and
          strong convergence of  the $\Lp[\nu]{2}(\R;\ran(A))$-component of the solution sequence for each $\nu>\nu_{0}$ and $F\in\Lp[\nu]{2}(\R;\mathcal{H})$.

    \item\label{item:ex4-limit} For \Cref{ex:pde-without-compactness}~\ref{item:ex4}, we exactly repeat the ideas from~\labelcref{item:ex3-limit}.

    The first row of~\labelcref{eq:kerranAdecomp2Dfirstex} results in computing
    \begin{equation*}
            \indicator_{\Omega_{1}}(1 + \indicator_{O_{n}} )+(1-\indicator_{\Omega_{1}})\epsilon_0 \to \indicator_{\Omega_{1}}\frac{3}{2}
            +(1-\indicator_{\Omega_{1}})\epsilon_0
          \end{equation*}
           in $\Lb(\Lp{2}(\Omega))$ in the weak operator topology.

          For the second row of~\labelcref{eq:kerranAdecomp2Dfirstex}, we have
          \begin{multline*}
            \indicator_{\Omega_{1}}\bigl(1 - (1 - \tfrac{1}{z} )\indicator_{O_{n}}\bigr)^{-1}  \begin{pmatrix}1 & 0 \\ 0 & 1\end{pmatrix} + (1-\indicator_{\Omega_{1}})\mu_0^{-1}
             \\ %
            \to
            \indicator_{\Omega_{1}}
            \begin{pmatrix}
              \frac{2}{1 + \frac{1}{z}} & 0 \\
              0 & \frac{1}{2} (1 + z)
            \end{pmatrix} + (1-\indicator_{\Omega_{1}})\mu_0^{-1}
          \end{multline*}
          in $\tau(\ran(\cgrad),\ran(J\grad))$ for $z\in\C_{\Re>\nu_{0}}$.
          Hence, the $\tau(\ran(J\grad),\ran(\cgrad))$-limit reads
          \begin{equation*}
            \indicator_{\Omega_{1}}\begin{pmatrix}
              \frac{1}{2} + \frac{1}{z}\frac{1}{2} & 0 \\
              0 & 2 (1 + z)^{-1}
            \end{pmatrix}+(1-\indicator_{\Omega_{1}})\mu_0
          \end{equation*}
          for $z\in\C_{\Re>\nu_{0}}$.
          If we substitute
          \begin{equation*}
            2(1 + z)^{-1} = \frac{1}{z} (z + 1 - 1) 2 (1 + z)^{-1}
            = \frac{1}{z} (2 - 2 (1 + z)^{-1}),
          \end{equation*}
          the homogenised material law $M(z)\in\mathcal{M}(\mathcal{H}_0,\mathcal{H}_1), z\in\C_{\Re>\nu_{0}},$ is given by $M(z)\coloneqq M_{0} + z^{-1}M_{1}$, where
          \begin{align}
            M_{0}
            &\coloneqq\indicator_{\Omega_{1}}
              \begin{pmatrix}
                \frac{3}{2} & 0\\
                0 & \begin{pmatrix}
                      \frac{1}{2} & 0\\
                      0 & 0
                    \end{pmatrix}
              \end{pmatrix}
              + (1-\indicator_{\Omega_{1}})
              \begin{pmatrix}
                \epsilon_0 & 0\\
                0 & \mu_0
              \end{pmatrix},
            \notag\\
            M_1
            &\coloneqq \indicator_{\Omega_{1}}
              \begin{pmatrix}
                0 & 0\\
                0 & \begin{pmatrix}
                      \frac{1}{2}& 0\\
                      0 & 2-2(1+z)^{-1}
                    \end{pmatrix}
              \end{pmatrix}.
              \label{eq:memoryterminhomlimit1}
          \end{align}
          Note that by the limit process, we obtained an operator with the memory term $(1+z)^{-1}\approx (1+\partial_{t,\nu})^{-1}$.
    \item\label{item:ex5-limit} For \Cref{ex:pde-without-compactness}~\ref{item:ex5}, we have to consider the decomposition~\labelcref{eq:kerranAdecomp3DMaxwellex}.

          For the first row, we need to find the $\tau(\ran(\cgrad),\ran(\curl))$-limit of
            \begin{equation*}
            \indicator_{\Omega_{1}}
                    \begin{psmallmatrix}
                      (1 - \indicator_{O_{n}}) \epsilon + \tfrac{1}{z} \indicator_{O_{n}}\sigma & 0 & 0 \\
                      0 & (1 - \indicator_{O_{n}}) \epsilon + \tfrac{1}{z} \indicator_{O_{n}}\sigma & 0 \\
                      0 & 0 & (1 - \indicator_{O_{n}}) \epsilon + \tfrac{1}{z} \indicator_{O_{n}}\sigma
                    \end{psmallmatrix}+(1-\indicator_{\Omega_{1}})\epsilon_0
                  \end{equation*}
                  for $z\in\C_{\Re>\nu_{0}}$.
                  Applying \Cref{thm:H-converges-equivalent-to-Schur-convergence}, \Cref{thm:localsums} and \Cref{thm:ExplicitLocalHLimitForStratMedia}, we obtain the limit
                  \begin{equation*}
                  \indicator_{\Omega_{1}}
                    \begin{pmatrix}
                      \frac{1}{z} 2(\epsilon^{-1} \frac{1}{z} + \sigma^{-1})^{-1} & 0 & 0 \\
                      0 & \frac{1}{2}\epsilon + \frac{1}{z} \frac{1}{2}\sigma & 0\\
                      0 & 0 & \frac{1}{2}\epsilon + \frac{1}{z} \frac{1}{2}\sigma
                    \end{pmatrix}+(1-\indicator_{\Omega_{1}})\epsilon_0
                  \end{equation*}
                  for $z\in\C_{\Re>\nu_{0}}$.

                  For the second row of~\labelcref{eq:kerranAdecomp3DMaxwellex}, we need to find the $\tau(\ran(\grad),\ran(\ccurl))$-limit of
                  \begin{equation*}
                    \indicator_{\Omega_{1}}
                    \begin{pmatrix}
                      (1 + \indicator_{O_{n}}) \mu & 0 & 0 \\
                      0 & (1 + \indicator_{O_{n}}) \mu & 0 \\
                      0 & 0 & (1 + \indicator_{O_{n}}) \mu
                    \end{pmatrix}+(1-\indicator_{\Omega_{1}})\mu_0\text{.}
                  \end{equation*}
                  \Cref{thm:H-converges-equivalent-to-Schur-convergence-diff-BD}, \Cref{thm:localsums} and \Cref{thm:ExplicitLocalHLimitForStratMedia} yield the limit
                  \begin{equation*}
                    \indicator_{\Omega_{1}}
                    \begin{pmatrix}
                      \frac{4}{3}\mu & 0 & 0 \\
                      0 & \frac{3}{2}\mu & 0 \\
                      0 & 0 &\frac{3}{2}\mu
                    \end{pmatrix}+(1-\indicator_{\Omega_{1}})\mu_0\text{.}
                  \end{equation*}

                  Thus, if we substitute
                  \begin{align*}
                    \frac{1}{z} 2\Bigl(\epsilon^{-1} \frac{1}{z} + \sigma^{-1}\Bigr)^{-1}
                    &= \frac{1}{z} 2 \epsilon z (\sigma + z \epsilon)^{-1} \sigma
                      = \frac{1}{z} 2 (\epsilon z + \sigma - \sigma) (\sigma + z \epsilon)^{-1} \sigma \\
                    &= \frac{1}{z} 2 \bigl(1 - \sigma (\sigma + z \epsilon)^{-1}\bigr) \sigma\text{,}
                  \end{align*}
          the homogenised material law $M(z)\in\mathcal{M}(\mathcal{H}_0,\mathcal{H}_1)$ is given by
          $M(z)\coloneqq M_{0} + z^{-1}M_{1}$, where
          \begin{align*}
            M_{0}
            &\coloneqq \indicator_{\Omega_{1}}
              \begin{pNiceArray}{ccw{c}{2em}ccc}[left-margin=0.7em, right-margin=0.7em]
                0 & 0 & 0 &  & &  \\
                0 & \frac{1}{2}\epsilon & 0 &  & 0 &  \\
                0 & 0 & \frac{1}{2}\epsilon &  &  &  \\
                &  &  & \frac{4}{3}\mu & 0 & 0 \\
                 & 0 &  & 0 & \frac{3}{2}\mu & 0 \\
                 &  &  & 0 & 0 & \frac{3}{2}\mu
                 \CodeAfter
                 \SubMatrix({1-1}{3-3})
                 \SubMatrix({4-4}{6-6})
              \end{pNiceArray}
              + (1-\indicator_{\Omega_{1}})
              \begin{pmatrix}
                \epsilon_0 & 0\\
                0 & \mu_0
              \end{pmatrix},
            \\
            M_{1}
            &\coloneqq \indicator_{\Omega_{1}}
              \begin{pmatrix}
                \begin{pmatrix}
                  2\sigma(1-\sigma(\sigma+z\epsilon)^{-1})&0&0\\
                  0&\frac{\sigma}{2}&0\\
                  0&0&\frac{\sigma}{2}
                \end{pmatrix}&0\\
                0&0
              \end{pmatrix}
          \end{align*}
          As in~\labelcref{item:ex4-limit}, this material law has a memory term. This time, it is $(\sigma+z\epsilon)^{-1}\approx (\sigma+\partial_{t,\nu}\epsilon)^{-1}$.
          \Cref{thm:main-nonlocal-homogenisation-result} yields well-posedness of the homogenised problem
                    \begin{equation*}
  \bigl(\partial_{t,\nu}M_{0} + M_{1} + A\bigr) U^{\hom} = F
\end{equation*}
          with weak convergence of the $\Lp[\nu]{2}(\R;\ker(A))$-component of the solution sequence and
          strong convergence of  the $\Lp[\nu]{2}(\R;\ran(A))$-component of the solution sequence for each $\nu>\nu_{0}$ and $F\in\Lp[\nu]{2}(\R;\mathcal{H})$.\myqedhere
  \end{enumerate}
\end{example}

\begin{remark}
  We refer to \cite[Sec.\ 7]{Wa18} for a more in depth discussion of homogenisation problems for Maxwell's equations in relation to the operator-theoretic perspective provided here; see, in particular \cite{We01} or \cite{Su08} for other approaches.
\end{remark}

\section{Numerical Simulations}\label{sec:numerics}

This section is devoted to complement a numerical study for our theoretical findings from the previous sections. The computations presented here will support our results concerning strong and weak convergence. The simulations are carried out in  $\mathbb{SOFE}$, a finite-element framework in Matlab and Octave,
see \texttt{github.com/SOFE-Developers/SOFE}.

For the temporal dimension in the examples to come, we use a discontinuous Galerkin method in time, whereas a continuous Galerkin method is applied for the spatial directions. The details of this method tailored for evolutionary equations are provided
in \cite{FrTrWa19}; we briefly summarise the essential parts next. For a fixed $T>0$ and the corresponding time interval $[0,T]$, consider a partition $0 = t_{0} < \dots < t_{\mbound} = T$ for some $\mbound \in\N$ and set
$I_{m}\coloneqq \lparen t_{m-1},t_{m} \rbrack$ for $1 \leq m \leq \mbound$.
For $\mathcal{U}$, a piecewise polynomial discontinuous space
of degree $1$ in time corresponding to the partition of $[0,T]$ and an $A$-conforming piecewise polynomial space of degree $1$ in space corresponding to the spatial Hilbert space $\mathcal{H}$, the method is now given by:
Find $U\in \mathcal{U}$ such that for all $m \in\{1,\dots,\mbound\}$ and $\Phi\in \mathcal{U}$
\begin{equation}
  \label{eq:numerical-formulation}
  \begin{multlined}[0.8\textwidth]
    \Qmr{(\partial_t M_0 + M_1 + A)U,\Phi}
    + \scprod{M_0 \jump{U}_{m-1}}{{\Phi}^{+}_{m-1}}_{\mathcal{H}} \\
    = \Qmr{F,\Phi} + \scprod{M_0U_0}{\Phi(0+)}_{\mathcal{H}}.
  \end{multlined}
\end{equation}
Here, $\Qmr{a,b}\coloneqq\Qm{\scprod{a}{b}_{\mathcal{H}}}$ is the weighted right-sided Gau\ss{}--Radau quadrature rule with the property
\begin{equation*}
  \Qm{p} = \int_{I_m}p(t)\euler^{-2\rho(t-t_{m-1})}\dx[t] \quad \text{for all}\quad p \in \mathcal{P}_{2}(I_m),\footnotemark
\end{equation*}%
\footnotetext{$\mathcal{P}_{2}(I_m)$ denotes the set of all polynomials on $I_{m}$ of degree smaller than or equal to $2$.}%
i.e., it approximates the exponentially weighted scalar product in space-time.
Furthermore, we denote by
\begin{equation*}
  \jump{U}_{m-1}\coloneqq
  \begin{cases}
    U(t_{m-1}+) - U(t_{m-1}-),& m\in\{2,\ldots,\mbound\}\\
    U(t_0+),& m=1,
  \end{cases}
\end{equation*}
the jump of $U$ at $t_{m-1}$ and set $\Phi^+_{m-1}\coloneqq \Phi(t_{m-1}+)$.

We replace $M_{0}$, $M_{1}$ and $U$ in \eqref{eq:numerical-formulation} by $M_{0,n}$, $M_{1,n}$ and $U_{n}$, respectively, for the non-limit case.

For convenience, we restricted our attention to finite intervals as time-horizon instead of $\R$ that was used to model the temporal scale in our theoretical sections.

\subsection{Continuation of \Cref{ex:ode}}
We consider the sequence of equations ($n\in\N$)
\begin{equation*}
  \partial_t u_n(t,x)+\sin\left(2\uppi n x\right)u_n(t,x)=f(t,x)
\end{equation*}
 in the domain $[0,2]\times[0,1]$.
Asking for initial zero conditions, we find
the solution of the homogenised problem
\begin{equation*}
  u^{\hom}(t,x)=\int_{0}^{t} I_{0}(t-s)f(s,x) \dx[s],
\end{equation*}
where $I_0$ is the modified Bessel function of first kind (see \Cref{ex:ode-limit} and \cite{Wa16a}). For our numerical
experiment we choose $f(t,x)=1$ in $[0,2]\times[0,1]$. \Cref{fig:ex1}
\begin{figure}[htb]
  \centering
  \begin{subfigure}[c]{0.66\textwidth}
    \includegraphics[width=0.49\textwidth]{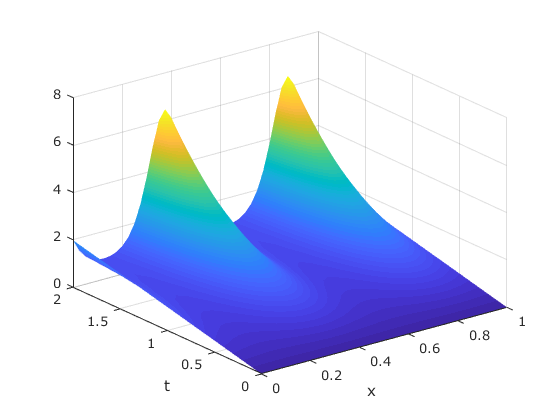}
    \includegraphics[width=0.49\textwidth]{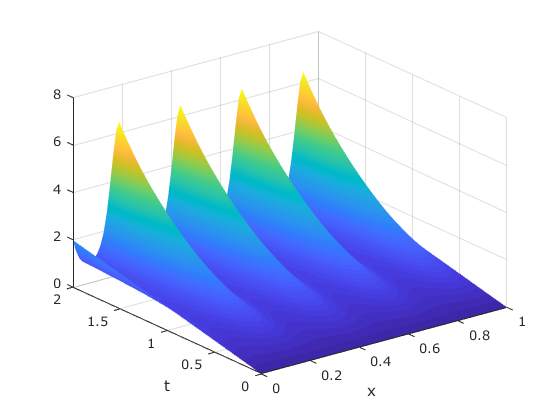}
    \includegraphics[width=0.49\textwidth]{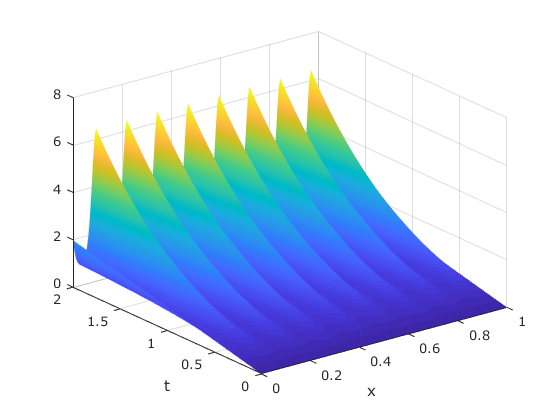}
    \includegraphics[width=0.49\textwidth]{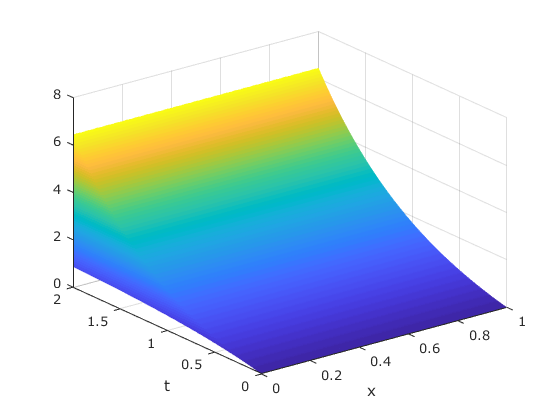}
    \caption{Solutions of \Cref{ex:ode} for $n\in\{2,4,8,1024\}$}
  \end{subfigure}
  \hfill
  \begin{subfigure}[c]{0.33\textwidth}
    \includegraphics[width=1\textwidth]{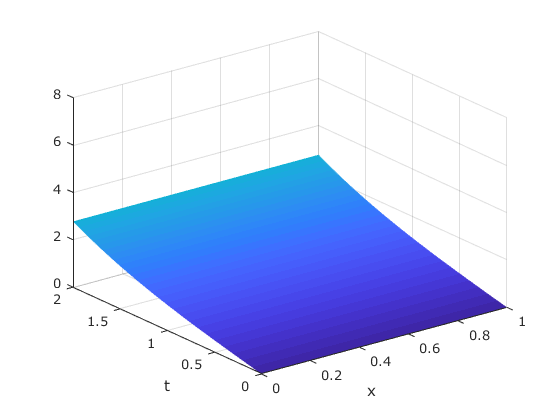}
    \caption{Solution of the homgenisation limit of \Cref{ex:ode}}
  \end{subfigure}
  \caption{\label{fig:ex1}Solutions of \Cref{ex:ode} for $n\in\{2,4,8,1024\}$ and for the homogenised problem}
\end{figure}
%
%
shows the solutions of \Cref{ex:ode} for different values of $n$. We observe that doubling $n$ doubles the number of waves
in the solution. Thus, we have a shrinking effect of a periodic function in $x$-direction. In contrast to that, the last
plot shows $u^{\hom}$, the solution of the homogenised problem.

This visualises weak but the lack of strong convergence.
Unfortunately, we cannot compute the following quantity for all $v \in \Lp{2}(\Omega)$
\begin{equation*}
  \abs{\scprod{u_{n} - u^{\hom}}{v}_{\Lp{2}}}.
\end{equation*}
Thus we use the following sample of example functions
\begin{equation*}
  v\in
  \set[\big]{%
  (t,x)\mapsto 1,%
  (t,x)\mapsto x,%
  (t,x)\mapsto x^2,%
  (t,x)\mapsto \sin(\uppi x),%
  (t,x)\mapsto t%
  }.
\end{equation*}
The first and last function are constant in space, while all but the last function are constant in time.

\Cref{fig:ex1_weak}
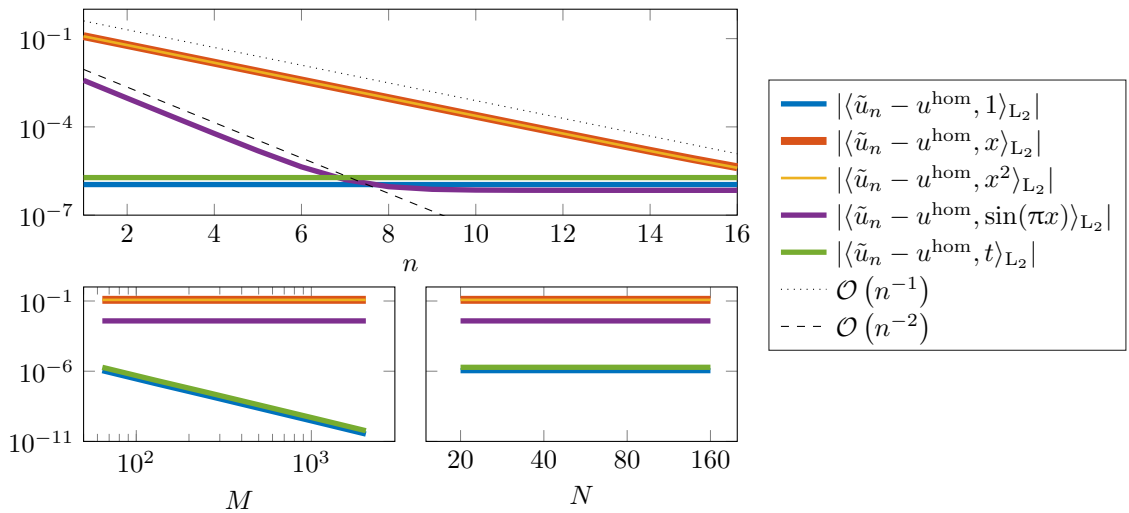
\begin{figure}[htb]
%
%
\definecolor{mycolor1}{rgb}{0.00000,0.44700,0.74100}%
\definecolor{mycolor2}{rgb}{0.85000,0.32500,0.09800}%
\definecolor{mycolor3}{rgb}{0.92900,0.69400,0.12500}%
\definecolor{mycolor4}{rgb}{0.49400,0.18400,0.55600}%
\definecolor{mycolor5}{rgb}{0.46600,0.67400,0.18800}%
\begin{tikzpicture}

\begin{axis}[%
width=0.63\textwidth,
height=0.2\textwidth,
scale only axis,
xmin=1,
xmax=16,
xlabel = $n$,
ymode=log,
ymin=1e-07,
ymax=1,
yminorticks=true,
axis background/.style={fill=white},
legend style={at={(0.66\textwidth,0)}, anchor=west, legend cell align=left, align=left, draw=white!15!black}
]
\addplot [color=mycolor1, line width=2.0pt]
  table[row sep=crcr]{%
 1	1.0963e-06\\
 2	1.0963e-06\\
 3	1.0963e-06\\
 4	1.0963e-06\\
 5	1.0963e-06\\
 6	1.0963e-06\\
 7	1.0963e-06\\
 8	1.0963e-06\\
 9	1.0963e-06\\
10	1.0963e-06\\
11	1.0963e-06\\
12	1.0963e-06\\
13	1.0963e-06\\
14	1.0963e-06\\
15	1.0963e-06\\
16	1.0963e-06\\
};
\addlegendentry{$\abs{\scprod{\tilde{u}_n-u^{\hom}}{1}_{\Lp{2}}}$}
\addplot [color=mycolor2, line width=3.0pt]
  table[row sep=crcr]{%
 1	1.2138e-01\\
 2	6.0692e-02\\
 3	3.0346e-02\\
 4	1.5173e-02\\
 5	7.5870e-03\\
 6	3.7938e-03\\
 7	1.8972e-03\\
 8	9.4885e-04\\
 9	4.7470e-04\\
10	2.3762e-04\\
11	1.1909e-04\\
12	5.9817e-05\\
13	3.0183e-05\\
14	1.5365e-05\\
15	7.9568e-06\\
16	4.2525e-06\\
};
\addlegendentry{$\abs{\scprod{\tilde{u}_n-u^{\hom}}{x}_{\Lp{2}}}$}
\addplot [color=mycolor3, line width=1.0pt]
  table[row sep=crcr]{%
 1	1.2018e-01\\
 2	6.0392e-02\\
 3	3.0271e-02\\
 4	1.5154e-02\\
 5	7.5821e-03\\
 6	3.7924e-03\\
 7	1.8967e-03\\
 8	9.4859e-04\\
 9	4.7450e-04\\
10	2.3744e-04\\
11	1.1890e-04\\
12	5.9634e-05\\
13	3.0000e-05\\
14	1.5183e-05\\
15	7.7740e-06\\
16	4.0697e-06\\
};
\addlegendentry{$\abs{\scprod{\tilde{u}_n-u^{\hom}}{x^2}_{\Lp{2}}}$}
\addplot [color=mycolor4, line width=2.0pt]
  table[row sep=crcr]{%
 1	3.8247e-03\\
 2	9.4544e-04\\
 3	2.3619e-04\\
 4	5.9528e-05\\
 5	1.5403e-05\\
 6	4.3740e-06\\
 7	1.6169e-06\\
 8	9.2767e-07\\
 9	7.5536e-07\\
10	7.1228e-07\\
11	7.0152e-07\\
12	6.9882e-07\\
13	6.9815e-07\\
14	6.9798e-07\\
15	6.9794e-07\\
16	6.9793e-07\\
};
\addlegendentry{$\abs{\scprod{\tilde{u}_n-u^{\hom}}{\sin(\uppi x)}_{\Lp{2}}}$}
\addplot [color=mycolor5, line width=2.0pt]
  table[row sep=crcr]{%
 1	1.8920e-06\\
 2	1.8920e-06\\
 3	1.8920e-06\\
 4	1.8920e-06\\
 5	1.8920e-06\\
 6	1.8920e-06\\
 7	1.8920e-06\\
 8	1.8920e-06\\
 9	1.8920e-06\\
10	1.8920e-06\\
11	1.8920e-06\\
12	1.8920e-06\\
13	1.8920e-06\\
14	1.8920e-06\\
15	1.8920e-06\\
16	1.8920e-06\\
};
\addlegendentry{$\abs{\scprod{\tilde{u}_n-u^{\hom}}{t}_{\Lp{2}}}$}
\addplot [color=black, dotted]
  table[row sep=crcr]{%
 1	4e-01\\
16	1.2207e-05\\
};
\addlegendentry{$\ord{n^{-1}}$}
\addplot [color=black, dashed]
  table[row sep=crcr]{%
 1	9e-03\\
16	8.3819e-12\\
};
\addlegendentry{$\ord{n^{-2}}$}
\end{axis}
%
\begin{axis}[%
width=0.3\textwidth,
height=0.15\textwidth,
at={(0,-3.0cm)},
scale only axis,
xmin=50,
xmax=3000,
xmode=log,
xlabel = $\mbound$,
ymode=log,
ymin=1e-11,
ymax=1,
yminorticks=true,
axis background/.style={fill=white}
]
\addplot [color=mycolor1, line width=2.0pt]
  table[row sep=crcr]{%
  64	1.0963e-06\\
 128	1.3390e-07\\
 256	1.7313e-08\\
 512	2.1665e-09\\
1024	2.7140e-10\\
2048	3.3330e-11\\
};
\addplot [color=mycolor2, line width=3.0pt]
  table[row sep=crcr]{%
  64	1.2138e-01\\
 128	1.2138e-01\\
 256	1.2138e-01\\
 512	1.2138e-01\\
1024	1.2138e-01\\
2048	1.2138e-01\\
};
\addplot [color=mycolor3, line width=1.0pt]
  table[row sep=crcr]{%
  64	1.2018e-01\\
 128	1.2018e-01\\
 256	1.2018e-01\\
 512	1.2018e-01\\
1024	1.2018e-01\\
2048	1.2018e-01\\
};
\addplot [color=mycolor4, line width=2.0pt]
  table[row sep=crcr]{%
  64	3.8247e-03\\
 128	3.8241e-03\\
 256	3.8186e-03\\
 512	3.8186e-03\\
1024	3.8186e-03\\
2048	3.8186e-03\\
};
\addplot [color=mycolor5, line width=2.0pt]
  table[row sep=crcr]{%
  64	1.8920e-06\\
 128	2.3095e-07\\
 256	2.9946e-08\\
 512	3.7486e-09\\
1024	4.6959e-10\\
2048	5.7799e-11\\
};
\end{axis}
%
\begin{axis}[%
width=0.3\textwidth,
height=0.15\textwidth,
at={(0.33\textwidth,-3.0cm)},
scale only axis,
xmin=15,
xmax=200,
xmode=log,
xlabel = $N$,xtick={20,40,80,160},xticklabels={20,40,80,160},
ymode=log,
ymin=1e-11,
ymax=1,
yminorticks=true,yticklabels={},
axis background/.style={fill=white}
]
\addplot [color=mycolor1, line width=2.0pt]
  table[row sep=crcr]{%
  20	1.0963e-06\\
  40	1.1006e-06\\
  80	1.1006e-06\\
 160	1.1006e-06\\
};
\addplot [color=mycolor2, line width=3.0pt]
  table[row sep=crcr]{%
  20	1.2138e-01\\
  40	1.2138e-01\\
  80	1.2138e-01\\
 160	1.2138e-01\\
};
\addplot [color=mycolor3, line width=1.0pt]
  table[row sep=crcr]{%
  20	1.2018e-01\\
  40	1.2019e-01\\
  80	1.2019e-01\\
 160	1.2019e-01\\
};
\addplot [color=mycolor4, line width=2.0pt]
  table[row sep=crcr]{%
  20	3.8247e-03\\
  40	3.8193e-03\\
  80	3.8193e-03\\
 160	3.8193e-03\\
};
\addplot [color=mycolor5, line width=2.0pt]
  table[row sep=crcr]{%
  20	1.8920e-06\\
  40	1.9000e-06\\
  80	1.9000e-06\\
 160	1.9000e-06\\
};
\end{axis}
\end{tikzpicture}%

  \caption{\label{fig:ex1_weak}Results investigating weak convergence for \Cref{ex:ode}}
\end{figure}
shows the results of these scalar products. The value of the homogenised problem
was calculated using MAPLE and the given formula for $u^{\hom}$. For the numerical simulation
we used a fixed equidistant mesh with $N=10\cdot n$ cells in the spatial direction and $\mbound=64$
cells in the time direction in the upper plot. We do observe, that the functions
$v$ being constant in space give a constant result, while we have
results of order $\ord{n^{-1}}$  for $v(t,x)=x$ and $v(t,x)=x^2$. For $v(t,x)=\sin(\uppi x)$, we observe an initial second order
convergence $\ord{n^{-2}}$ but the graph stagnates at a certain value. The reason for this
behaviour lies in the approximation of $u_{n}$.
In fact, \Cref{fig:ex1_weak} does not show $\scprod{u_n-u^{\hom}}{v}_{\Lp{2}}$ but rather
\begin{equation*}
  \scprod{\tilde{u}_n-u^{\hom}}{v}_{\Lp{2}}
  = \scprod{\tilde{u}_n-u_n}{v}_{\Lp{2}}+\scprod{u_n-u^{\hom}}{v}_{\Lp{2}}
\end{equation*}
replacing $u_n$ by a numerical approximation $\tilde{u}_n$ using \eqref{eq:numerical-formulation}.
In order to investigate the effect of this numerical approximation, we fixed $n=1$ for the lower
plots in \Cref{fig:ex1_weak} and varied the mesh in time with $\mbound\in\{64,\,128,\,256,\,512,\,1024,\,2048\}$
and in space with $N\in\{20,\,40,\,80,\,160\}$, respectively.
We observe a reduction for $v=1$ and $v(t,x)=t$, when refining in time, but none for refining in space.
This indicates, that the accuracy is limited by the approximation errors in time, which are
of order $10^{-6}$ in this example.

These plots underpin the weak convergence $u_{n}\weakto u^{\hom}$, and even suggest vaguely a rate of order $n^{-1}$. If true in general, this would complement quantitative findings for strong convergence for evolutionary equations (cf.~\cite{CoEsWa24,FrWa18}).

\subsection{Continuation of  \Cref{ex:pde-without-compactness}~\ref{item:1D-glueing-example}}

We consider the sequence of equations ($n\in\N$)
\begin{multline*}
  \left[\partial_t
    \begin{pmatrix}
      1 & 0\\
      0 & 1
    \end{pmatrix}
    +
    \begin{pmatrix}
      \sin\left(2\uppi n x\right) & 0\\
      0 & \sin\left(2\uppi n x\right)
    \end{pmatrix}
    \right.
     \\ %
    +
    \left.
    \begin{pmatrix}
      0 & \partial_x\indicator_{[-1,0]}\\
      -(\partial_x\indicator_{[-1,0]})^* & 0
    \end{pmatrix}
  \right]
  \begin{pmatrix}
    u_n\\v_n
  \end{pmatrix}
  =
  \begin{pmatrix}
    \sin(2\uppi t)\\
    x-\frac{1}{2}
  \end{pmatrix}
\end{multline*}
 in $[0,2]\times[-1,1]$.
\begin{figure}[htbp]
  \begin{center}
    \fbox{\includegraphics[width=0.45\textwidth]{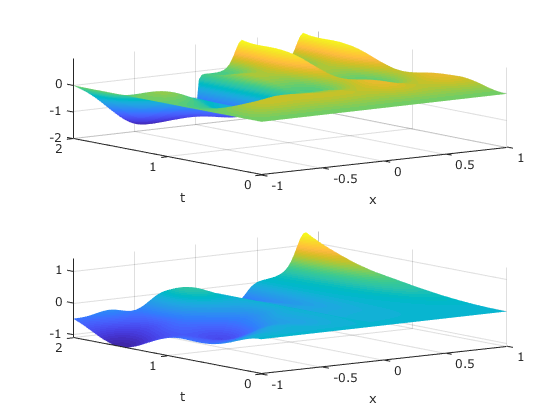}}
    \fbox{\includegraphics[width=0.45\textwidth]{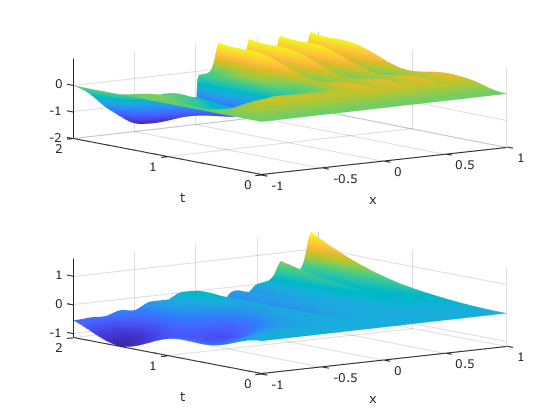}}\\
    \fbox{\includegraphics[width=0.45\textwidth]{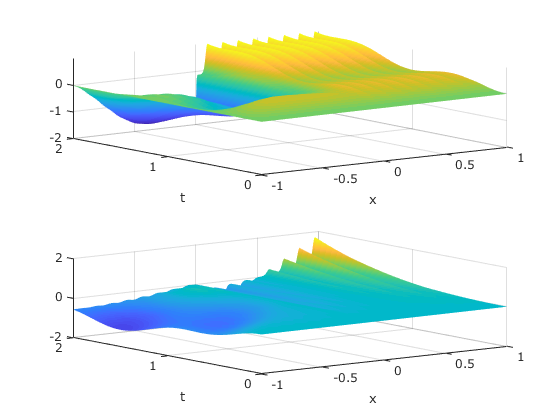}}
    \fbox{\includegraphics[width=0.45\textwidth]{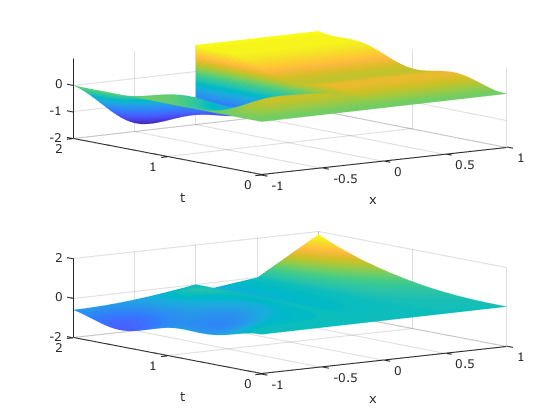}}
  \end{center}
  \caption{\label{fig:1D-glueing-example}Solutions of \Cref{ex:pde-without-compactness}~\ref{item:1D-glueing-example} for $n\in\{2,4,8,1024\}$ from top left to bottom right.}
\end{figure}
\Cref{fig:1D-glueing-example} shows $u_n$ in the upper and $v_n$ in the
lower subplots for different $n\in\N$. As expected, they suggest strong convergence on the spatial $[0,1]$-part and weak convergence on the $[-1,0]$-part.
In order to investigate this further, we use a numerical solution
$(\tilde{u}^{\hom},\tilde{v}^{\hom}) \coloneqq (\tilde{u}_{n},\tilde{v}_{n})$ by setting
$n=2^{13}$, as well as mesh parameters $N=40n=327\,680$, $\mbound=2^8=256$, and we increase the degrees of the piecewise
polynomials to 3 in space and 2 in time in \eqref{eq:numerical-formulation}.

For the numerical simulations, we again vary $n$
and chose $\mbound=2^6$, $N=40\cdot n$ and polynomial degrees 2 in space and 1 in time.
\begin{figure}[htbp]
%
%
\definecolor{mycolor1}{rgb}{0.00000,0.44700,0.74100}%
\definecolor{mycolor2}{rgb}{0.85000,0.32500,0.09800}%
\definecolor{mycolor3}{rgb}{0.92900,0.69400,0.12500}%
\definecolor{mycolor4}{rgb}{0.49400,0.18400,0.55600}%
\definecolor{mycolor5}{rgb}{0.46600,0.67400,0.18800}%
\begin{tikzpicture}

\begin{axis}[%
width=0.615\textwidth,
height=0.2\textwidth,
scale only axis,
xmin=1,
xmax=10,
xtick={1,2,3,4,5,6,7,8,9,10},
xticklabels={},
ymode=log,
ymin=5e-06,
ymax=1e-1,
yminorticks=true,
axis background/.style={fill=white},
legend style={at={(1.03,0.5)}, anchor=west, legend cell align=left, align=left, draw=white!15!black}
]
\addplot [color=mycolor1, line width=2.0pt]
  table[row sep=crcr]{%
 1	2.6876e-02\\                
 2	1.5720e-02\\                
 3	8.3994e-03\\                
 4	4.3219e-03\\                
 5	2.1812e-03\\                
 6	1.0856e-03\\                
 7	5.3156e-04\\                
 8	2.5297e-04\\                
 9	1.1329e-04\\                
10	4.3353e-05\\                
};
\addlegendentry{$\abs{\scprod{\tilde{u}_n-\tilde{u}^{\hom}}{1}_{\Lp{2}}}$}
\addplot [color=mycolor2, line width=2.0pt]
  table[row sep=crcr]{%
 1	5.2193e-03\\
 2	1.1863e-03\\
 3	2.8753e-04\\
 4	7.3944e-05\\
 5	2.0112e-05\\
 6	5.6652e-06\\
 7	1.4627e-06\\
 8	1.0451e-07\\
 9	3.9037e-07\\
10	5.9195e-07\\
};
\addlegendentry{$\abs{\scprod{\tilde{u}_n-\tilde{u}^{\hom}}{x}_{\Lp{2}}}$}
\addplot [color=mycolor3, line width=2.0pt]
  table[row sep=crcr]{%
 1	5.0660e-02\\
 2	2.6640e-02\\
 3	1.3570e-02\\
 4	6.8296e-03\\
 5	3.4162e-03\\
 6	1.6997e-03\\
 7	8.3911e-04\\
 8	4.0827e-04\\
 9	1.9270e-04\\
10	8.4887e-05\\
};
\addlegendentry{$\abs{\scprod{\tilde{u}_n-\tilde{u}^{\hom}}{x^2}_{\Lp{2}}}$}
\addplot [color=mycolor4, line width=2.0pt]
  table[row sep=crcr]{%
 1	2.0667e-02\\
 2	1.1536e-02\\
 3	6.1232e-03\\
 4	3.1577e-03\\
 5	1.6047e-03\\
 6	8.0987e-04\\
 7	4.0781e-04\\
 8	2.0560e-04\\
 9	1.0420e-04\\
10	5.3430e-05\\
};
\addlegendentry{$\abs{\scprod{\tilde{u}_n-\tilde{u}^{\hom}}{\sin(\uppi x)}_{\Lp{2}}}$}
\addplot [color=mycolor5, line width=2.0pt]
  table[row sep=crcr]{%
 1	2.5203e-02\\
 2	1.6238e-02\\
 3	8.9809e-03\\
 4	4.6899e-03\\
 5	2.3821e-03\\
 6	1.1880e-03\\
 7	5.8099e-04\\
 8	2.7499e-04\\
 9	1.2137e-04\\
10	4.4399e-05\\
};
\addlegendentry{$\abs{\scprod{\tilde{u}_n-\tilde{u}^{\hom}}{t}_{\Lp{2}}}$}
\addplot [color=black, dotted]
  table[row sep=crcr]{%
 1	9e-02\\
10	1.7578e-04\\
};
\addlegendentry{$\ord{n^{-1}}$}
\end{axis}

\begin{axis}[%
width=0.615\textwidth,
height=0.2\textwidth,
at={(0,-3.75cm)},
scale only axis,
xmin=1,
xmax=10,
xtick={1,2,3,4,5,6,7,8,9,10},
xticklabels={$2^1$,$2^2$,$2^3$,$2^4$,$2^5$,$2^6$,$2^7$,$2^8$,$2^9$,$2^{10}$},
xlabel = $n$,
ymode=log,
ymin=1e-05,
ymax=1e-0,
yminorticks=true,
axis background/.style={fill=white},
legend style={at={(1.03,0.5)}, anchor=west, legend cell align=left, align=left, draw=white!15!black}
]
\addplot [color=mycolor1, line width=2.0pt]
  table[row sep=crcr]{%
 1	1.3463e-01\\                  
 2	6.7930e-02\\                  
 3	3.4210e-02\\                  
 4	1.7166e-02\\                  
 5	8.5883e-03\\                  
 6	4.2838e-03\\                  
 7	2.1275e-03\\                  
 8	1.0483e-03\\                  
 9	5.0847e-04\\                  
10	2.3848e-04\\                  
};
\addlegendentry{$\abs{\scprod{\tilde{v}_n-\tilde{v}^{\hom}}{1}_{\Lp{2}}}$}
\addplot [color=mycolor2, line width=2.0pt]
  table[row sep=crcr]{%
 1	1.3731e-02\\
 2	6.6376e-03\\
 3	3.4097e-03\\
 4	1.7461e-03\\
 5	8.8651e-04\\
 6	4.4786e-04\\
 7	2.2609e-04\\
 8	1.1456e-04\\
 9	5.8638e-05\\
10	3.0634e-05\\
};
\addlegendentry{$\abs{\scprod{\tilde{v}_n-\tilde{v}^{\hom}}{x}_{\Lp{2}}}$}
\addplot [color=mycolor3, line width=2.0pt]
  table[row sep=crcr]{%
 1	5.2210e-02\\
 2	2.6934e-02\\
 3	1.3679e-02\\
 4	6.8909e-03\\
 5	3.4554e-03\\
 6	1.7272e-03\\
 7	8.6050e-04\\
 8	4.2650e-04\\
 9	2.0934e-04\\
10	1.0072e-04\\
};
\addlegendentry{$\abs{\scprod{\tilde{v}_n-\tilde{v}^{\hom}}{x^2}_{\Lp{2}}}$}
\addplot [color=mycolor4, line width=2.0pt]
  table[row sep=crcr]{%
 1	4.4856e-02\\
 2	2.4297e-02\\
 3	1.2480e-02\\
 4	6.3054e-03\\
 5	3.1634e-03\\
 6	1.5801e-03\\
 7	7.8557e-04\\
 8	3.8759e-04\\
 9	1.8843e-04\\
10	8.8808e-05\\
};
\addlegendentry{$\abs{\scprod{\tilde{v}_n-\tilde{v}^{\hom}}{\sin(\uppi x)}_{\Lp{2}}}$}
\addplot [color=mycolor5, line width=2.0pt]
  table[row sep=crcr]{%
 1	2.1389e-01\\
 2	1.0731e-01\\
 3	5.3900e-02\\
 4	2.7011e-02\\
 5	1.3504e-02\\
 6	6.7320e-03\\
 7	3.3413e-03\\
 8	1.6448e-03\\
 9	7.9617e-04\\
10	3.7180e-04\\
};
\addlegendentry{$\abs{\scprod{\tilde{v}_n-\tilde{v}^{\hom}}{t}_{\Lp{2}}}$}
\addplot [color=black, dotted]
  table[row sep=crcr]{%
 1	9e-01\\
10	1.7578e-03\\
};
\addlegendentry{$\ord{n^{-1}}$}
\end{axis}
\end{tikzpicture}%

  \caption{\label{fig:1D-glueing-example-weak}Results investigating weak convergence for \Cref{ex:pde-without-compactness}~\ref{item:1D-glueing-example}}
\end{figure}
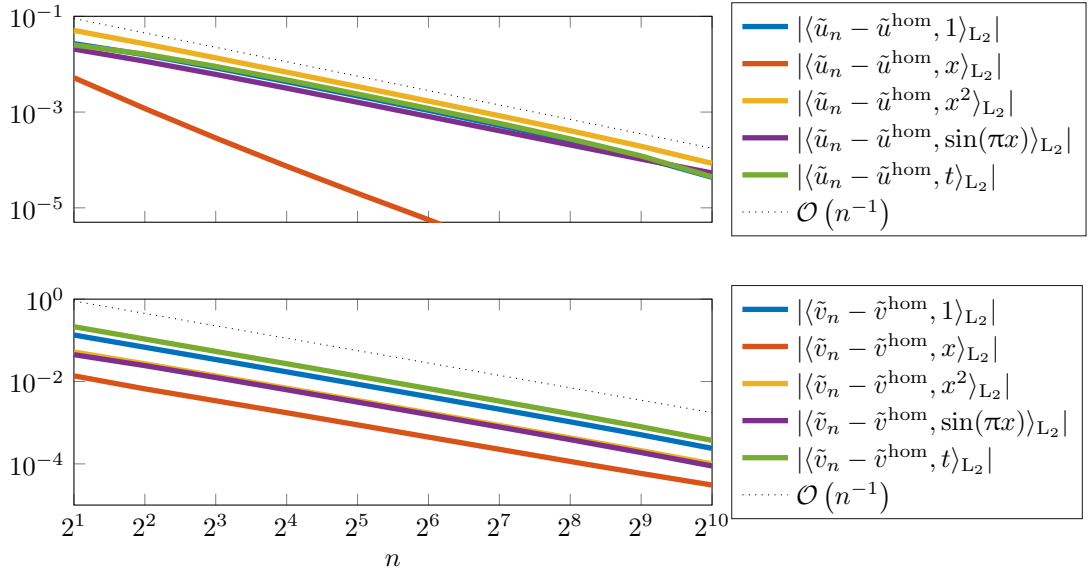
\Cref{fig:1D-glueing-example-weak}
shows the results of the simulations for $u$ in the upper and $v$ in the lower graphs. Subject to (small) approximation errors elaborated on in the previous example, both components of the solution weakly converge. The plots suggest convergence of order $\ord{n^{-1}}$ for either component.

\subsection{Continuation of \Cref{ex:pde-without-compactness}~\ref{item:ex3}}
We consider the sequence of equations ($n\in\N$)
\begin{equation*}
  (\partial_t M_{0,n} + M_{1,n} + A) U_{n} = \begin{pmatrix}
    \sin(2\uppi t)\\
    0
  \end{pmatrix}
\end{equation*}
with a final time horizon $T= 2$.

Again, we want to confirm the convergence behaviour of $U_n=(u_n,v_n)$ and its components. Our theoretical finding predicts strong convergence for $u_n$ and weak convergence for $v_n$.  For the numerical method \eqref{eq:numerical-formulation}, we use
piecewise polynomials of degree 2 in space and one in time that are continuous in
space -- and therefore $\sobH^{1}$-conforming -- for the first component. For the second component, we use
Raviart--Thomas-elements $RT_1$ of degree 1 in space and 1 in time -- these are
$\sobH(\div)$-conforming.
As mesh parameter, we choose $\mbound=2^6$ and a tensor product mesh in space consisting of
a piecewise equidistant mesh with meshsize $1/(4n)$ in $[-1,1]$ and $1/n$ outside for the $x$-direction and
an equidistant mesh with 80 cells in the $y$-direction. We fix the mesh in $y$-direction as the effects are anticipated to occur only in $x$-direction. Thus, the number of cells in total is $10\cdot n\cdot 80$. Note that the spatial mesh resolves the stratified structure of the coefficients.
\begin{figure}
  \includegraphics[width=0.24\textwidth]{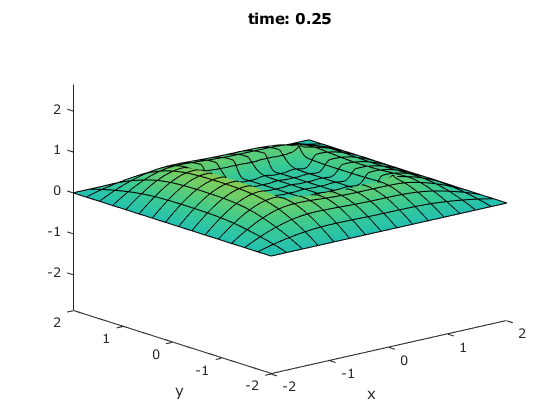}
  \includegraphics[width=0.24\textwidth]{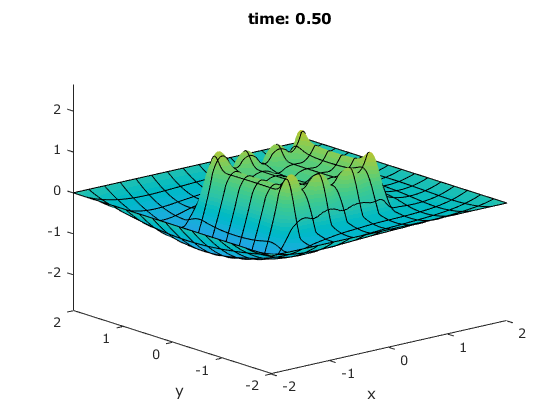}
  \includegraphics[width=0.24\textwidth]{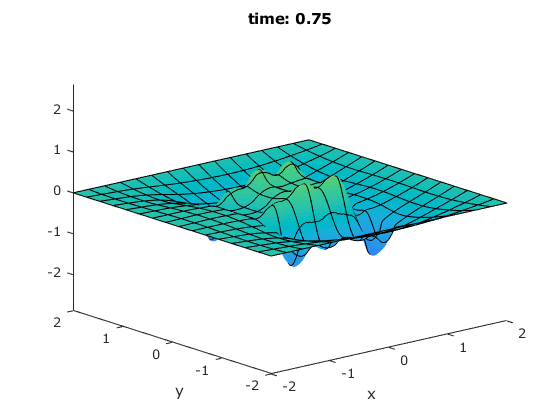}
  \includegraphics[width=0.24\textwidth]{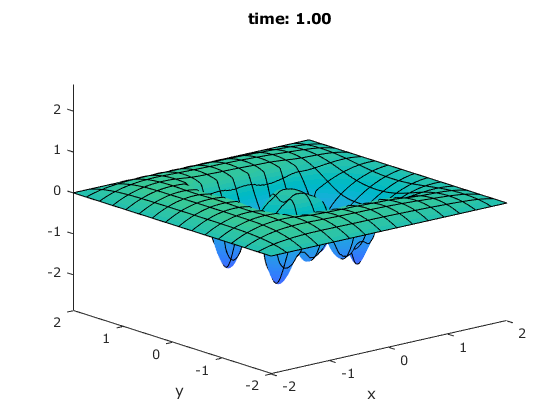}\\
  \includegraphics[width=0.24\textwidth]{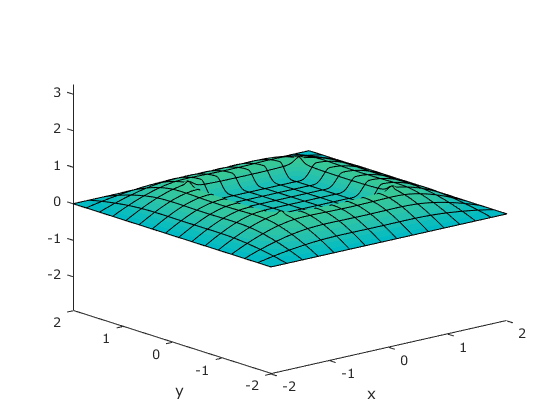}
  \includegraphics[width=0.24\textwidth]{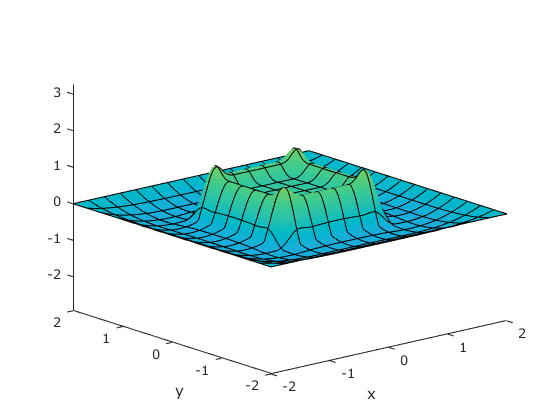}
  \includegraphics[width=0.24\textwidth]{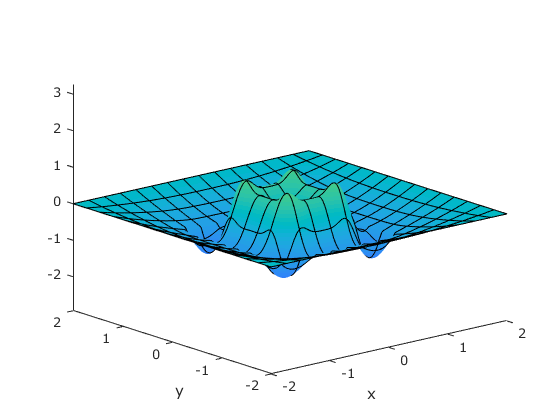}
  \includegraphics[width=0.24\textwidth]{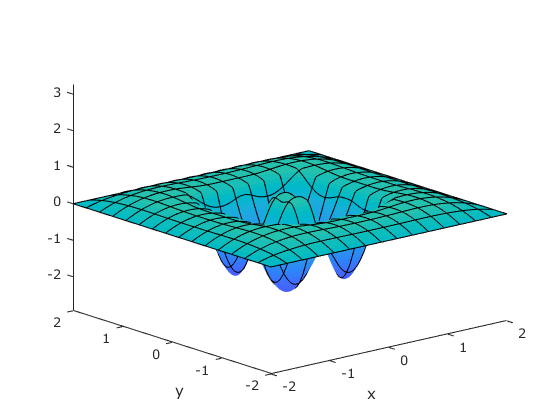}
  \caption{\label{fig:ex3}Plots of the first component of the solution of \Cref{ex:pde-without-compactness}~\ref{item:ex3} $u_{8}$ (top)
    and $u^{\hom}$ (bottom) for $t\in\{0.25,0.5,0.75,1.0\}$ (left to right)}
\end{figure}

For $t\in\{0.25,0.5,0.75,1.0\}$, \Cref{fig:ex3} shows plots of the first component of $u_{8}$ of $U_8$ and $u^{\hom}$ of $U^{\hom}$.
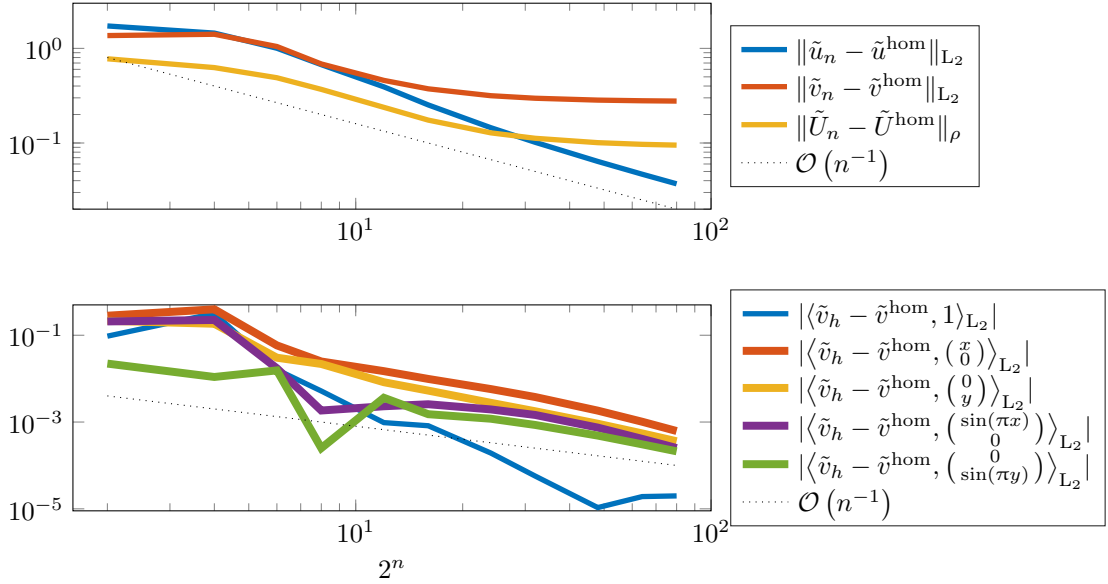
\begin{figure}[htbp]
%
%
\definecolor{mycolor1}{rgb}{0.00000,0.44700,0.74100}%
\definecolor{mycolor2}{rgb}{0.85000,0.32500,0.09800}%
\definecolor{mycolor3}{rgb}{0.92900,0.69400,0.12500}%
\definecolor{mycolor4}{rgb}{0.49400,0.18400,0.55600}%
\definecolor{mycolor5}{rgb}{0.46600,0.67400,0.18800}%
\begin{tikzpicture}

\begin{axis}[%
width=0.615\textwidth,
height=0.2\textwidth,
scale only axis,
xmode=log,
xmin=1.6,
xmax=100,
ymode=log,
ymin=2e-02,
ymax=3,
yminorticks=true,
axis background/.style={fill=white},
legend style={at={(1.03,0.5)}, anchor=west, legend cell align=left, align=left, draw=white!15!black}
]
\addplot [color=mycolor1, line width=2.0pt]
  table[row sep=crcr]{%
  2	1.725e+00\\
  4	1.450e+00\\
  6	1.003e+00\\
  8	6.728e-01\\
 12	3.921e-01\\
 16	2.525e-01\\
 24	1.451e-01\\
 32	1.017e-01\\
 48	6.385e-02\\
 64	4.678e-02\\
 80	3.703e-02\\
};
\addlegendentry{$\norm{\tilde u_n-\tilde u^{\hom}}_{\Lp{2}}$}
\addplot [color=mycolor2, line width=2.0pt]
  table[row sep=crcr]{%
  2	1.365e+00\\
  4	1.409e+00\\
  6	1.046e+00\\
  8	6.843e-01\\
 12	4.587e-01\\
 16	3.736e-01\\
 24	3.160e-01\\
 32	2.971e-01\\
 48	2.839e-01\\
 64	2.794e-01\\
 80	2.773e-01\\
};
\addlegendentry{$\norm{\tilde v_n-\tilde v^{\hom}}_{\Lp{2}}$}
\addplot [color=mycolor3, line width=2.0pt]
  table[row sep=crcr]{%
  2	7.761e-01\\
  4	6.262e-01\\
  6	4.892e-01\\
  8	3.687e-01\\
 12	2.381e-01\\
 16	1.745e-01\\
 24	1.280e-01\\
 32	1.121e-01\\
 48	1.008e-01\\
 64	9.678e-02\\
 80	9.493e-02\\
};
\addlegendentry{$\norm{\tilde U_n-\tilde U^{\hom}}_{\rho}$}
\addplot [color=black, dotted]
  table[row sep=crcr]{%
  2	0.80\\
 80	0.02\\
};
\addlegendentry{$\ord{n^{-1}}$}
\end{axis}

\begin{axis}[%
at={(0,-4cm)},
width=0.615\textwidth,
height=0.2\textwidth,
scale only axis,
xmode=log,
xmin=1.6,
xmax=100,
xlabel = $n$,
ymode=log,
ymin=9e-06,
ymax=5e-1,
yminorticks=true,
axis background/.style={fill=white},
legend style={at={(1.03,0.5)}, anchor=west, legend cell align=left, align=left, draw=white!15!black}
]
\addplot [color=mycolor1, line width=2.0pt]
  table[row sep=crcr]{%
 2	9.515e-02\\    
 4	3.185e-01\\    
 6	1.600e-02\\    
 8	5.185e-03\\    
12	9.700e-04\\    
16	8.214e-04\\    
24	1.930e-04\\    
32	5.692e-05\\    
48	1.069e-05\\    
64	1.917e-05\\    
80	1.987e-05\\    
};
\addlegendentry{$\abs{\scprod{\tilde{v}_h-\tilde v^{\hom}}{1}_{\Lp{2}}}$}
\addplot [color=mycolor2, line width=3.0pt]
  table[row sep=crcr]{%
 2	2.819e-01\\
 4	3.957e-01\\
 6	5.831e-02\\
 8	2.482e-02\\
12	1.488e-02\\
16	9.841e-03\\
24	5.751e-03\\
32	3.769e-03\\
48	1.832e-03\\
64	1.016e-03\\
80	6.212e-04\\
};
\addlegendentry{$\abs{\scprod*{\tilde{v}_h-\tilde v^{\hom}}{\left(\begin{smallmatrix}x\\0\end{smallmatrix}\right)}_{\Lp{2}}}$}
\addplot [color=mycolor3, line width=3.0pt]
  table[row sep=crcr]{%
 2	2.150e-01\\
 4	1.837e-01\\
 6	3.035e-02\\
 8	2.176e-02\\
12	8.273e-03\\
16	5.299e-03\\
24	2.771e-03\\
32	1.772e-03\\
48	9.142e-04\\
64	5.472e-04\\
80	3.610e-04\\
};
\addlegendentry{$\abs{\scprod*{\tilde{v}_h-\tilde v^{\hom}}{\left(\begin{smallmatrix}0\\y\end{smallmatrix}\right)}_{\Lp{2}}}$}
\addplot [color=mycolor4, line width=3.0pt]
  table[row sep=crcr]{%
 2	2.079e-01\\
 4	2.248e-01\\
 6	1.729e-02\\
 8	1.844e-03\\
12	2.308e-03\\
16	2.570e-03\\
24	1.953e-03\\
32	1.456e-03\\
48	7.445e-04\\
64	4.093e-04\\
80	2.518e-04\\
};
\addlegendentry{$\abs{\scprod*{\tilde{v}_h-\tilde v^{\hom}}{\left(\begin{smallmatrix}\sin(\uppi x)\\0\end{smallmatrix}\right)}_{\Lp{2}}}$}
\addplot [color=mycolor5, line width=3.0pt]
  table[row sep=crcr]{%
 2	2.220e-02\\
 4	1.092e-02\\
 6	1.548e-02\\
 8	2.509e-04\\
12	3.620e-03\\
16	1.516e-03\\
24	1.192e-03\\
32	8.566e-04\\
48	4.870e-04\\
64	3.064e-04\\
80	2.117e-04\\
};
\addlegendentry{$\abs{\scprod*{\tilde{v}_h-\tilde v^{\hom}}{\left(\begin{smallmatrix}0\\\sin(\uppi y)\end{smallmatrix}\right)}_{\Lp{2}}}$}
\addplot [color=black, dotted]
  table[row sep=crcr]{%
 2	4e-3\\
80	1e-4\\
};
\addlegendentry{$\ord{n^{-1}}$}
\end{axis}
\end{tikzpicture}%

  \caption{\label{fig:ex3_conv}Results investigating convergence for \Cref{ex:pde-without-compactness}~\ref{item:ex3}}
\end{figure}
%
%
\Cref{fig:ex3_conv} investigates the convergence in the classical $\Lp{2}$-norm,
and also weak convergence.
We observe convergence of order $\ord{n^{-1}}$ for the $\Lp{2}$-norm of the first component, but a
stagnation in the norm-convergence of the second component and consequently of the weighted norm. Still, the second component converges weakly. Again, the graphs suggest weak convergence of order $\ord{n^{-1}}$.

\subsection{Continuation of \Cref{ex:pde-without-compactness}~\ref{item:ex4}}
We once again consider the sequence of equations ($n\in\N$)
\begin{equation*}
  (\partial_t M_{0,n} + M_{1,n} + A) U_{n} = \begin{pmatrix}
    \sin(2\uppi t)\\
    0
  \end{pmatrix}
\end{equation*}
with a final time horizon $T= 2$.
A direct approach incorporating the non-local operator in $M_1$ (see~\labelcref{eq:memoryterminhomlimit1}) necessitates the inclusion of history terms as right-hand data while evaluating the time steps. This is potentially numerically cumbersome. Fortunately, we can deal with it differently:

Let us define an intrinsic variable $w$ with $w(t,x)=0$ for $t\leq 0$ by
\begin{equation*}
  (1+\partial_{t})^{-1} v_{2} = w
  \quad\Leftrightarrow\quad
  v_{2} = (1 + \partial_{t})w
  \quad\Leftrightarrow\quad
  \partial_{t} w + w - v_{2} = 0.
\end{equation*}
With this we rewrite the homogenised equation of \Cref{ex:pde-without-compactness-limits}~\ref{item:ex4-limit} as
\begin{align}\label{eq:ex4_hom2}
  (\partial_t \widehat M_0+\widehat M_1+\widehat A)\widehat U^{\hom} &= \widehat{F},
\end{align}
where $\widehat U^{\hom} \coloneqq (u^{\hom},v^{\hom},w^{\hom})\transposed$, $\widehat F \coloneqq (f,g,0)\transposed$ for $F = (f,g)\transposed$ and
\begin{align*}
  \widehat{M}_{0}
  &\coloneqq\indicator_{\Omega_{1}}
    \begin{pmatrix}
      \frac{3}{2} & 0 & 0\\
      0 & \begin{pmatrix}
            \frac{1}{2} & 0\\
            0 & 0
          \end{pmatrix}& 0\\
      0 & 0 & 2
    \end{pmatrix}
    + (1 - \indicator_{\Omega_{1}})
    \begin{pmatrix}
      \epsilon_0 & 0     & 0\\
      0 & \mu_0 & 0\\
      0 & 0 & 0
    \end{pmatrix},\\
  \widehat{M}_{1}
  &\coloneqq\indicator_{\Omega_{1}}
    \begin{pmatrix}
      0 & 0 & 0\\
      0 &
          \begin{pmatrix}
            \frac{1}{2} & 0\\
            0 & 2
          \end{pmatrix}
      &
        \begin{pmatrix}0\\-2\end{pmatrix}\\
      0 & \begin{pmatrix}0 & -2\end{pmatrix} & 2
    \end{pmatrix}
    + (1 - \indicator_{\Omega_{1}})
    \begin{pmatrix}
      0 & 0 & 0\\
      0 & 0 & 0\\
      0 & 0 & 1
    \end{pmatrix},
  \\
  \widehat{A}
  &\coloneqq
    \begin{pmatrix}
      0 & \div & 0\\
      \cgrad & 0 & 0\\
      0 & 0 & 0
    \end{pmatrix}.
\end{align*}
We added the equation $(1 - \indicator_{\Omega_{1}})w^{\hom} = 0$ in order to define $w^{\hom}$ outside $\Omega_{1}$.
We can easily verify that \eqref{eq:ex4_hom2} falls under the regime of \Cref{thm:wpee}. This formulation of the homogenised problem does not contain memory terms and
the standard method \eqref{eq:numerical-formulation} developed in \cite{FrTrWa19} applies, if we choose an appropriate function space for the third component and set homogeneous initial conditions. As compatibility in space is not needed, discontinuous, piecewise polynomial finite elements of order one less than for the first component are possible. All in all, numerical costs are higher compared to the direct approach, but the implementation is easier.

\begin{figure}
  \includegraphics[width=0.24\textwidth]{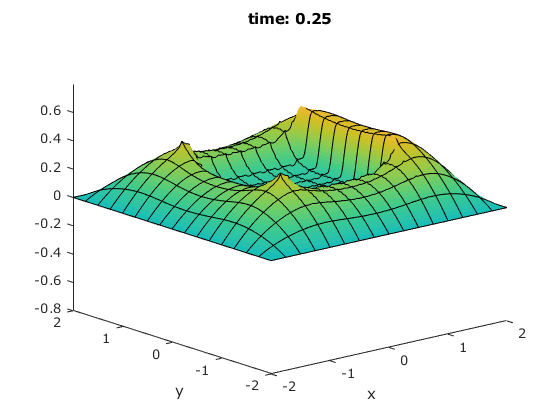}
  \includegraphics[width=0.24\textwidth]{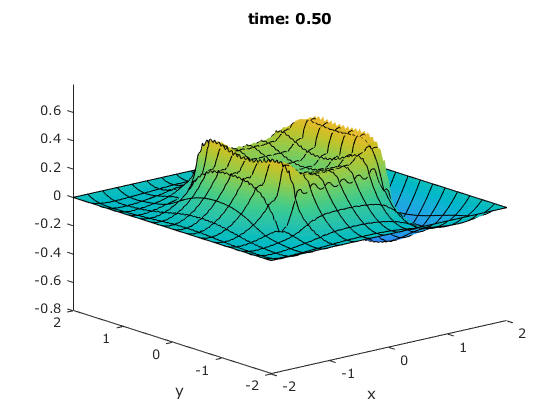}
  \includegraphics[width=0.24\textwidth]{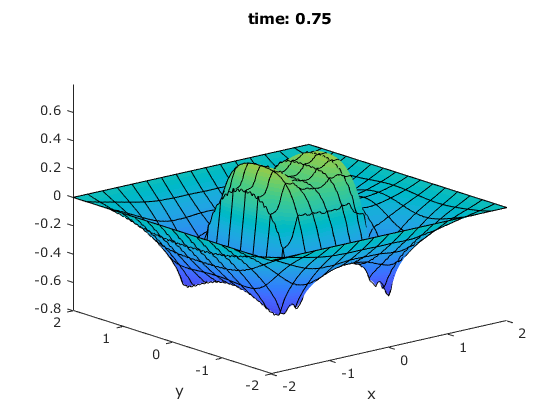}
  \includegraphics[width=0.24\textwidth]{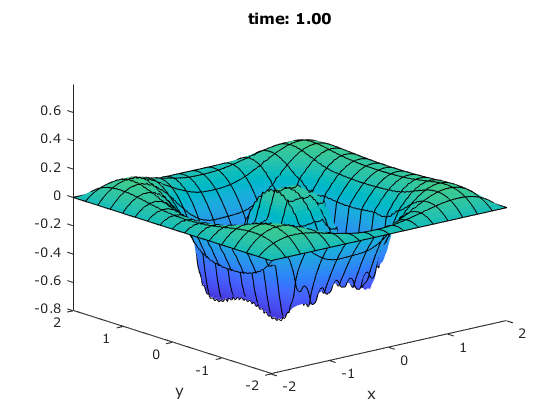}\\
  \includegraphics[width=0.24\textwidth]{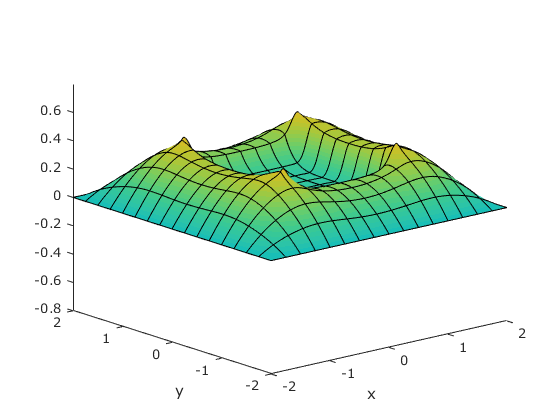}
  \includegraphics[width=0.24\textwidth]{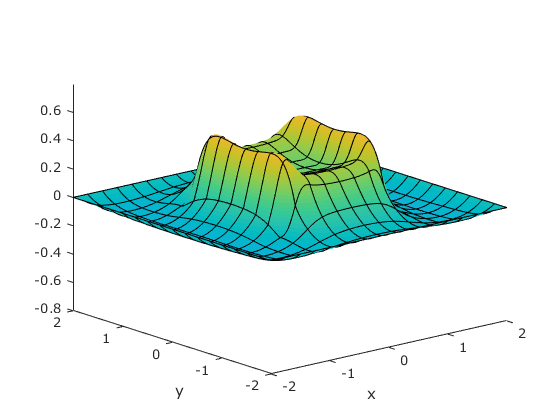}
  \includegraphics[width=0.24\textwidth]{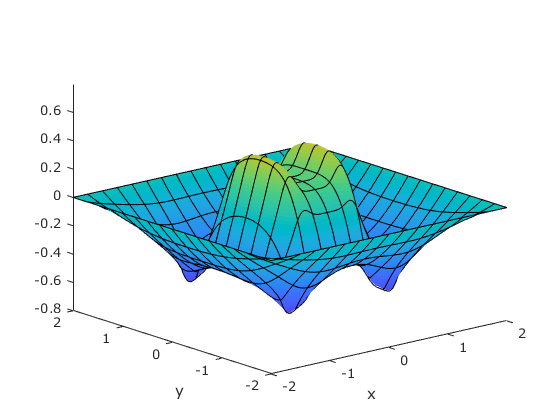}
  \includegraphics[width=0.24\textwidth]{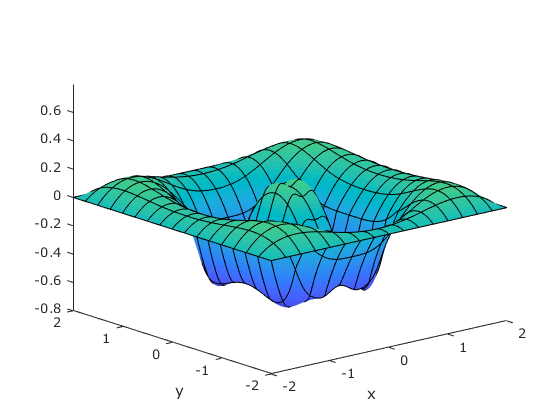}\\
  \includegraphics[width=0.24\textwidth]{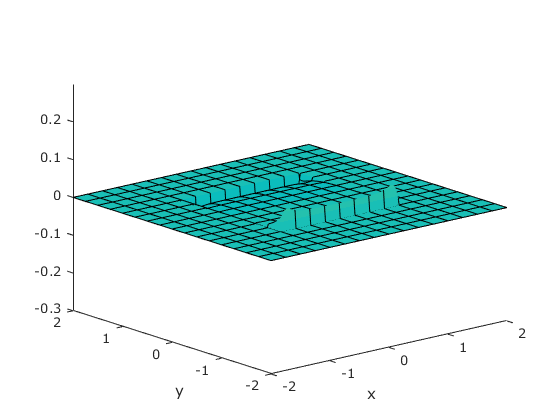}
  \includegraphics[width=0.24\textwidth]{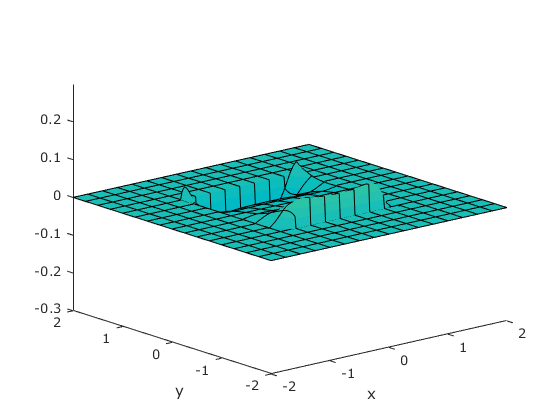}
  \includegraphics[width=0.24\textwidth]{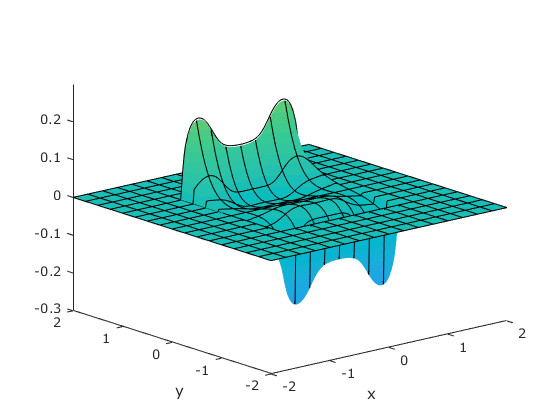}
  \includegraphics[width=0.24\textwidth]{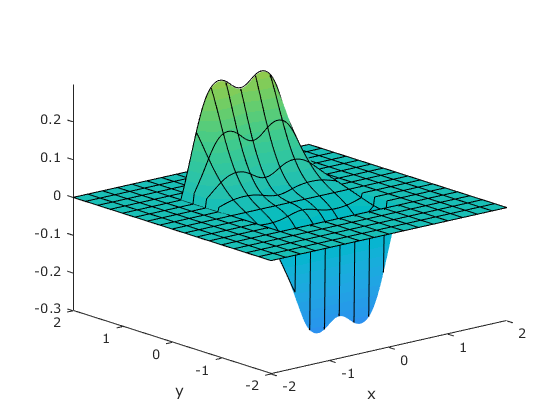}
  \caption{\label{fig:ex4}Plots of the first component of the solution $U_8$ of \Cref{ex:pde-without-compactness}~\ref{item:ex4} (top),
    the first (middle) and third component (bottom) of $\widehat U^{\hom}$ of \eqref{eq:ex4_hom2}
    for $t\in\{0.25,0.5,0.75,1.0\}$ (left to right)}
\end{figure}
For $t\in\{0.25,0.5,0.75,1.0\}$, \Cref{fig:ex4} shows plots of the first component of $U_8$ and
of $U^{\hom}$. 
In addition, the pictures in the bottom row exemplify the behaviour of the intrinsic variable $w^{\hom}$, i.e., the memory term, which apparently may not be neglected.

\begin{figure}[htbp]
%
%
\definecolor{mycolor1}{rgb}{0.00000,0.44700,0.74100}%
\definecolor{mycolor2}{rgb}{0.85000,0.32500,0.09800}%
\definecolor{mycolor3}{rgb}{0.92900,0.69400,0.12500}%
\definecolor{mycolor4}{rgb}{0.49400,0.18400,0.55600}%
\definecolor{mycolor5}{rgb}{0.46600,0.67400,0.18800}%
\begin{tikzpicture}

\begin{axis}[%
width=0.615\textwidth,
height=0.2\textwidth,
scale only axis,
xmode=log,
xmin=1.6,
xmax=110,xticklabels={,10,100},
ymode=log,
ymin=1e-02,
ymax=2,
yminorticks=true,
axis background/.style={fill=white},
legend style={at={(1.03,0.5)}, anchor=west, legend cell align=left, align=left, draw=white!15!black}
]
\addplot [color=mycolor1, line width=2.0pt]
  table[row sep=crcr]{%
  2	7.338e-01\\
  4	6.983e-01\\
  6	4.625e-01\\
  8	3.342e-01\\
 12	2.435e-01\\
 16	1.625e-01\\
 24	8.797e-02\\
 32	6.291e-02\\
 48	4.083e-02\\
 64	3.049e-02\\
 80	2.452e-02\\
100	1.985e-02\\
};
\addlegendentry{$\norm{\tilde u_n-\tilde u^{\hom}}_{\Lp{2}}$}
\addplot [color=mycolor2, line width=2.0pt]
  table[row sep=crcr]{%
  2	1.276e+00\\
  4	1.365e+00\\
  6	1.104e+00\\
  8	9.899e-01\\
 12	9.225e-01\\
 16	8.971e-01\\
 24	8.773e-01\\
 32	8.711e-01\\
 48	8.665e-01\\
 64	8.648e-01\\
 80	8.639e-01\\ 
100	8.633e-01\\ 
};
\addlegendentry{$\norm{\tilde v_n-\tilde v^{\hom}}_{\Lp{2}}$}
\addplot [color=mycolor3, line width=2.0pt]
  table[row sep=crcr]{%
  2	4.879e-01\\
  4	5.071e-01\\
  6	4.299e-01\\
  8	3.872e-01\\
 12	3.518e-01\\
 16	3.324e-01\\
 24	3.186e-01\\
 32	3.148e-01\\
 48	3.122e-01\\
 64	3.113e-01\\
 80	3.108e-01\\ 
100	3.105e-01\\ 
};
\addlegendentry{$\norm{\tilde U_n- \begin{psmallmatrix} \tilde{u}^{\hom} \\ \tilde{v}^{\hom} \end{psmallmatrix}}_{\rho}$}
\addplot [color=black, dotted]
  table[row sep=crcr]{%
  2	0.50\\
100	0.01\\
};
\addlegendentry{$\ord{n^{-1}}$}
\end{axis}

\begin{axis}[%
at={(0,-4cm)},
width=0.615\textwidth,
height=0.2\textwidth,
scale only axis,
xmode=log,
xmin=1.6,
xmax=110,xticklabels={,10,100},
xlabel = $n$,
ymode=log,
ymin=8e-05,
ymax=3e-1,
yminorticks=true,
axis background/.style={fill=white},
legend style={at={(1.03,0.5)}, anchor=west, legend cell align=left, align=left, draw=white!15!black}
]
\addplot [color=mycolor1, line width=2.0pt]
  table[row sep=crcr]{%
  2	7.601e-03\\
  4	1.407e-01\\
  6	5.991e-02\\
  8	3.611e-02\\
 12	1.941e-02\\
 16	1.291e-02\\
 24	8.345e-03\\
 32	6.189e-03\\
 48	4.095e-03\\
 64	3.063e-03\\
 80	2.447e-03\\ 
100	1.956e-03\\ 
};
\addlegendentry{$\abs{\scprod{\tilde{v}_n-\tilde v^{\hom}}{1}_{\Lp{2}}}$}
\addplot [color=mycolor2, line width=2.0pt]
  table[row sep=crcr]{%
  2	6.380e-02\\
  4	6.953e-02\\
  6	3.241e-02\\
  8	4.265e-03\\
 12	5.946e-03\\
 16	3.171e-03\\
 24	2.960e-03\\
 32	2.289e-03\\
 48	1.479e-03\\
 64	1.099e-03\\
 80	9.034e-04\\ 
100	7.743e-04\\ 
};
\addlegendentry{$\abs{\scprod{\tilde{v}_n-\tilde v^{\hom}}{\begin{psmallmatrix}x\\0\end{psmallmatrix}}_{\Lp{2}}}$}
\addplot [color=mycolor3, line width=2.0pt]
  table[row sep=crcr]{%
  2	4.939e-03\\
  4	1.100e-02\\
  6	3.262e-03\\
  8	1.045e-02\\
 12	9.328e-03\\
 16	6.517e-03\\
 24	4.915e-03\\
 32	3.791e-03\\
 48	2.627e-03\\
 64	2.076e-03\\
 80	1.777e-03\\ 
100	1.567e-03\\ 
};
\addlegendentry{$\abs{\scprod*{\tilde{v}_n-\tilde v^{\hom}}{\begin{psmallmatrix}0\\y\end{psmallmatrix}}_{\Lp{2}}}$}
\addplot [color=mycolor4, line width=2.0pt]
  table[row sep=crcr]{%
  2	1.184e-01\\
  4	2.962e-02\\
  6	2.359e-02\\
  8	1.760e-03\\
 12	3.613e-04\\
 16	1.769e-03\\
 24	8.541e-04\\
 32	6.711e-04\\
 48	6.025e-04\\
 64	5.463e-04\\
 80	4.980e-04\\ 
100	4.554e-04\\ 
};
\addlegendentry{$\abs{\scprod*{\tilde{v}_n-\tilde v^{\hom}}{\begin{psmallmatrix}\sin(\pi x) \\ 0\end{psmallmatrix}}_{\Lp{2}}}$}
\addplot [color=mycolor5, line width=2.0pt]
  table[row sep=crcr]{%
  2	3.214e-02\\
  4	1.483e-02\\
  6	2.004e-03\\
  8	1.034e-04\\
 12	1.414e-03\\
 16	3.112e-04\\
 24	5.711e-04\\
 32	4.039e-04\\
 48	2.557e-04\\
 64	1.974e-04\\
 80	1.655e-04\\
100	1.413e-04\\
};
\addlegendentry{$\abs{\scprod*{\tilde{v}_n-\tilde v^{\hom}}{\begin{psmallmatrix}0 \\ \sin(\uppi y)\end{psmallmatrix}}_{\Lp{2}}}$}
\addplot [color=black, dotted]
  table[row sep=crcr]{%
  2	5e-3\\
100	1e-4\\
};
\addlegendentry{$\ord{n^{-1}}$}
\end{axis}
\end{tikzpicture}%

  \caption{\label{fig:ex4_conv}Results investigating convergence for \Cref{ex:pde-without-compactness}~\ref{item:ex4} and \eqref{eq:ex4_hom2}}
\end{figure}
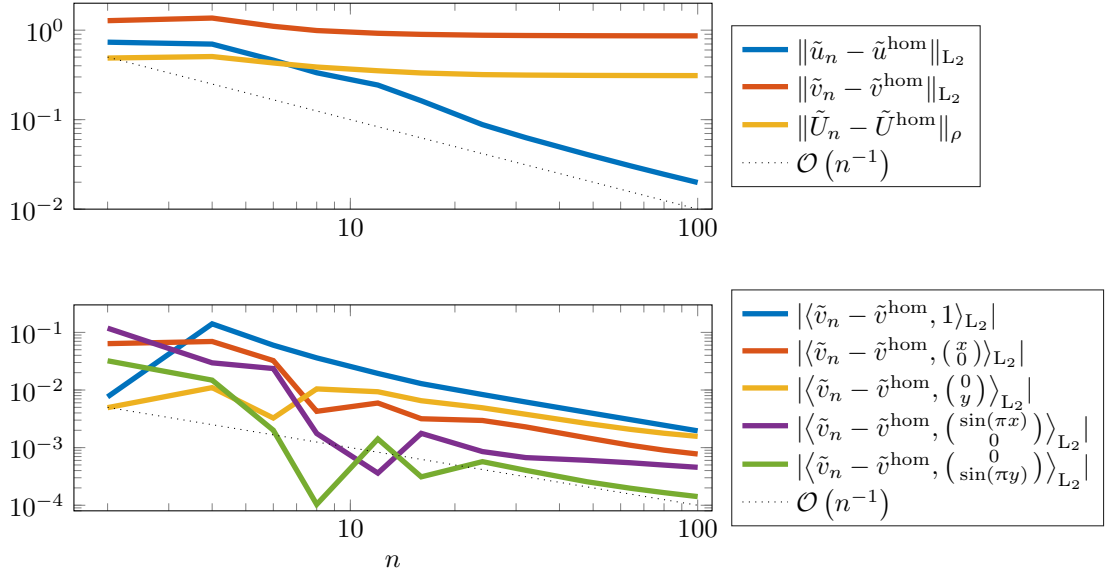
\Cref{fig:ex4_conv} shows the convergence results of the numerical simulations. In the computation of $\tilde U_n$ using \eqref{eq:numerical-formulation}, we used a piecewise polynomials of one degree
higher than in the previous calculations, because the numerical solution was still
very oscillatory using quadratic elements for $\tilde u_n$. Using higher order finite element methods can provide a
natural stabilisation that reduces the oscillations, see e.g., \cite{BR94}, and we can observe first order
convergence in the $\Lp{2}$-norm of $u_n$ to $u^{\hom}$ in the upper graph of \Cref{fig:ex4_conv}. As in the previous example, $v_n$ does not
converge strongly in the $\Lp{2}$-norm to $v^{\hom}$, but weakly as depicted in the lower graph.

\subsection{Continuation of \Cref{ex:pde-without-compactness}~\ref{item:ex5}}
We once again consider the sequence of equations ($n\in\N$)
\begin{equation*}
  (\partial_t M_{0,n} + M_{1,n} + A) U_{n} = \begin{pmatrix}
    \sin(2\uppi t)\\
    0
  \end{pmatrix}
\end{equation*}
with a final time horizon $T= 2$.
According to \Cref{ex:pde-without-compactness-limits}~\ref{item:ex5-limit} its homogenised limit has a memory term.
In the same way as in the previous example, we can introduce an intrinsic variable $v = \sigma \eps (\sigma + \eps\partial_{t})^{-1} E_{1}$
and rewrite the system in the form
\begin{align}\label{eq:ex5_hom2}
  (\partial_t \widehat M_0+\widehat M_1+\widehat A)\widehat U^{\hom}&=\widehat F,
\end{align}
where $\widehat U^{\hom} = (E^{\hom},H^{\hom},v^{\hom})\transposed$, $\widehat F = (J,K,0)\transposed$ for $F = (J,K)\transposed$ and
\begin{align*}
  \widehat M_0
  &= \indicator_{\Omega_{1}}
    \begin{pNiceMatrix}[left-margin=0.7em]
      0 & 0 & 0 & & & & \\
      0 & \tfrac{1}{2} \epsilon & 0 & & 0 & & 0 \\
      0 & 0 & \tfrac{1}{2} \epsilon & & & &  \\
        &   & & \tfrac{4}{3} \mu & 0 & 0 &\\
        & 0 & & 0 & \tfrac{3}{2}\mu & 0 & 0\\
        &   & & 0 & 0 & \tfrac{3}{2}\mu & \\
        & 0  & & & 0& & 2\epsilon
      \CodeAfter
      \SubMatrix({1-1}{3-3})
      \SubMatrix({4-4}{6-6})
    \end{pNiceMatrix}
    + (1-\indicator_{\Omega_{1}})
    \begin{pmatrix}
      \epsilon_0 & 0&0\\
      0 & \mu_0&0\\
      0&0&0
    \end{pmatrix},
  \\
  \widehat M_1
  &= \indicator_{\Omega_{1}}
    \begin{pNiceMatrix}[left-margin=0.7em, right-margin=0.7em]
      2\sigma & 0 & 0 & \phantom{000}& -2\sigma \\
      0 & \frac{1}{2}\sigma & 0 & 0 & 0 \\
      0 & 0 & \frac{1}{2}\sigma & & 0 \\
      & 0 & & 0 & 0 \\
      -2\sigma & 0 & 0 & 0 & 2\sigma 
      \CodeAfter
      \SubMatrix({1-1}{3-3})
      \SubMatrix({1-5}{3-5})
      \SubMatrix({5-1}{5-3})
    \end{pNiceMatrix}
    + (1-\indicator_{\Omega_{1}})
    \begin{pmatrix}
      0&0&0\\
      0&0&0\\
      0&0&1
    \end{pmatrix},\\
  \widehat A
  &= \begin{pmatrix}
       0&-\curl&0\\
       \ccurl&0&0\\
       0&0&0
     \end{pmatrix}.
\end{align*}
Once again, we have added the equation $(1-\indicator_{\Omega_{1}})v^{\hom} = 0$
in order to define $v^{\hom}$ everywhere, and, once again, \Cref{thm:wpee}
renders \eqref{eq:ex5_hom2} uniquely solvable. For its numerical simulation, we
observe the curse of dimensions and cannot use as large numbers of $n$ in \eqref{eq:numerical-formulation} as for the previous problems.
Again, we use a reference solution with
a fine mesh instead of a given exact solution to the homogenised problem. Just to give an impression of its size: its number of degrees of freedom is 7 millions,
9 millions and 12 millions for the three respective components. For the stratified problem, the finest resolution
had 31 millions and 40 millions degrees of freedom in its two components.
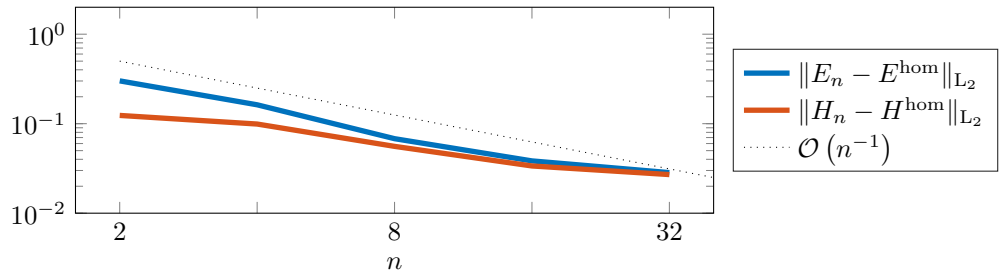
\begin{figure}[htbp]
%
%
\definecolor{mycolor1}{rgb}{0.00000,0.44700,0.74100}%
\definecolor{mycolor2}{rgb}{0.85000,0.32500,0.09800}%
\definecolor{mycolor3}{rgb}{0.92900,0.69400,0.12500}%
\definecolor{mycolor4}{rgb}{0.49400,0.18400,0.55600}%
\definecolor{mycolor5}{rgb}{0.46600,0.67400,0.18800}%
\begin{tikzpicture}

\begin{axis}[%
width=0.615\textwidth,
height=0.2\textwidth,
scale only axis,
xmode=log,
xmin=1.6,
xmax=40,xtick={2,4,8,16,32},xticklabels={2,,8,,32},
xlabel = $n$,
ymode=log,
ymin=1e-02,
ymax=2,
yminorticks=true,
axis background/.style={fill=white},
legend style={at={(1.03,0.5)}, anchor=west, legend cell align=left, align=left, draw=white!15!black}
]
\addplot [color=mycolor1, line width=2.0pt]
  table[row sep=crcr]{%
  2	0.3009\\
  4	0.1626\\
  8	0.0680\\
 16	0.0384\\
 32	0.0284\\
};
\addlegendentry{$\norm{E_n-E^{\hom}}_{\Lp{2}}$}
\addplot [color=mycolor2, line width=2.0pt]
  table[row sep=crcr]{%
  2	0.1237\\
  4	0.0991\\
  8	0.0558\\
 16	0.0337\\
 32	0.0270\\
};
\addlegendentry{$\norm{H_n-H^{\hom}}_{\Lp{2}}$}
\addplot [color=black, dotted]
  table[row sep=crcr]{%
  2	0.50\\
100	0.01\\
};
\addlegendentry{$\ord{n^{-1}}$}
\end{axis}

\end{tikzpicture}%

  \caption{\label{fig:ex5a_conv}Computed errors for the full 3D Maxwell example \eqref{eq:ex5_hom2}.}
\end{figure}
The values for the differences are given in \Cref{fig:ex5a_conv}. Both components suggest a convergence in $n$ to the solution of the homogenised problem, however, with an order less than 1.
Whether or not this behaviour can be confirmed will be further investigated looking at weak convergence. Even though we took the following small sample of test functions
\begin{align*}
  v_1(t,\vecsymb{x}) &= \pmtrx{1 \\ 0 \\ 0}, & v_2(t,\vecsymb{x}) &= \pmtrx{0 \\ 1 \\ 0}, & v_3(t,\vecsymb{x}) &= \pmtrx{0 \\ 0 \\ 1},
  \\
  v_4(t,\vecsymb{x}) &= \pmtrx{\sin(\uppi x)\\0\\0},& v_5(t,\vecsymb{x}) &= \pmtrx{0 \\
  \sin(\uppi y)\\0},& v_6(t,\vecsymb{x})&=\pmtrx{0\\0\\\sin(\uppi z)},
  \\
  v_7(t,\vecsymb{x}) &= \pmtrx{xy\\0\\0},& v_8(t,\vecsymb{x}) &= \pmtrx{0\\y^2+z\\0}, & v_9(t,\vecsymb{x})&=\pmtrx{0\\0\\xyz},
\end{align*}
\begin{figure}[h]
%
%
\definecolor{mycolor1}{rgb}{0.00000,0.44700,0.74100}%
\definecolor{mycolor2}{rgb}{0.85000,0.32500,0.09800}%
\definecolor{mycolor3}{rgb}{0.92900,0.69400,0.12500}%
\definecolor{mycolor4}{rgb}{0.49400,0.18400,0.55600}%
\definecolor{mycolor5}{rgb}{0.46600,0.67400,0.18800}%
\begin{tikzpicture}

\begin{axis}[%
at={(0,-4cm)},
width=0.615\textwidth,
height=0.2\textwidth,
scale only axis,
xmode=log,
xmin=1.6,
xmax=40,xtick={2,4,8,16,32},xticklabels={2,,8,,32},
xlabel = $n$,
ymode=log,
ymin=1e-05,
ymax=3e-1,
yminorticks=true,
axis background/.style={fill=white},
legend style={at={(1.03,0.5)}, anchor=west, legend cell align=left, align=left, draw=white!15!black}
]
\addplot [color=mycolor1, line width=2.0pt]
  table[row sep=crcr]{%
  2	3.739e-03\\
  4	3.350e-03\\
  8	3.383e-03\\
 16	3.162e-03\\
 32	3.221e-03\\
};
\addlegendentry{$\abs{\scprod{E_n-E^{\hom}}{v_1}_{\Lp{2}}}$}
\addplot [color=mycolor2, line width=2.0pt]
  table[row sep=crcr]{%
  2	1.057e-04\\
  4	2.222e-04\\
  8	5.866e-04\\
 16	7.763e-04\\
 32	8.122e-04\\
};
\addlegendentry{$\abs{\scprod{E_n-E^{\hom}}{v_2}_{\Lp{2}}}$}
\addplot [color=mycolor3, line width=2.0pt]
  table[row sep=crcr]{%
  2	5.754e-02\\
  4	1.145e-03\\
  8	7.631e-03\\
 16	5.851e-03\\
 32	6.066e-03\\
};
\addlegendentry{$\abs{\scprod{E_n-E^{\hom}}{v_3}_{\Lp{2}}}$}
\addplot [color=mycolor4, line width=2.0pt]
  table[row sep=crcr]{%
  2	2.767e-04\\
  4	3.166e-04\\
  8	4.963e-04\\
 16	5.470e-05\\
 32	1.738e-04\\
};
\addlegendentry{$\abs{\scprod{E_n-E^{\hom}}{v_4}_{\Lp{2}}}$}
\addplot [color=mycolor5, line width=2.0pt]
  table[row sep=crcr]{%
  2	1.500e-03\\
  4	2.430e-04\\
  8	5.559e-04\\
 16	9.850e-04\\
 32	1.218e-03\\
};
\addlegendentry{$\abs{\scprod{E_n-E^{\hom}}{v_5}_{\Lp{2}}}$}
\addplot [color=mycolor1, line width=2.0pt, dashed]
  table[row sep=crcr]{%
  2	8.247e-03\\
  4	1.047e-02\\
  8	2.282e-03\\
 16	7.302e-04\\
 32	1.956e-04\\
};
\addlegendentry{$\abs{\scprod{E_n-E^{\hom}}{v_6}_{\Lp{2}}}$}
\addplot [color=mycolor2, line width=2.0pt, dashed]
  table[row sep=crcr]{%
  2	3.377e-03\\
  4	1.725e-03\\
  8	8.566e-04\\
 16	1.566e-04\\
 32	1.344e-04\\
};
\addlegendentry{$\abs{\scprod{E_n-E^{\hom}}{v_7}_{\Lp{2}}}$}
\addplot [color=mycolor3, line width=2.0pt, dashed]
  table[row sep=crcr]{%
  2	2.999e-03\\
  4	2.227e-03\\
  8	2.449e-03\\
 16	2.451e-03\\
 32	2.375e-03\\
};
\addlegendentry{$\abs{\scprod{E_n-E^{\hom}}{v_8}_{\Lp{2}}}$}
\addplot [color=mycolor4, line width=2.0pt, dashed]
  table[row sep=crcr]{%
  2	2.105e-02\\
  4	1.808e-03\\
  8	6.694e-03\\
 16	6.122e-03\\
 32	5.170e-03\\
};
\addlegendentry{$\abs{\scprod{E_n-E^{\hom}}{v_9}_{\Lp{2}}}$}
\addplot [color=black, dotted]
  table[row sep=crcr]{%
  2	5e-3\\
100	1e-4\\
};
\addlegendentry{$\ord{n^{-1}}$}
\end{axis}
\end{tikzpicture}%

  \caption{\label{fig:ex5b_conv}Investigation of probable weak convergence of the full 3D Maxwell example \eqref{eq:ex5_hom2}.}
\end{figure}
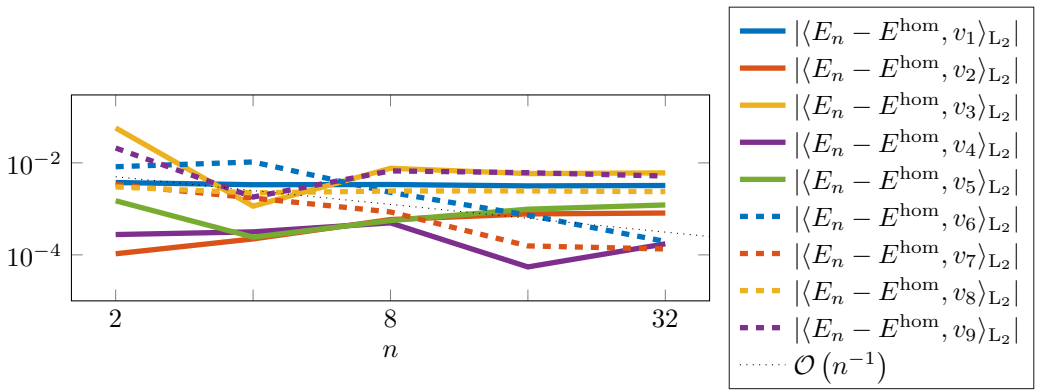
a similar convergence behaviour cannot be confirmed as \Cref{fig:ex5b_conv} shows (the results for $H_n$ are similar). We are probably still far away from the convergent regime, but at the same time at the end of our computational possibilities. Although our theoretical findings assert that the numbers will go down eventually, this did not happen numerically for $n\leq 32$.

\section{Conclusion}\label{sec:con}

We have applied an abstract continuous dependence result to homogenisation problems for ordinary differential equations as well as partial differential equations of potentially mixed type. Further theoretical insight led to norm-convergence statements, where only weak convergence was explicitly known. Numerical findings support these abstract results.

\section*{Acknowledgements}

M.W.\ acknowledges useful discussions with Dirk Pauly and Rainer Picard concerning the Helmholtz decomposition in 2D.

\section*{Data Availability}

The data leading to the present study can be shared upon personal request to S.F.

\appendix

\section{}\label{sec:app}%

\begin{lemma}\label{le:dirich-grad-vanishes-for-skew-selfad}
  For $d\in\N$, an open $\Omega\subseteq\R^{d}$, and a skew-selfadjoint $C\in\R^{d\times d}$,
  \begin{equation*}
    \scprod{C\cgrad u}{\cgrad v}_{\Lp{2}(\Omega)^{d}}=0
  \end{equation*}
  holds for all $u,v\in \cH^{1}(\Omega)$.
\end{lemma}

\begin{proof}
  Clearly, we only have to prove the claim for $u,v\in \Cc (\Omega)$. In that case integration by parts yields ($C=(c_{ij})_{1\leq j,k\leq d}$)
  \begin{equation}\label{eq:dirich-grad-vanishes-for-skew-selfad-test-fct}
    \scprod{C\cgrad u}{\cgrad v}_{\Lp{2}(\Omega)^{d}}=\int_{\Omega} \sum_{j=1}^{d}\sum_{k=1}^{d} c_{jk}\overline{\partial_{k}u}\partial_{j} v \dx
    =-\int_{\Omega} \overline{u}\sum_{j=1}^{d}\sum_{k=1}^{d} c_{jk}\partial_{k}\partial_{j} v \dx
  \end{equation}
  The skew-selfadjointness and Schwarz's theorem imply
  \begin{equation*}
    c_{jk}\partial_{k}\partial_{j}v=-c_{kj}\partial_{k}\partial_{j}v=-c_{kj}\partial_{j}\partial_{k}v
  \end{equation*}
  for $1\leq j,k\leq d$. With that, we can immediately deduce that~\labelcref{eq:dirich-grad-vanishes-for-skew-selfad-test-fct} vanishes.
\end{proof}

\begin{lemma}\label{le:equivalent-M-for-real-matrices}
  Let $a \in \R^{d\times d}$. Then
  \begin{equation*}
    \Big( \forall x \in \R^{d}\colon  \scprod{a x}{x}_{\R^{d}} \geq \alpha \norm{x}^{2}_{\R^{d}}\Big)
    \quad\Leftrightarrow\quad \Big( \forall z \in \C^{d}\colon
    \Re \scprod{a z}{z}_{\C^{d}} \geq \alpha \norm{z}^{2}_{\C^{d}}\Big)
  \end{equation*}
\end{lemma}

\begin{proof}
  Clearly the implication ``$\Leftarrow$'' is true. Hence, let $z = x + \iu y \in \C^{d}$, where $x, y \in \R^{d}$. Then
  \begin{align*}
    \Re \scprod{a z}{z}_{\C^{d}}
    &= \Re \bigl( \scprod{ax}{x}_{\C^{d}} + \scprod{ax}{\iu y}_{\C^{d}} + \scprod{a \iu y}{x}_{\C^{d}} + \scprod{a \iu y}{\iu y}_{\C^{d}}\bigr)
    = \scprod{a x}{x}_{\R^{d}} + \scprod{a y}{y}_{\R^{d}} \\
    &\geq \alpha {(\norm{x}_{\R^{d}}^{2} + \norm{y}_{\R^{d}}^{2})} = \alpha \norm{z}_{\C^{d}}^{2}.\tag*{\qedhere}
  \end{align*}
\end{proof}

\begin{theorem}\label{thm:WeakStarLimitPerMultOp}
For $d\in\mathbb{N}$, consider $f\in\Lp{\infty}(\R^d)$ with $f(\argdot+k)=f$ for all $k\in\mathbb{Z}^d$. Furthermore, let $\Omega\subseteq\R^d$ be measurable with non-zero measure.
Then, the sequence of bounded multiplication operators $(f(n\argdot))_{n\in\mathbb{N}}$ in $\Lb(\Lp{2}(\Omega))$ converges to $\int_{(0,1)^d}f(x)\dx$ in the weak operator topology.
\end{theorem}
\begin{proof}
See, e.g., \cite[Theorem~13.2.4]{SeTrWa22}.
\end{proof}

\end{document}